\documentclass[a4paper,11pt]{amsart}

\usepackage[latin1]{inputenc}
\usepackage[T1]{fontenc}
\usepackage[linktocpage]{hyperref} 
\hypersetup{colorlinks=true,linkcolor=blue,citecolor=blue,filecolor=magenta,urlcolor=blue}
\usepackage{geometry}       
\usepackage[english]{babel}  
\usepackage{amsmath,amsfonts,amssymb,amsthm,mathrsfs}
\usepackage{dsfont}
\usepackage{tikz,tkz-graph,tikz-cd}
	\usetikzlibrary[patterns]
\usepackage{array,multirow,makecell}
\usepackage{multicol}

\makeatletter
\def\@tocline#1#2#3#4#5#6#7{\relax
  \ifnum #1>\c@tocdepth 
  \else
    \par \addpenalty\@secpenalty\addvspace{#2}%
    \begingroup \hyphenpenalty\@M
    \@ifempty{#4}{%
      \@tempdima\csname r@tocindent\number#1\endcsname\relax
    }{%
      \@tempdima#4\relax
    }%
    \parindent\z@ \leftskip#3\relax \advance\leftskip\@tempdima\relax
    \rightskip\@pnumwidth plus3em \parfillskip-\@pnumwidth
    #5\leavevmode\hskip-\@tempdima
      \ifcase #1
       \or\or \hskip 1em \or \hskip 2em \else \hskip 3em \fi%
      #6\nobreak\relax
    \dotfill\hbox to\@pnumwidth{\@tocpagenum{#7}}\par
    \nobreak
    \endgroup
  \fi}
\makeatother

\newtheorem{thm}{Theorem}[section]
\newtheorem{propo}[thm]{Proposition}

\newtheorem{cor}[thm]{Corollary}

\theoremstyle{definition}
\newtheorem{de}[thm]{Definition}         
\newtheorem{example}[thm]{Example}

\theoremstyle{remark}
\newtheorem{rmk}[thm]{Remark}

\newcommand{\CC}{\mathds{C}}
\newcommand{\RR}{\mathds{R}}
\newcommand{\NN}{\mathds{N}}
\newcommand{\ZZ}{\mathds{Z}}
\newcommand{\QQ}{\mathds{Q}}

\newcommand{\KK}{\mathds{K}}
\newcommand{\PP}{\mathds{P}}

\newcommand{\A}{\mathcal{A}}
\newcommand{\B}{\mathcal{B}}

\newcommand{\Q}{\mathcal{Q}}
\newcommand{\V}{\mathcal{V}}
\newcommand{\C}{\mathcal{C}}
\newcommand{\D}{\mathcal{D}}
\newcommand{\F}{\mathcal{F}}
\newcommand{\M}{\mathcal{M}}
\newcommand{\N}{\mathcal{N}}
\newcommand{\I}{\mathcal{I}}
\newcommand{\E}{\mathcal{E}}

\newcommand{\K}{\mathcal{K}}

\newcommand{\R}{\mathcal{R}}
\renewcommand{\S}{\mathcal{S}}

\renewcommand{\L}{\mathcal{L}}
\newcommand{\inv}{^{-1}}

\newcommand{\set}[1]{\left\{ #1 \right\}}
\renewcommand{\epsilon}{\varepsilon}
\newcommand{\m}{\mathfrak{m}}

\newcommand\PR{\RR\PP}
\newcommand\PC{\CC\PP}
\newcommand\PK{\KK\PP}

\newcommand{\PCc}{\Check{\PC}}
\newcommand{\PKc}{\Check{\PK}}

\DeclareMathOperator{\HH}{H}

\DeclareMathOperator{\Aut}{Aut}
\DeclareMathOperator{\corank}{corank}
\DeclareMathOperator{\depth}{\overline{depth}}

\DeclareMathOperator{\Sl}{sl}
\DeclareMathOperator{\pl}{pl}
\DeclareMathOperator{\ch}{ch}

\DeclareMathOperator{\PGL}{PGL}

\begin{document}

\title{Configurations of points and topology of real line arrangements}
\author{Beno\^it Guerville-Ball\'e}

\author{Juan Viu-Sos}
\address{
Instituto de Ci\^encias Matem\'aticas e de Computa\c c\~ao,
Universidade de S\~ao Paulo,
Avenida Trabalhador Sancarlense, 400 - Centro,
S\~ao Carlos - SP, 13566-590, Brasil
}
\email{benoit.guerville-balle@math.cnrs.fr, jviusos@math.cnrs.fr}

\thanks{The first author is supported by a JSPS postdoctoral grant. The second author is partially supported by MTM2013-45710-C02-01-P and Grupo Geometr\'ia of Gobierno de Arag\'on/Fondo Social Europeo.}				


\begin{abstract}
	A central question in the study of line arrangements in the complex projective plane $\PC^2$ is the following: when does the combinatorial data of the arrangement determine its topological properties? In the present work, we introduce a topological invariant of complexified real line arrangements, called the \emph{chamber weight}. This invariant is based on the weight counting over the points of the associated dual configuration, located in particular chambers of the real projective plane $\PR^2$.
	
	Using this dual setting, we construct several examples of complexified real line arrangements with the same combinatorial data and different embeddings in $\PC^2$ (i.e. Zariski pairs) which are distinguished by this invariant. In particular, we obtain new Zariski pairs of 13, 15 and 17 lines defined over $\QQ$ and containing only double and triple points. For each one of our examples, we derive some degenerations containing points of multiplicity 2, 3 and 5, which are also Zariski pairs.
	
	We compute explicitly the moduli space of the combinatorics of one of these examples, and prove that it has exactly two connected components. We also obtain three geometric characterizations of these components: the existence of two smooth conics, one tangent to six lines and the other containing six triple points, as well as the collinearity of three specific triple points.
\end{abstract}

\maketitle

\setcounter{tocdepth}{1}
\tableofcontents

\newpage

\section*{Introduction}
A line arrangement $\A$ is a finite collection of distinct lines in $\PC^2$. A classical problem in algebraic geometry related to this object is the study of its topology, defined as the homeomorphism type of the pair $(\PC^2,\A)$. Even if line arrangements are the simplest type of reducible plane curves, their topology is still not well understood. Currently many open questions in this field deal with the influence of the combinatorial type (i.e. the intersection lattice) on the topology. It is easy to prove that topology determines combinatorics, but the converse is shown to be false by Rybnikov~\cite{Rybnikov}. He showed that the fundamental group of the complement $\PC^2\setminus\A$ is not determined by the combinatorics of $\A$. Following Artal in~\cite{Artal:couples}, two line arrangements with the same combinatorics but different topologies are called a \emph{Zariski pair}.

Historically, this problem was first considered in the more general case of algebraic plane curves in $\PC^2$ by Zariski in~\cite{zariski,zariski1,zariski2}, who constructed two irreducible sextics with six ordinary cusps and possessing the same combinatorics, but different topologies. The difference between the embeddings of these sextics is exhibited by a geometrical property: in one sextic, its six singular points lie on a conic, whereas in the other one they do not (see~\cite{ACT:survey} for a survey on Zariski pairs of algebraic curves).\\

In general, it is difficult to determine if two arrangements have different homomorphism types. Finite presentations of the fundamental group of the complement of line arrangements can be obtained by various methods (see~\cite{arvola} and~\cite{Moishezon,zariski,vanKampen}). 
However, determining whether two finite presented groups are isomorphic is in general an \emph{undecidable} problem.

Studying what topological properties are combinatorially determined is becoming an increasing focus of research. For instance, the first Betti numbers and the monodromy of the Milnor fiber, twisted cohomology and characteristic varieties are not known to be combinatorially determined (see~\cite{CS:charcateristic, DP:Hypersurface_complement, Libgober:Monodromy, Yoshinaga:milnor_fibers}). However, the (co)homology ring of the complement $\PC^2\setminus\A$ is known to be of combinatorial nature~\cite{OS:combinatorics}.\\

Although there has been active research on this topic, only three examples of Zariski pairs of line arrangements are known. As stated above, the first known example is due to Rybnikov and appeared in 1994. Rybnikov constructed a pair of arrangements by gluing together first two positive and then one positive and one negative MacLane arrangements. He proved that their fundamental groups are not isomorphic by means of an invariant based on the lower central series. Both arrangements contain only double and triple singular points and equations (of each line) can be defined over $\QQ(\zeta_3)$, where $\zeta_3$ is a primitive third-root of unity, but not over $\QQ$. 

The second example of a Zariski pair is obtained by Artal, Carmona, Cogolludo, and Marco~\cite{ACCM:real_ZP}. It is a pair of arrangements of 11 lines containing points of multiplicity at most five and whose equations can be found to be Galois-conjugated in $\QQ(\sqrt{5})$. Their topological type can be distinguished by an invariant of the pseudo-Coxeter element of the braid monodromy. The fact that they are Galois-conjugated implies that the pair cannot be detected by algebraic methods~\cite[Remark~5.4]{ACCM:real_ZP}. Again, these arrangements cannot have rational equations.

The last example was found by the first author in~\cite{Guerville:ZP}, and it was detected using the linking invariant $\I$ (see~\cite{AFG:invariant,GM:invariant}) which is an adaptation of the linking number of the braid monodromy. These are arrangements of 12 lines containing points of multiplicity at most five. They can be realized with equations which are Galois-conjugated in $\QQ(\zeta_5)$, where $\zeta_5$ is a primitive fifth-root of unit. Moreover, their fundamental groups are non-isomorphic as is later proved by Artal, Cogolludo, Marco and the first author~\cite{ACGM:ZP}. Once again, these arrangements are not rational.

It is worth mentioning that the use of a computer was necessary to detect all of these Zariski pairs due to the increasing complexity of the computations resulting from  the high number of lines and the multiplicity of their singularities. Note also that, until now, no known Zariski pair of line arrangements admits realizations over the rational numbers.\\

An interesting family of line arrangements is formed by the \emph{complexified real arrangements}, i.e. those for which a coordinate system of $\PC^2$ exists such that any line of the arrangement is defined by an $\RR$-linear form. In this case, the real picture of this kind of arrangements determines, not only the fundamental groups~\cite{arvola}, or the abelian covers~\cite{Eriko}, but all the topological information~\cite{Car_Thesis}. Using this, several methods and algorithms are developed to compute topological invariants of complexified real arrangements (see~\cite{Yoshinaga:milnor_fibers}, \cite{Yoshinaga:resonant_bands} and more recently~\cite{BS:real_monodromy}).

It is natural to ask whether or not a certain topological invariant is combinatorially determined for complexified real arrangements (even if it is not determined in the general complex case): for instance, the linking invariant $\I$,  introduced by Artal, Florens and the first author in~\cite{AFG:invariant}. Heuristically, even if complexified real arrangements seem to have no ``linking information'', no formal arguments are known to confirm or refute this hypothesis. It is interesting to note that the only known Zariski pair of complexified real arrangements is that of Artal, Carmona, Cogolludo, and Marco, which cannot be distinguished by the $\I$-invariant or by any other known invariants other than that introduced in~\cite{ACCM:real_ZP}. This illustrates the difficulty of distinguishing such Zariski pairs.\\

In the present work, we consider the question of recognizing Zariski pairs of complexified real arrangements in a more geometric way. Our main tool to deal with this question is the study of the $\I$-invariant from a dual point of view, i.e. with configurations of points in the dual real projective plane. More precisely, we introduce the notion of $(t,m)$-configuration as a set of points (containing $t$ marked points) in $\PR^2$ in such a way that the complexified dual arrangement has optimal properties in order to use the $\I$-invariant. These configurations are provided with a \emph{plumbing} map associating to each point a weight modulo~$m$.

From the previous data, we define the \emph{chamber weight} of the $(3,m)$-configuration by counting the weight of the points located in a given chamber of $\PR^2$ (determined by the $3$~marked points). As stated in Theorem~\ref{thm:invariant} and Corollary~\ref{cor:invariant}, the value of the chamber weight is a topological invariant of the dual line arrangement of the configuration. In addition, it is related to the $\I$-invariant as proven in Theorem~\ref{thm:relation}. It is worth noticing how simple the computation of this invariant is. Since this construction is elementary, it makes sense to ask about the effectiveness of this invariant for complexified real arrangements. We illustrate this by the distinction of several new Zariski pairs.

Indeed, we present three examples of Zariski pairs of complexified real arrangements by constructing different combinatorially-equivalent $(3,2)$-configurations with distinct chamber weights. The first example obtained by this method contains 13~lines defined over $\QQ$ and  only double and triple points. In addition, these arrangements are not rigid, allowing us to construct two distinct degenerations that admit singular points of multiplicity~5. Using the same kind of construction and arguments, we also construct Zariski pairs of complexified real arrangements with~15 and 17~lines. For each one, we derive two or three degenerations which are also Zariski pairs containing points of multiplicity~5. Summarizing, we present 10~Zariski pairs of line arrangements with rational equations and which are topologically distinguished using this geometrical invariant.

Moreover, we explicitly compute their \emph{moduli space} (i.e. the space of all arrangements having a given combinatorics) of the degeneration with two quintuple points of the Zariski pair with 13~lines. This allows us to prove that this moduli space has exactly two connected components, each one corresponding to an arrangement of the Zariski pair. From this, we obtain three Zariski-like geometric characterizations of these connected components:
\begin{itemize}
	\item \emph{Tangency to a conic}: In the first connected component of the moduli space, six lines of the arrangement are tangent to a smooth conic, while they are not in the second one.
	\item \emph{Singularities in a conic}: In addition, six triple points of the arrangement are contained in a smooth conic in the first component, but this is not true in the second one.
	\item \emph{Alignment of singularities}: Three triple points are aligned in any arrangement of the first one, but this does not hold in the second connected component.
\end{itemize}
To our knowledge, this is the first time that such a geometric characterization is revealed for Zariski pairs of line arrangements. Moreover, they are the first examples of Zariski pairs of rational line arrangements.\footnote{An appendix containing detailed figures of our first Zariski pair with 13 lines as well as those of its degeneration with two quintuple points can be found in the authors' websites.}


In an upcoming paper~\cite{GBVS:real_pi1}, the authors prove that some of the Zariski pairs presented in this paper can be distinguished by their fundamental groups. More precisely, the corresponding lower central factors differ by a torsion element. These computations give a negative answer to \textsc{Question}~8.7 of~\cite{Suciu}. Unfortunately, the cohomology with twisted coefficients and the Betti numbers or the monodromy of the Milnor fiber do not differ in these pairs.\\

In our examples of Zariski pairs, the simplicity of computation of the chamber weight and the possibility of constructing other examples using the present method are a significant progress for the fundamental questions relating combinatorial data and topology of line arrangements. As an illustration of these possibilities, at the end of the paper we give a sketch of an alternative proof of Artal's result~\cite{artal:char_var} stating that the characteristic varieties are not determined by the weak combinatorics. We prove this result by using simultaneously the notion of chamber weight and Artal's construction~\cite{artal:position}.

The paper is organized as follows:\\
In Section~\ref{sec:configuration}, we define the notion of $(t,m)$-configuration and we introduce our topological invariant: the chamber weight of a $(3,m)$-configuration. Our main invariance theorem is also stated together with a useful corollary. Section~\ref{sec:ZP} recalls aspects of the topology of line arrangements, but it is mainly devoted to the study of the first example of Zariski pair with 13 lines.  The degenerations of this Zariski pair are given in Section~\ref{sec:deg}. We show in Section~\ref{sec:MS} that the moduli space of one of the previous Zariski pairs has two connected components which can be distinguished by geometrical properties. A list of other examples of Zariski pairs and degenerations are described in Section~\ref{sec:list}. After reviewing  the $\I$-invariant, all the proofs related to our main results are presented in Section~\ref{sec:proof}. \\
%
%
In the appendix, we provide detailed pictures of the Zariski pair with 13 lines possessing at most triple points as well as those of its degeneration with two quintuple points.\\

\noindent\textbf{Acknowledgments.} 
The main part of this work was carried out during the second author's visit to Japan. He would like to thank Tokyo Gakugei University, as well as the first author and his wife for their hospitality and the grant {\sc MTM2013-45710-C02-01-P} for the travel support. Both authors would like to thank Hokkaido University and the organizers of the \emph{Summer Conference on Hyperplane Arrangements in Sapporo 2016}, as important advances of the present work were made during this week. In particular, we are grateful to Prof.~Yoshinaga and Prof.~Falk for rewarding discussions. We also would like to thank Prof.~Artal and Prof.~Cogolludo for the subsequent fruitful comments, in particular those about the properties of these new Zariski pairs.

\bigskip
\section{Configurations of points and dual arrangements}\label{sec:configuration}
In the present section, we define our main object: the $(t,m)$-configurations. We draw attention to the fact that the definition of configuration considered here is slightly different from the one of Gr\"unbaum in~\cite{Grunbaum}. After recalling some definitions about the topology of line arrangements (see also~\cite{OrlikTerao92,Dimca:book}), we define the \emph{chamber weight} of a configuration and then we state in Theorem~\ref{thm:invariant} that this is a topological invariant of the associated dual arrangement.

\subsection{Configurations}\mbox{}

For $P,Q$ points in $\PR^2$, denote by $(P,Q)$ the line passing through $P$ and $Q$.

\begin{de}\label{def:configuration}
	A \emph{$(t,m)$-configuration} $\C$ is the data $(\V,\S,\L,\pl)$ composed of two finite sets of points $\V=\set{V_1,\dots,V_t}$ and $\S=\set{S_1,\dots,S_n}$ of $\PR^2$; a finite set of lines $\L=\set{(S,V)\mid S\in\S, V\in\V}$ in $\PR^2$ and a map $\pl:\V\sqcup\S\rightarrow \ZZ/m\ZZ$ with $\V=\pl^{-1}(0)$, such that:
	\begin{enumerate}
		\item for any $V_i,V_j\in\V$: $\S\cap (V_i,V_j)=\emptyset$,
		\item for any line $L\in\L$: $\sum\limits_{S\in L\cap\S} \pl(S) = 0$.
	\end{enumerate}
	The points in $\V$ (resp. in $\S$) are called the \emph{vertices} (resp. \emph{surrounding-points}) of $\C$. The map $\pl$ is called a \emph{$m$-plumbing} of $\V\sqcup\S$.
\end{de}

\pagebreak
\begin{rmk}\mbox{}
	\begin{enumerate}
		\item The first condition implies that $\V\cap\S=\emptyset$.
		\item Any line $L\in\L$ contains exactly one vertex and the multiplicity of $\L$ in any surrounding-point is exactly $t$. See Example~\ref{ex:configuration}.
	\end{enumerate}
\end{rmk}

\begin{de}
    A \emph{$(t,m)$-configuration} $\C=(\V,\S,\L,\pl)$ is called \emph{planar} if the projective subspace generated by the vertices $\V$ is the whole $\PR^2$.
\end{de}

A $(t,m)$-configuration is \emph{uniform} if its plumbing map $\pl$ is constant on $\S$, i.e. there exists an element $\zeta\in\ZZ/m\ZZ$ such that $\pl(S)=\zeta$, for any $S\in\S$.

\begin{rmk}
	Any $(3,2)$-configuration is necessarily uniform.
\end{rmk}

\begin{example}\label{ex:configuration}
	Examples of $(3,2)$, $(3,m)$ and $(4,2)$-configurations are given in Figure~\ref{fig:config_examples}. Remark that the dashed lines in the figures are not elements of $\L$, but they take an important role in our setting as we show in Section~\ref{subsec:thm}.
\end{example}
	
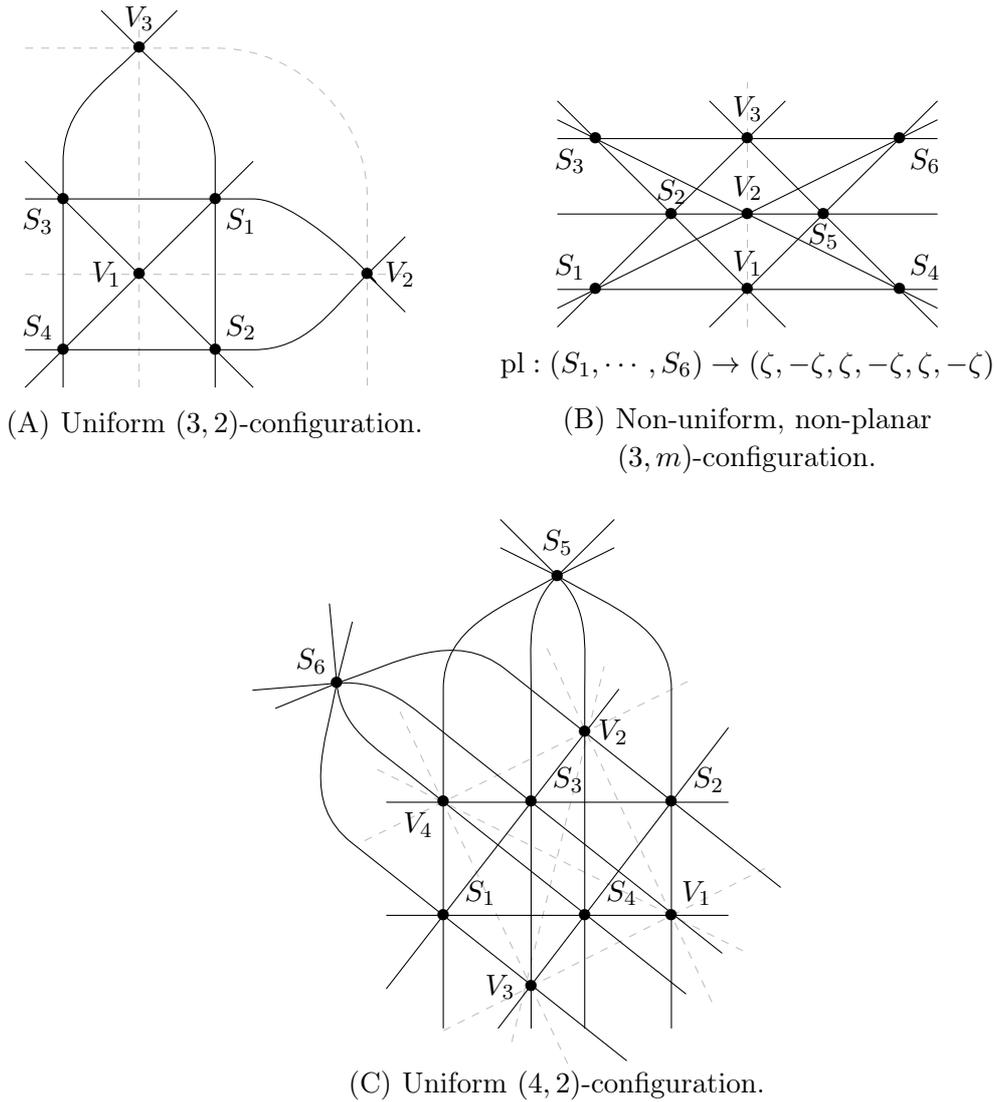
\begin{figure}
	\begin{tikzpicture}
		\begin{scope}[shift={(-4,0)}]
	\draw[dashed, color=gray!50] (-1.5,0) -- (3.5,0);
	\draw[dashed, color=gray!50] (0,-1.5) -- (0,3.5);
	\draw[dashed, color=gray!50] (-1.5,3) -- (1,3) to[out=0,in=90] (3,1) -- (3,-1.5);
	
	\draw (-1,-1.5) -- (-1,1.5) to[out=90,in=-135] (0,3) -- (0.5,3.5);
	\draw (1,-1.5) -- (1,1.5) to[out=90,in=-45] (0,3) -- (-0.5,3.5);
	\draw (-1.5,1) -- (1.5,1) to[out=0,in=-45] (3,0) -- (3.5,-0.5);
	\draw (-1.5,-1) -- (1.5,-1) to[out=0,in=-135] (3,0) -- (3.5,0.5);
	\draw (-1.5,1.5) -- (1.5,-1.5);
	\draw (-1.5,-1.5) -- (1.5,1.5); 

	\node at (0,0) {$\bullet$};
	\node[left] at (-0.1,0) {$V_1$};
	\node at (3,0) {$\bullet$};
	\node[right] at (3.1,0) {$V_2$};
	\node at (0,3) {$\bullet$};
	\node[above] at (0,3.1) {$V_3$};
	
	\node at (1,1) {$\bullet$};
	\node[below right] at (1,1) {$S_1$};
	\node at (1,-1) {$\bullet$};
	\node[above right] at (1,-1) {$S_2$};
	\node at (-1,1) {$\bullet$};
	\node[below left] at (-1,1) {$S_3$};
	\node at (-1,-1) {$\bullet$};
	\node[above left] at (-1,-1) {$S_4$};
	
	\node at (1,-2) {(A) Uniform $(3,2)$-configuration.};
		\end{scope}
		\begin{scope}[shift={(4,-0.2)}, rotate=90]
		\draw[dashed, color=gray!50] (-0.5,0) -- (2.75,0);

		\draw (0,-2.5) -- (0,2.5);
		\draw (1,-2.5) -- (1,2.5);
		\draw (2,-2.5) -- (2,2.5);
		\draw (-0.25,-2.5) -- (2.25,2.5);
		\draw (-0.25,2.5) -- (2.25,-2.5);
		\draw (-0.5,-0.5) -- (2.5,2.5);
		\draw (2.5,0.5) -- (-0.5,-2.5);
		\draw (-0.5,2.5) -- (2.5,-0.5);
		\draw (-0.5,0.5) -- (2.5,-2.5);
	
		\node at (0,0) {$\bullet$};
		\node[above] at (0.05,0) {$V_1$};
		\node at (1,0) {$\bullet$};
		\node[above] at (1.05,0) {$V_2$};
		\node at (2,0) {$\bullet$};
		\node[above] at (2.1,0) {$V_3$};
		
		\node at (0,2) {$\bullet$};
		\node[above left] at (0,2) {$S_1$};
		\node at (1,1) {$\bullet$};
		\node[above] at (1,1) {$S_2$};
		\node at (2,2) {$\bullet$};
		\node[below left] at (2,2) {$S_3$};
		
		\node at (0,-2) {$\bullet$};
		\node[above right] at (0,-2) {$S_4$};
		\node at (1,-1) {$\bullet$};
		\node[below] at (1,-1) {$S_5$};
		\node at (2,-2) {$\bullet$};
		\node[below right] at (2,-2) {$S_6$};
		
		\node at (-1,0) {$\pl:(S_1,\cdots,S_6)\rightarrow (\zeta,-\zeta,\zeta,-\zeta,\zeta,-\zeta)$};		
		
		\node at (-1.75,0) {(B) Non-uniform, non-planar};
		\node at (-2.25,0) {$(3,m)$-configuration.};
		\end{scope}
		\begin{scope}[shift={(0,-8.5)},scale=1.5]
		\draw[dashed, color=gray!50] (-0.69,0.66) -- (2.19,2.09);
		\draw[dashed, color=gray!50] (-0.38,1.8) -- (1.09,-1.31);
		\draw[dashed, color=gray!50] (0.91,2.36) -- (2.37,-0.8);
		\draw[dashed, color=gray!50] (0,-1.02) -- (2.87,0.44);
		\draw[dashed, color=gray!50] (-0.58, 1.29) -- (2.68, -0.34);
		\draw[dashed, color=gray!50] (0.6, -1.12) -- (1.41, 2.2);
		\draw (0,-1) -- (0,2) to[out=90,in=-150] (1,3) -- (1.5,3.25);
		\draw (0.77,-1) -- (0.77,2) to[out=90,in=-135] (1,3) -- (1.5,3.5);
		\draw (1.24,-1) -- (1.24,2) to[out=90,in=-45] (1,3) -- (0.5,3.5);
		\draw (2,-1) -- (2,2) to[out=90,in=-30] (1,3) -- (0.5,3.25);
		\draw (-0.5,0) -- (2.5,0);
		\draw (-0.5,1) -- (2.5,1);	
		\draw (-0.5,-0.65) -- (1.54,2);
		\draw (0.48,-1) -- (2.5,1.66);
		\begin{scope}[xscale=1.05,rotate=50,shift={(0,-1)}]
			\draw (0,-1) -- (0,2) to[out=90,in=-150] (1,3) -- (1.5,3.25);
			\draw (0.77,-1) -- (0.77,2) to[out=90,in=-135] (1,3) -- (1.5,3.5);
			\draw (1.24,-1) -- (1.24,2) to[out=90,in=-45] (1,3) -- (0.5,3.5);
			\draw (2,-1) -- (2,2) to[out=90,in=-30] (1,3) -- (0.5,3.25);
			\node at (1,3) {$\bullet$};
			\node[above left] at (1,3) {$S_6$};
		\end{scope}
	
		\node at (2,0) {$\bullet$};
		\node[above right] at (2,0) {$V_1$};
		\node at (1.24,1.62) {$\bullet$};
		\node[right] at (1.27,1.62) {$V_2$};
		\node at (0.77,-0.63) {$\bullet$};
		\node[left] at (0.72,-0.63) {$V_3$};
		\node at (0,1) {$\bullet$};
		\node[below left] at (0,1) {$V_4$};
		\node at (0,0) {$\bullet$};
		\node[above right] at (0.1,0) {$S_1$};
		\node at (2,1) {$\bullet$};
		\node[above right] at (2.1,1) {$S_2$};
		\node at (0.77,1) {$\bullet$};
		\node[above right] at (0.87,1) {$S_3$};
		\node at (1.24,0) {$\bullet$};
		\node[above right] at (1.34,0) {$S_4$};
		\node at (1,3) {$\bullet$};
		\node[above] at (1,3.1) {$S_5$};

		\node at (1,-1.5) {(C) Uniform $(4,2)$-configuration.};
		\end{scope}
	\end{tikzpicture}
	\caption{Examples of $(t,m)$-configurations\label{fig:config_examples}}
\end{figure}

We encode the combinatorial information given by a $(t,m)$-configuration using non-trivial collinearity between points.

\begin{de}
	The \emph{combinatorics} of a $(t,m)$-configuration $\C=(\V,\S,\L,\pl)$ is the collection of all triplets of collinear points in $\V\sqcup\S$. 
\end{de}

In order to simplify the notation of a combinatorics, if $k\geq 4$ different points $P_1,\dots,P_k$ in $\V\sqcup\S$ are collinear, we write the set $\set{P_1,\dots,P_k}$ instead of all the triplets contained in $\set{P_1,\dots,P_k}$. A set formed by exactly $k\in\NN$ collinear points of a configuration is called a \emph{$k$-point line}.

We say that two $(t,m)$-configurations $\C_1=(\V_1,\S_1,\L_1,\pl_1)$ and $\C_2=(\V_2,\S_2,\L_2,\pl_2)$ \emph{have the same combinatorics} if there exists a bijection between $\V_1\sqcup\S_1$ and $\V_2\sqcup\S_2$ respecting collinearity relations.

\begin{example}
	The combinatorics of the $(3,2)$-configuration of Figure~\ref{fig:config_examples}-(A) is given by
	\[
		\big\{\ \{V_1,S_1,S_4\},\ \{V_1,S_2,S_3\},\ \{V_2,S_1,S_3\},\ \{V_2,S_2,S_4\},\ \{V_3,S_1,S_2\},\ \{V_3,S_3,S_4\}\ \big\}.
	\]
\end{example}

\begin{rmk}
		The combinatorics of a $(t,m)$-configuration is not invariant by isotopy. It is possible to create an extra alignment of points during a deformation (see Section~\ref{sec:deg}).
\end{rmk}

\begin{de}
	Let $\C=(\V,\S,\L,\pl)$ be a $(t,m)$-configuration and $\K$ the associated combinatorics of $\C$.
	\begin{itemize}
		\item An \emph{automorphism of $\K$} is a bijection $\phi$ of $\V\sqcup\S$ which preserves $\K$. The group of such bijections is called the \emph{automorphism group of $\K$} and is denoted by $\Aut(\K)$.
		
		\item An automorphism of $\K$ is \emph{stabilizing} if $\phi(\V)=\V$ (equivalently, if $\phi(\S)=\S$). The subgroup of stabilizing automorphisms is denoted by $\Aut^\text{Stab}(\K)$.
		
		\item The configuration $\C$ is \emph{stable} if  $\Aut(\K)=\Aut^\text{Stab}(\K)$.
	\end{itemize}
\end{de}

\begin{rmk}
	The $(3,2)$-configuration of Example~\ref{ex:configuration} is stable, but the Pappus configuration (see Figure~\ref{fig:Pappus_nonPappus}--(A)) is not.
\end{rmk}

\subsection{Dual arrangements}\mbox{}

A \emph{complex line arrangement} in $\PC^2$ is a finite set of distinct projective lines $\A=\set{D_1,\dots,D_s}$. We say that $\A$ is a \emph{complexified real arrangement} if there exists a coordinate system of $\PC^2$ for which each line of $\A$ admits an equation with real coefficients.

\begin{de}
	The \emph{topology} of a line arrangement $\A$ in $\PC^2$ is the homeomorphism type of the pair $(\PC^2,\A)$. This topology is said to be \emph{ordered} if the homeomorphism respects a fixed order on $\A$. 
\end{de}

The abelian group $\HH_1(\PC^2\setminus\A;\ZZ)$ is generated by the meridians $\m_1,\dots,\m_s$ around the lines $D_1,\dots,D_s$, with (only) relation $\sum_{i=1}^s \m_i =0$. A \emph{character} $\xi$ of an arrangement $\A$ is a morphism from $\HH_1(\PC^2\setminus\A;\ZZ)$ to $\CC^*$, and it is defined by assigning a non-zero complex number to each meridian (or equivalently to each line) such that their sum is~$1$. The character $\xi$ is \emph{torsion} if for any $D\in\A$, there exists $r\in\NN^*$ such that $\xi(\m_D)^r=1$.\\

For an arbitrary field $\KK$, consider $\PKc^2=\{L \mid L\subset \PK^2\ \text{line}\}$ the \emph{dual projective space}, which is naturally isomorphic to the set of $\KK$-linear forms in $\KK^3$ modulo non-zero scalars. There is a natural correspondence between $\PK^2$ and $\PKc^2$ given by the duality $\KK^3\simeq (\KK^3)^*$ and respecting incidences, i.e. a point $P$ lies in a line $L$ if and only if $L^*\in P^*$. Note that, for any point $P$ and any line $L$, we have $(P^*)^*=P$ and $(L^*)^*=L$.

\begin{de}
	Let $\C=(\V,\S,\L,\pl)$ be a $(t,m)$-configuration. The \emph{dual plumbed arrangement (DPA)} associated to $\C$ is the triplet $(\A^\V,\A^\S,\xi)$, where:
	\begin{itemize}
		\item $\A^\V=\set{V_1^*\otimes\CC,\dots,V_t^*\otimes\CC}$ and $\A^\S=\set{S_1^*\otimes\CC,\dots,S_n^*\otimes\CC}$ are line arrangements in $\PC^2$,
		
		\item $\xi$ is a torsion character of $\A^\C=\A^\V\cup\A^\S$ assigning $1$ to any line of $\A^\V$ and $\exp\big( 2\pi i \pl(S)/m \big)$ to $S^*\otimes\CC\in\A^\S$.
	\end{itemize}
	The line arrangement $\A^\C$ is called the \emph{dual arrangement} of $\C$, and $\A^\V$ the \emph{support} of the dual plumbed arrangement. In order to simplify the notation, a line $V_i^*\otimes\CC$ (resp. $S_i^*\otimes\CC$) is denoted by $V_i^\bullet$ (resp. $S_i^\bullet$).
\end{de}

\begin{rmk}\mbox{}
	\begin{enumerate}
		\item The above notion of ``support'' is related to the support defined in~\cite{AFG:invariant}.
		\item By construction, any line arrangement obtained as a DPA of a $(t,m)$-configuration is a complexified real arrangement.
	\end{enumerate}
\end{rmk}

Note that the reverse of the operation described above is not well-defined in general, i.e. if we consider any couple of an arrangement and a character, then the dual configuration of points and the associated plumbing may not satisfy the last condition of Definition~\ref{def:configuration}. Nevertheless, this operation is well-defined for a particular type of arrangements described in Section~\ref{sec:proof}.

\subsection{Topological invariance of the chamber weight of configurations}\label{subsec:thm}\mbox{}

Our aim is to introduce a topological invariant of arrangements coming from $(t,m)$-configurations: the \emph{chamber weight} of a configuration. %
This number is computed from a configuration $\C$ and associated to the dual arrangement $\A^\C$ as a dual version of the $\I$-invariant. Thus, it is worth noticing that even if the chamber weight is only defined for configurations and thus associated to real complexified arrangements, it is the computed value of a topological invariant which is defined for arbitrary complex line arrangements~\cite{AFG:invariant}. %
All the proofs of the results of this section are based on the interpretation of this number as the $\I$-invariant and are given in Section~\ref{sec:proof}.

Let $\C=(\V,\S,\L,\pl)$ be a planar $(3,m)$-configuration. The lines $(V_1,V_2)$, $(V_2,V_3)$ and $(V_3,V_1)$ divide $\PR^2$ in 4 chambers, noted $\ch_1,\dots,\ch_4$. For any chamber $\ch_i$, define the value
\[
    \tau_i(\C) = \sum\limits_{S \in \S \cap \ch_i} \pl(S)\in\ZZ/m\ZZ.
\]

\begin{propo}\label{propo:tau_constant}
	The value $\tau_i(\C)$ is the same for all $i=1,\ldots,4$, i.e. it does not depend on the choice of the chamber $\ch_i$. Moreover, it takes values over $\set{[0],\left[\frac{m}{2}\right]}\subset\ZZ/m\ZZ$ if $m$ is even, and it is zero if $m$ is odd.
	
\end{propo}

\begin{proof}
    For any two different chambers $\ch_i$ and $\ch_j$, take $V\in\V$ such that $\ch_i\cup\ch_j$ is contained in the cone with vertex $V$ and delimited by the two lines joining $V$ with the other vertices in $\V$. Denote by $\L_V$ the subset of lines in $\L$ passing through $V$, and take $\S_i=\S\cap\ch_i$ and $\S_j=\S\cap\ch_j$, respectively. Note that, by construction, any point in $\S_i$ or $\S_j$ is contained in an unique line $L\in\L_V$. By the second condition in Definition~\ref{def:configuration}, we have
    \[
	    0 =\sum_{L\in\L_V}\sum_{S\in L\cap\S}\pl(S)=\sum_{L\in\L_V}\sum_{S\in L\cap\S_i}\pl(S) + \sum_{L\in\L_V}\sum_{S\in L\cap\S_j}\pl(S)=\sum_{S\in \S_i}\pl(S) + \sum_{S\in \S_j}\pl(S)
    \]
    Thus, $\tau_i(\C)+\tau_j(\C)=0$ in $\ZZ/m\ZZ$ for any $i\neq j$. Considering the previous equations, we deduce that $\tau_i(\C)=\tau_j(\C)$. In particular, this implies that $\tau_i(\C)=-\tau_i(\C)$. Then, $\tau_i(\C)\in\set{[0],\left[\frac{m}{2}\right]}\subset\ZZ/m\ZZ$ if $m$ is even, and $\tau_i(\C)$ is zero if $m$ is odd.
\end{proof}

An alternative, less combinatorial proof of the previous proposition can be found in Section~\ref{sec:proof}.

\begin{de}
	The \emph{chamber weight} of $\C$ is
	\begin{equation*}
	\tau(\C) = \sum\limits_{S \in \S \cap \ch_i} \pl(S)\in\ZZ/m\ZZ.
	\end{equation*}
\end{de}

\begin{rmk}\label{rmk:parity}
	In the case of a $(3,2)$-configuration, $\tau(\C)$ is the parity of the number of points of $\S$ contained in any chamber $\ch_i$.
\end{rmk}

\begin{example}
	In Figure~\ref{fig:config_examples}--(A), we have one point $S_i$ in each chamber $\ch_i$. Since the configuration is uniform and $\pl(\S)\subset\ZZ/2\ZZ\setminus\{0\}$, then the chamber weight is
	\begin{equation*}
		\tau(\C)=\pl(S_i)=1.
	\end{equation*}
\end{example}

\begin{thm}\label{thm:invariant}
	Let $\C$ and $\C'$ be two planar $(3,m)$-configurations. Assume that there exists a homeomorphism of $\PC^2$ sending $\A^\C$ to $\A^{\C'}$ and respecting a fixed order on these arrangements. Then
	\begin{equation*}
		\tau(\C) = \tau(\C').
	\end{equation*}
\end{thm}

The result above can be improved with the removal of the ordered condition by assuming that no automorphism of the combinatorics modifies the chamber weight. We could assume that there is no automorphism of the combinatorics other than the identity, but it is sufficient to assume a weaker hypothesis given in the following corollary.

\begin{cor}\label{cor:invariant}
	Let $\C$ be a stable planar uniform $(3,m)$-configuration. The chamber weight $\tau(\C)$ is an invariant of the topology of $\A^\C$.
\end{cor}

\bigskip
\section{Topology of line arrangements: Zariski pairs}\label{sec:ZP}

The chamber weight of a configuration provides a new way to study a classical problem for line arrangements: how is the combinatorial data of an arrangement related with its topological properties? After recalling basic notions about combinatorics of arrangements and its relation with those of configurations, we exhibit an example of a Zariski pair with 13 lines using the chamber weight. We conclude by listing some additional properties of these arrangements.

\subsection{Combinatorics and topology of arrangements}\label{subsec:arrangement}\mbox{}

Analogously to the $(t,m)$-configurations case, we can encode all the combinatorial data of line arrangements by using incidence properties between lines.
\begin{de}
	The \emph{combinatorics} of a line arrangement $\A$ is the collection of all triplets of concurrent lines at the same point in $\A$.
\end{de}
In order to simplify the above notation, if $k\geq 4$ lines $L_1,\dots,L_k$ are concurrent at a point, we write the set $\set{L_1,\dots,L_k}$ instead of all the triplets contained in $\set{L_1,\dots,L_k}$. We say that two line arrangements $\A_1$ and $\A_2$ \emph{have the same combinatorics} if there exists a bijection between $\A_1$ and $\A_2$ respecting incidence relations.

The action of dualizing transforms collinearity relations into incidence relations and vice-versa. As a direct consequence of the construction of dual arrangements, we have the following result.
  
\begin{propo}\label{propo:equivalence}
	A $(t,m)$-configuration $\C$ and its dual arrangement $\A^\C$ have the same combinatorics, i.e. the map sending $P\in\V\sqcup\S$ to $P^\bullet$ maps a triple of collinear points to a triple of concurrent lines.
\end{propo}

\begin{rmk}\label{rmk:weak_combinatorics}
Classically, a \emph{weaker} notion of combinatorics for line arrangements is also considered (see, for example~\cite{artal:char_var}), defined by the number of lines and the number of singular points of each multiplicity (globally and on each line).
\end{rmk}

We are interested in the study of the topology/combinatorics interaction on line arrangements. We introduce a classical object in the study of the topology of line arrangements: a \emph{Zariski pair} is a couple of arrangements having isomorphic combinatorics and non-homeomorphic topological types. In order to study this interaction, and thus to produce Zariski pairs, we use $(t,m)$-configurations. In particular, we focus on the construction of planar $(3,2)$-configurations, in order to use either Theorem~\ref{thm:invariant} or Corollary~\ref{cor:invariant}. Recall that any $(3,2)$-configuration $\C$ admits uniquely a constant plumbing $\pl(S)=1\in\ZZ/2\ZZ\setminus\{0\}$, for any $S\in\S$, and in that case the chamber weight is simply the parity of points contained in a chamber.

\subsection{A Zariski pair with 13 lines}\label{subsec:ZP13}\mbox{}

We construct four $(3,2)$-configurations $\C_{1,1}, \C_{1,-1}, \C_{-1,1}$ and $\C_{-1,-1}$ with the same combinatorics and verifying that $\tau(C_{\alpha,\beta})\neq\tau(C_{\alpha',\beta'})$ if $\alpha\beta\neq\alpha'\beta'$. Using Theorem~\ref{thm:invariant} and Corollary~\ref{cor:invariant}, we conclude that the associated plumbed dual arrangements have different topological types.

For any $\alpha,\beta\in\{-1,1\}$, let $\C_{\alpha,\beta}=(\V,\S_{\alpha,\beta},\L_{\alpha,\beta},\pl)$ be four uniform (3,2)-configurations defined by the following data
\[
	\V=\{V_1,V_2,V_3\},\quad \S_{\alpha,\beta}=\S\sqcup \S_{\alpha}\sqcup\S_{\beta},
\]
\[
	 \S=\{S_1,\ldots,S_4\},\quad \S_{\alpha}=\{S_5^\alpha,S_6^\alpha,S_7^\alpha\},\quad \S_{\beta}=\{S_8^\beta,S_9^\beta,S_{10}^\beta\},
\]
where
\[
	V_1=(0:1:0),\quad V_2=(1:0:0),\quad V_3=(0:0:1),
\]
\[
	S_1=(1:1:1),\quad S_2=(4:4:1),\quad S_3=(-1:8:2),\quad S_4=(8:-1:2),
\]
\[
	S_5^\alpha=(-1:8:4\alpha),\quad S_6^\alpha=(-1:4\alpha :2),\quad S_7^\alpha=(-\alpha:4:4),
\]
\[
	S_8^\beta=(8:-1:4\beta),\quad S_9^\beta=(4\beta:-1:2),\quad S_{10}^\beta=(4:-\beta:4).
\]

Note that any line in the above four $(3,2)$-configurations contains exactly two surrounding-points, which is compatible with the constant $2$-plumbings
\[
	\pl(S)=1,\quad \forall S\in\S_{\alpha,\beta},
\]
for any $\alpha,\beta\in\{-1,1\}$. These four $(3,2)$-configurations are plotted\footnote{In order to have clearer pictures, they are not plotted to scale but up to deformation respecting the combinatorics.} in Figure~\ref{fig:dual_ZP}, and detailed in Appendix~\ref{sec:appendix_ZP13}.

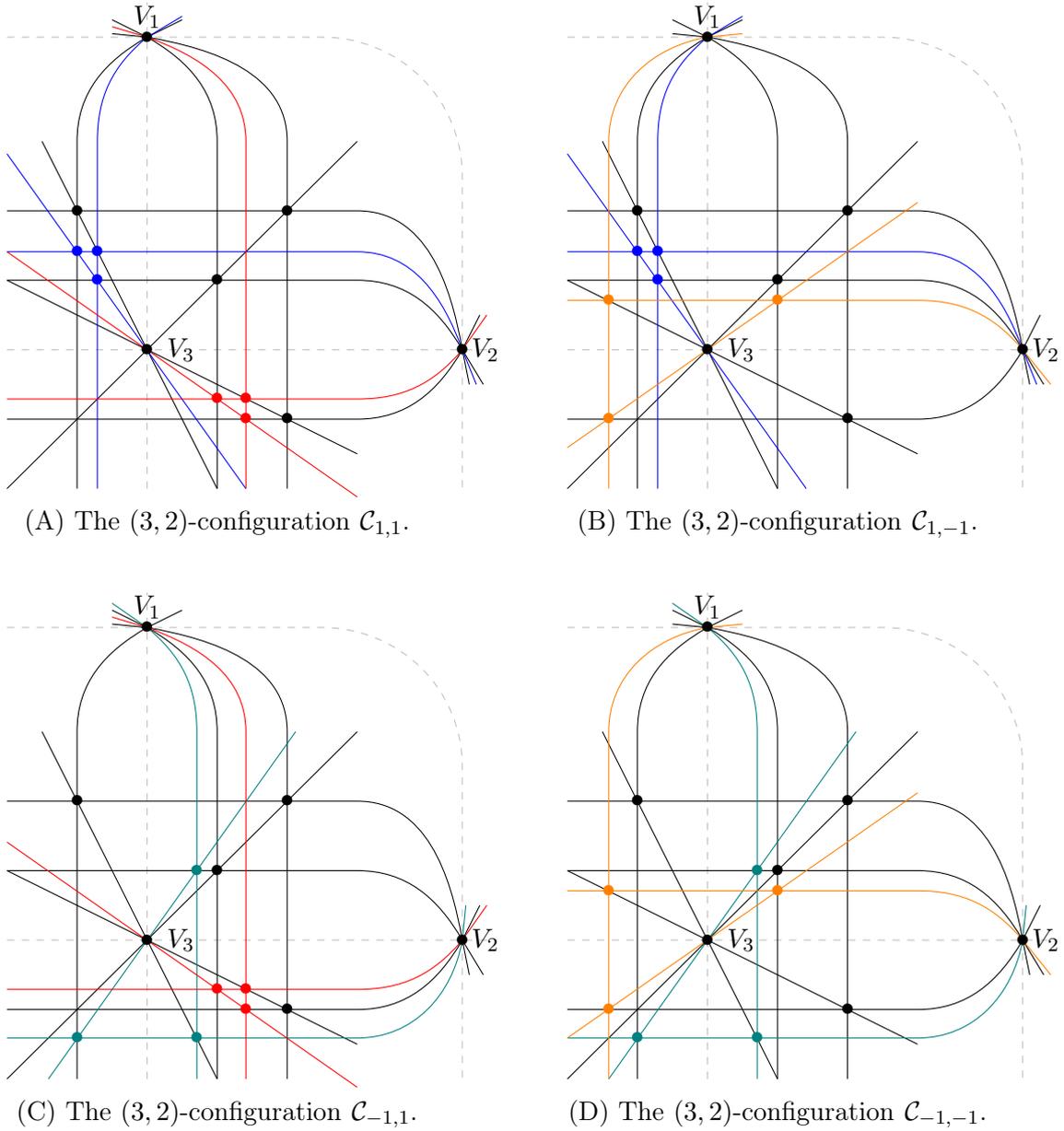
\begin{figure}[h]
	\begin{tikzpicture}

	\def\colorap{blue} 
	\def\coloram{teal}  
	\def\colorbp{red} 
	\def\colorbm{orange} 

	
	\begin{scope}[shift={(-4,4.25)}]
	
			\node at (1,-2.5) {(A) The $(3,2)$-configuration $\C_{1,1}$.};

			\draw[dashed, color=gray!50] (0,-2) -- (0,5) ;
			\draw[dashed, color=gray!50] (-2,0) -- (5,0) ;
			\draw[dashed, color=gray!50] (-2,4.5) -- (2.5,4.5) to[out=0,in=90] (4.5,2.5) -- (4.5,-2);

			\draw (-2,-1) -- (3,-1) to[out=0,in=-120] (4.5,0) -- (4.75,0.5);
			\draw (-2,2) -- (3,2) to[out=0,in=100] (4.5,0) -- (4.6,-0.5);
			\draw (-2,1) -- (3,1) to[out=0,in=120] (4.5,0) -- (4.8,-0.5);
			\draw (-1,-2) -- (-1,3) to[out=90,in=-150] (0,4.5) -- (0.5,4.75);
			\draw (1,-2) -- (1,3) to[out=90,in=-30] (0,4.5) -- (-0.5,4.75);
			\draw (2,-2) -- (2,3) to[out=90,in=-10] (0,4.5) -- (-0.5,4.55);
			\draw (-2,-2) -- (3,3);
			\draw (-2,1) -- (3,-1.5);
			\draw (-1.5,3) -- (1,-2);

			\draw[color=\colorap] (-2,1.41) -- (3,1.41) to[out=0,in=110] (4.5,0) -- (4.7,-0.5);
			\draw[color=\colorap] (-0.71,-2) -- (-0.71,3) to[out=90,in=-140] (0,4.5) -- (0.5,4.8);
			\draw[color=\colorap] (-2,2.82) -- (1.41,-2);


			\draw[color=\colorbp] (-2,-0.71) -- (3,-0.71) to[out=0,in=-130] (4.5,0) -- (4.85,0.5);
			\draw[color=\colorbp] (1.41,-2) -- (1.41,3) to[out=90,in=-20] (0,4.5) -- (-0.5,4.65);
			\draw[color=\colorbp] (-2,1.41) -- (3,-2.12);



			 \node at (0,0) {$\bullet$};
			 \node[right] at (0.15,0) {$V_3$};
			 \node at (0,4.5) {$\bullet$};
			 \node[above] at (0,4.5) {$V_1$};
			 \node at (4.5,0) {$\bullet$};
			 \node[right] at (4.5,0) {$V_2$};
		 
			 \node at (1,1) {$\bullet$}; 
			 \node at (2,2) {$\bullet$}; 
			 \node at (-1,2) {$\bullet$}; 
			 \node at (2,-1) {$\bullet$}; 
		 
			 \node[text=\colorap] at (-0.71,1.41) {$\bullet$}; 
			 \node[text=\colorap] at (-1,1.41) {$\bullet$}; 
			 \node[text=\colorap] at (-0.71,1) {$\bullet$};  
		 
			 \node[text=\colorbp] at (1.41,-0.71) {$\bullet$}; 
			 \node[text=\colorbp] at (1.41,-1) {$\bullet$}; 
			 \node[text=\colorbp] at (1,-0.71) {$\bullet$}; 
		\end{scope}

	
	\begin{scope}[shift={(4,4.25)}]
	
			\node at (1,-2.5) {(B) The $(3,2)$-configuration $\C_{1,-1}$.};

			\draw[dashed, color=gray!50] (0,-2) -- (0,5) ;
			\draw[dashed, color=gray!50] (-2,0) -- (5,0) ;
			\draw[dashed, color=gray!50] (-2,4.5) -- (2.5,4.5) to[out=0,in=90] (4.5,2.5) -- (4.5,-2);

			\draw (-2,-1) -- (3,-1) to[out=0,in=-120] (4.5,0) -- (4.75,0.5);
			\draw (-2,2) -- (3,2) to[out=0,in=100] (4.5,0) -- (4.6,-0.5);
			\draw (-2,1) -- (3,1) to[out=0,in=120] (4.5,0) -- (4.8,-0.5);
			\draw (-1,-2) -- (-1,3) to[out=90,in=-150] (0,4.5) -- (0.5,4.75);
			\draw (1,-2) -- (1,3) to[out=90,in=-30] (0,4.5) -- (-0.5,4.75);
			\draw (2,-2) -- (2,3) to[out=90,in=-10] (0,4.5) -- (-0.5,4.55);
			\draw (-2,-2) -- (3,3);
			\draw (-2,1) -- (3,-1.5);
			\draw (-1.5,3) -- (1,-2);

			\draw[color=\colorap] (-2,1.41) -- (3,1.41) to[out=0,in=110] (4.5,0) -- (4.7,-0.5);
			\draw[color=\colorap] (-0.71,-2) -- (-0.71,3) to[out=90,in=-140] (0,4.5) -- (0.5,4.8);
			\draw[color=\colorap] (-2,2.82) -- (1.41,-2);



			\draw[color=\colorbm] (-2,0.71) -- (3,0.71) to[out=0,in=130] (4.5,0) -- (4.9,-0.5);
			\draw[color=\colorbm] (-1.41,-2) -- (-1.41,3) to[out=90,in=-170] (0,4.5) -- (0.5,4.55);
			\draw[color=\colorbm] (-2,-1.41) -- (3,2.12);


			 \node at (0,0) {$\bullet$};
			 \node[right] at (0.15,0) {$V_3$};
			 \node at (0,4.5) {$\bullet$};
			 \node[above] at (0,4.5) {$V_1$};
			 \node at (4.5,0) {$\bullet$};
			 \node[right] at (4.5,0) {$V_2$};
		 
			 \node at (1,1) {$\bullet$}; 
			 \node at (2,2) {$\bullet$}; 
			 \node at (-1,2) {$\bullet$}; 
			 \node at (2,-1) {$\bullet$}; 
		 
			 \node[text=\colorap] at (-0.71,1.41) {$\bullet$}; 
			 \node[text=\colorap] at (-1,1.41) {$\bullet$}; 
			 \node[text=\colorap] at (-0.71,1) {$\bullet$};  
		 
			 \node[text=\colorbm] at (-1.41,0.71) {$\bullet$}; 
			 \node[text=\colorbm] at (-1.41,-1) {$\bullet$}; 
			 \node[text=\colorbm] at (1,0.71) {$\bullet$}; 
		\end{scope}

	
	\begin{scope}[shift={(-4,-4.25)}]
	
			\node at (1,-2.5) {(C) The $(3,2)$-configuration $\C_{-1,1}$.};

			\draw[dashed, color=gray!50] (0,-2) -- (0,5) ;
			\draw[dashed, color=gray!50] (-2,0) -- (5,0) ;
			\draw[dashed, color=gray!50] (-2,4.5) -- (2.5,4.5) to[out=0,in=90] (4.5,2.5) -- (4.5,-2);

			\draw (-2,-1) -- (3,-1) to[out=0,in=-120] (4.5,0) -- (4.75,0.5);
			\draw (-2,2) -- (3,2) to[out=0,in=100] (4.5,0) -- (4.6,-0.5);
			\draw (-2,1) -- (3,1) to[out=0,in=120] (4.5,0) -- (4.8,-0.5);
			\draw (-1,-2) -- (-1,3) to[out=90,in=-150] (0,4.5) -- (0.5,4.75);
			\draw (1,-2) -- (1,3) to[out=90,in=-30] (0,4.5) -- (-0.5,4.75);
			\draw (2,-2) -- (2,3) to[out=90,in=-10] (0,4.5) -- (-0.5,4.55);
			\draw (-2,-2) -- (3,3);
			\draw (-2,1) -- (3,-1.5);
			\draw (-1.5,3) -- (1,-2);


			\draw[color=\coloram] (-2,-1.41) -- (3,-1.41) to[out=0,in=-100] (4.5,0) -- (4.55,0.5);
			\draw[color=\coloram] (0.71,-2) -- (0.71,3) to[out=90,in=-40] (0,4.5) -- (-0.5,4.85);
			\draw[color=\coloram] (-1.41,-2) -- (2.12,3);

			\draw[color=\colorbp] (-2,-0.71) -- (3,-0.71) to[out=0,in=-130] (4.5,0) -- (4.85,0.5);
			\draw[color=\colorbp] (1.41,-2) -- (1.41,3) to[out=90,in=-20] (0,4.5) -- (-0.5,4.65);
			\draw[color=\colorbp] (-2,1.41) -- (3,-2.12);



			 \node at (0,0) {$\bullet$};
			 \node[right] at (0.15,0) {$V_3$};
			 \node at (0,4.5) {$\bullet$};
			 \node[above] at (0,4.5) {$V_1$};
			 \node at (4.5,0) {$\bullet$};
			 \node[right] at (4.5,0) {$V_2$};
		 
			 \node at (1,1) {$\bullet$}; 
			 \node at (2,2) {$\bullet$}; 
			 \node at (-1,2) {$\bullet$}; 
			 \node at (2,-1) {$\bullet$}; 
		 
			 \node[text=\coloram] at (0.71,-1.41) {$\bullet$}; 
			 \node[text=\coloram] at (-1,-1.41) {$\bullet$}; 
			 \node[text=\coloram] at (0.71,1) {$\bullet$}; 
		 
			 \node[text=\colorbp] at (1.41,-0.71) {$\bullet$}; 
			 \node[text=\colorbp] at (1.41,-1) {$\bullet$}; 
			 \node[text=\colorbp] at (1,-0.71) {$\bullet$}; 
		\end{scope}

	
	\begin{scope}[shift={(4,-4.25)}]
	
			\node at (1,-2.5) {(D) The $(3,2)$-configuration $\C_{-1,-1}$.};

			\draw[dashed, color=gray!50] (0,-2) -- (0,5) ;
			\draw[dashed, color=gray!50] (-2,0) -- (5,0) ;
			\draw[dashed, color=gray!50] (-2,4.5) -- (2.5,4.5) to[out=0,in=90] (4.5,2.5) -- (4.5,-2);

			\draw (-2,-1) -- (3,-1) to[out=0,in=-120] (4.5,0) -- (4.75,0.5);
			\draw (-2,2) -- (3,2) to[out=0,in=100] (4.5,0) -- (4.6,-0.5);
			\draw (-2,1) -- (3,1) to[out=0,in=120] (4.5,0) -- (4.8,-0.5);
			\draw (-1,-2) -- (-1,3) to[out=90,in=-150] (0,4.5) -- (0.5,4.75);
			\draw (1,-2) -- (1,3) to[out=90,in=-30] (0,4.5) -- (-0.5,4.75);
			\draw (2,-2) -- (2,3) to[out=90,in=-10] (0,4.5) -- (-0.5,4.55);
			\draw (-2,-2) -- (3,3);
			\draw (-2,1) -- (3,-1.5);
			\draw (-1.5,3) -- (1,-2);


			\draw[color=\coloram] (-2,-1.41) -- (3,-1.41) to[out=0,in=-100] (4.5,0) -- (4.55,0.5);
			\draw[color=\coloram] (0.71,-2) -- (0.71,3) to[out=90,in=-40] (0,4.5) -- (-0.5,4.85);
			\draw[color=\coloram] (-1.41,-2) -- (2.12,3);


			\draw[color=\colorbm] (-2,0.71) -- (3,0.71) to[out=0,in=130] (4.5,0) -- (4.9,-0.5);
			\draw[color=\colorbm] (-1.41,-2) -- (-1.41,3) to[out=90,in=-170] (0,4.5) -- (0.5,4.55);
			\draw[color=\colorbm] (-2,-1.41) -- (3,2.12);


			 \node at (0,0) {$\bullet$};
			 \node[right] at (0.15,0) {$V_3$};
			 \node at (0,4.5) {$\bullet$};
			 \node[above] at (0,4.5) {$V_1$};
			 \node at (4.5,0) {$\bullet$};
			 \node[right] at (4.5,0) {$V_2$};
		 
			 \node at (1,1) {$\bullet$}; 
			 \node at (2,2) {$\bullet$}; 
			 \node at (-1,2) {$\bullet$}; 
			 \node at (2,-1) {$\bullet$}; 
		 
			 \node[text=\coloram] at (0.71,-1.41) {$\bullet$}; 
			 \node[text=\coloram] at (-1,-1.41) {$\bullet$}; 
			 \node[text=\coloram] at (0.71,1) {$\bullet$}; 
		 
			 \node[text=\colorbm] at (-1.41,0.71) {$\bullet$}; 
			 \node[text=\colorbm] at (-1.41,-1) {$\bullet$}; 
			 \node[text=\colorbm] at (1,0.71) {$\bullet$}; 
		\end{scope}
		
		\end{tikzpicture}

	\caption{The $(3,2)$-configurations $\C_{\alpha,\beta}$. In black, the common points $\V$, $\S$ and their lines. In color, the points and lines corresponding to each configuration for {\color{blue}{$\alpha=1$}}, {\color{teal}{$\alpha=-1$}}, {\color{red}{$\beta=1$}}, {\color{orange}{$\beta=-1$}}.}\label{fig:dual_ZP}
\end{figure}

\begin{propo}\label{propo:ZP_comb}
	For any $\alpha,\beta\in\set{-1,1}$, the configuration $\C_{\alpha,\beta}$ is stable and has the following combinatorics
	\begin{equation*}
		\begin{array}{c}
			\left\{\ 
				\big\{V_1,S_1,S_{10}^\beta\big\},\ \big\{V_1,S_2,S_4\big\},\ \big\{V_1,S_3,S_6^\alpha\big\},\ \big\{V_1,S_5^\alpha,S_7^\alpha\big\},\ \big\{V_1,S_8^\beta,S_9^\beta\big\},  
				\right. \\[0.5em]
				\quad \big\{V_2,S_1,S_7^\alpha\big\},\ \big\{V_2,S_2,S_3\big\},\ \big\{V_2,S_4,S_9^\beta\big\},\ \big\{V_2,S_5^\alpha,S_6^\alpha\big\},\ \big\{V_2,S_8^\beta,S_{10}^\beta\big\},  \\[0.5em]
				\left. 
				\qquad \big\{V_3,S_1,S_2\big\},\ \big\{V_3,S_3,S_5^\alpha\big\},\ \big\{V_3,S_4,S_8^\beta\big\},\ \big\{V_3,S_6^\alpha,S_7^\alpha\big\},\ \big\{V_3,S_9^\beta,S_{10}^\beta\big\}\ 
			\right\}.
		\ \end{array}
	\end{equation*}
\end{propo}

\begin{proof}
	The combinatorics can be computed using the coordinates of the vertices and surrounding-points. Let us prove that $\V$ is stabilized by any automorphism of the combinatorics. The points of $\V$ are the only ones contained in five different 3-point lines (since any surrounding point is contained in exactly three 3-point lines). By definition, the automorphisms of the combinatorics respect the collinearity, thus they fix globally the vertices, and we obtain that the configuration $\C_{\alpha,\beta}$ is stable for any $\alpha,\beta\in\set{-1,1}$.
\end{proof}

\begin{rmk}
	In fact, we can prove that the automorphism group is of order $4$ and generated by the permutations
	\begin{equation*}
	\begin{array}{l}
		\sigma_1 :  \left(V_1,V_2,V_3,S_1,\dots,S_{10}^\beta\right) \mapsto \left(V_2,V_1,V_3, S_2, S_1, S_{10}^\beta,S_7^\alpha,S_9^\beta,S_8^\beta,S_4,S_6^\alpha,S_5^\alpha,S_3\right), \\[1em]
		\sigma_2 :  \left(V_1,V_2,V_3,S_1,\dots,S_{10}^\beta\right) \mapsto \left(V_2,V_1,V_3, S_1, S_2, S_4, S_3, S_8^\beta, S_9^\beta, S_{10}^\beta, S_5^\alpha, S_6^\alpha, S_7^\alpha\right). \\
	\end{array}
\end{equation*}
They can be seen as two geometric symmetries in Figure~\ref{fig:dual_ZP}: the permutation $\sigma_1$ is induced by the inversion of $S_1$ and $S_2$, and $\sigma_2$ corresponds to the symmetry with axis $\ell:\ x=y$.
\end{rmk}

For any $\alpha,\beta\in\{-1,1\}$, we denote by $\A^{\alpha,\beta}$ the dual arrangement of the $(3,2)$-configuration $\C_{\alpha,\beta}$. Recall that, from Proposition~\ref{propo:ZP_comb}, $\A^{1,1}$, $\A^{1,-1}$, $\A^{-1,1}$, $\A^{-1,-1}$ have the same combinatorics.

\begin{thm}\label{thm:homeo_ZP13}
	Let $\alpha,\alpha',\beta,\beta'\in\{-1,1\}$ be such that $\alpha\beta\neq\alpha'\beta'$. There is no homeomorphism between $\left(\PC^2, \A^{\alpha,\beta}\right)$ and $\left(\PC^2, \A^{\alpha',\beta'}\right)$.
\end{thm}

\begin{proof}
	We proved in Proposition~\ref{propo:ZP_comb} that the $(3,2)$-configuration $\C_{\alpha,\beta}$ is stable for any $\alpha,\beta\in\{-1,1\}$; in addition, the associated $2$-plumbing of any of these configurations is constant on $\S$. By Corollary~\ref{cor:invariant}, it is sufficient to compute the chamber weight of each configuration and check that they differ if and only if $\alpha\beta\neq\alpha'\beta'$.
	
	We may assume, up to projective transformation, that $\ell=(V_1,V_2)$ is the line at infinity in $\PR^2$, and that $(V_1,V_3)$ and $(V_2,V_3)$ represent the canonical axis of $\RR^2=\PR^2\setminus\ell$. By Remark~\ref{rmk:parity}, computing the chamber weight of each configuration amounts to determine the parity on the number of surrounding-points in, for example, the first quadrant $\ch_1$ of $\RR^2$. Using Figure~\ref{fig:dual_ZP}, we conclude the proof with the following computations: there is an even number of surrounding-points in $\ch_1$ for the configurations $\C_{1,1}$ and $\C_{-1,-1}$, and an odd number of surrounding-points for $\C_{1,-1}$ and $\C_{-1,1}$ in the same quadrant. Then, we have the corresponding chamber weights $\tau(\C_{1,1})=\tau(\C_{-1,-1})=0$ and $\tau(\C_{1,-1})=\tau(\C_{-1,1})=1$, and the result follows.
\end{proof}

\begin{cor}
	For any $\alpha,\beta,\alpha',\beta'\in\set{1,-1}$ such that $\alpha\beta\neq\alpha'\beta'$, the arrangements $\A^{\alpha,\beta}$ and $\A^{\alpha',\beta'}$ form a Zariski pair.
\end{cor}

It is worth pointing out that the above result provides the first example of Zariski pair of line arrangements defined by rational coefficients. 

\subsection{Some properties of this Zariski pair}

\subsubsection{A Zariski-like geometric property}

As in Zariski's original example, the above Zariski pair can be distinguished by geometric properties. Let us denote by $L_{i,j}$ the line joining $S_i$ and $S_j$ in $C_{\alpha,\beta}$.

\begin{propo}\label{propo:pts_align}
	If the lines $L_{1,2}$, $L_{5,6}$ and $L_{8,9}$ are concurrent (or equivalently, the points $L_{1,2}^\bullet$, $L_{5,6}^\bullet$ and $L_{8,9}^\bullet$ are collinear) then $\tau(\C_{\alpha,\beta})=0$, otherwise $\tau(\C_{\alpha,\beta})=1$.
\end{propo}

\begin{rmk}
	This concurrency of lines is not the only one present in these pairs (see for example $L_{1,2}$, $L_{5,7}$ and $L_{8,10}$) but it is the only one which seems to be invariant up to lattice preserving deformation.	
	
	The above condition seems to characterize the connected components of the realization space of the combinatorics of $\C_{\alpha,\beta}$ where the chamber weight $\tau$ is zero. This is the case for one of the degenerations of this pair presented in Section~\ref{sec:deg}, as it is proved in Proposition~\ref{propo:geom_char}.
\end{rmk}

\subsubsection{Arrangements which are $C_{\leq 3}$ but not of simple type}

In~\cite{NY:connectivity}, Nazir and Yoshinaga defined the family of \emph{$C_{\leq 3}$ arrangements}, composed of those which possess at most 3 lines $\{L,L',L''\}$ containing any singular point of multiplicity greater than 3. In addition, a $C_{\leq 3}$ arrangement is of \emph{simple type} if either $L,L',L''$ are concurrent or if exactly one of these lines contains an unique point of multiplicity greater than 3.

They prove in~\cite[Thm.~3.11]{NY:connectivity} that the moduli spaces of $C_{\leq 3}$ of simple type are irreducible. As a consequence, such arrangements cannot form a Zariski pair.\\

Note that the arrangements $\A^{\alpha,\beta}$ are $C_{\leq 3}$ arrangements. Indeed, any singular point of multiplicity greater than 3 is contained on the support of the DPA of $\C_{\alpha,\beta}$. However, they are not of simple type since any line of the support contains exactly 5 singular points of multiplicity greater than 3. This fact illustrates the importance of the hypothesis ``of simple type'' in the result of Nazir and Yoshinaga.

\begin{rmk}
	In fact, the second condition in Definition~\ref{def:configuration} implies that the dual arrangement of any planar $(3,m)$-configuration is neither a $C_{\leq 3}$ arrangement nor a $C_{\leq 3}$ arrangement of non-simple type.
\end{rmk}

\subsubsection{Milnor fiber}

For a line arrangement $\A$, consider the \emph{Milnor fiber} of $\A$ as the affine hypersurface in $\CC^3$ defined by $\F_\A=\Q_\A^{-1}(1)$, where $\Q_\A$ is the defining polynomial of $\A$, i.e. a square-free product of all linear forms $\alpha_L$ defining each line $L\in\A$. Define the complex number $\lambda=\exp(2\pi i/n)$, where $n$ is the number of lines of $\A$ and consider the cyclic group $\mu_n=\langle\lambda\rangle$. Following \cite{CS:milnor}, an element of $\mu_n$ defines a \emph{monodromy automorphism} $h:\F_\A\to\F_\A$, which induces an algebraic monodromy $\HH^1(\F_\A;\CC)\to \HH^1(\F_\A;\CC)$. Consider the eigenspace decomposition of $\HH^1(\F_\A;\CC)$ (see~\cite[Theorem 1.6]{CS:milnor})
\begin{equation}\label{eq:Milnor_decomp}
		\HH^1(\F_\A;\CC)\simeq\bigoplus_{\xi\in\mu_{n}} \HH^1(\F_\A;\CC)_{\xi}.
\end{equation}
In fact, we can identify $\HH^1(\F_\A;\CC)_1\simeq \HH^1(\PC^2\setminus \A;\CC)$, which is determined by the combinatorics of $\A$. In the case of complexified real arrangements, Yoshinaga~\cite{Yoshinaga:milnor_fibers} gives a geometric algorithm to compute the above eigenspaces.

For any $\alpha,\beta\in\{-1,1\}$, consider $\F^{\alpha,\beta}$ being the respective Milnor fiber of $\A^{\alpha,\beta}$. Remark that the arrangements $\A^{\alpha,\beta}$ have exactly 13 lines and only double and triple points. As a direct consequence, we obtain the following proposition.

\begin{propo}\label{propo:Milnor}
	For any $\alpha,\beta\in\{-1,1\}$, we have that $\HH^1(\F^{\alpha,\beta};\CC)\simeq \HH^1(\PC^2\setminus \A^{\alpha,\beta};\CC)$. Furthermore, the characteristic polynomial of the Milnor fiber monodromy $h$ is combinatorially determined, as well as the Betti numbers of the Milnor fiber.
\end{propo}

\begin{proof}
	Consider the eigenspace decomposition of $\HH^1(\F^{\alpha,\beta};\CC)$ given in (\ref{eq:Milnor_decomp}), where we have $\HH^1(\F^{\alpha,\beta};\CC)_1\simeq \HH^1(\PC^2\setminus \A^{\alpha,\beta};\CC)$. As the number of lines of the arrangement $\A^{\alpha,\beta}$ is prime, by~\cite[Corollary 3.12]{Yoshinaga:milnor_fibers}, we deduce that $\HH^1(\F^{\alpha,\beta};\CC)_{\lambda^k}\simeq 0$ for any $k\not\equiv 0\ \text{mod}\ 13$. Then, $\HH^1(\F^{\alpha,\beta};\CC)\simeq \HH^1(\PC^2\setminus \A^{\alpha,\beta};\CC)$, which is combinatorially determined~\cite{OS:combinatorics}. In particular,  the characteristic polynomial of the monodromy $h$ and the Betti numbers of $\F^{\alpha,\beta}$ are also combinatorially determined.
\end{proof}

\bigskip
\section{Degenerations of the configuration $\C_{\alpha,\beta}$}\label{sec:deg}
Using the Zariski pair given in the previous section, we construct two degenerations, which provide other examples of Zariski pairs with 13 lines. Indeed, the $\C_{\alpha,\beta}$ configurations are not rigid, since there exist several possible deformations. For example, we can move the point $S_3$ on the line $(V_2,S_2)$ or the point $S_4$ along the line $(V_1,S_2)$. The first move induces modifications on the surrounding-points $S_5^\alpha$, $S_6^\alpha$ and $S_7^\alpha$, while the second one modifies $S_8^\beta$, $S_9^\beta$ and  $S_{10}^\beta$. We use these two deformations in order to construct degenerated versions of the previous Zariski pair producing one or two points of multiplicity 5 in the new arrangements.

In this section, the proofs are similar to those in Section~\ref{sec:ZP} and are left for the reader.

\subsection{With 1 point of multiplicity 5}\mbox{}

Moving the point $S_3$ along the line $(V_2,S_2)$, we can fix it on the intersection point of $(V_1,S_1)$ and $(V_2,S_2)$. Then we create a 5-point line composed of $V_1,S_1,S_3,S_6^\alpha,S_{10}^\beta$. The planar $(3,2)$-configurations obtained in this way (see Figure~\ref{fig:dual_ZP13_1pt5}) are defined by $\C_{\alpha,\beta}^1=(\V^1,\S_{\alpha,\beta}^1,\L_{\alpha,\beta}^1,\pl)$, where $\V^1=\V=\set{V_1,V_2,V_3}$, with
\[
	V_1=(0:1:0),\quad V_2=(1:0:0),\quad V_3=(0:0:1),
\]
and $\S_{\alpha,\beta}^1=\set{S_1,S_2,S_3,S_4}\sqcup\set{S_5^\alpha,S_6^\alpha,S_7^\alpha}\sqcup\set{S_8^\beta,S_9^\beta,S_{10}^\beta}$ where
\[
	S_1=(1:1:1),\quad S_2=(4:4:1),\quad S_3=(1:4:1),\quad S_4=(8:-1:2),
\]
\[
	S_5^\alpha=(1:4:2\alpha),\quad S_6^\alpha=(1:2\alpha :1),\quad S_7^\alpha=(\alpha:2:2),
\]
\[
	S_8^\beta=(8:-1:4\beta),\quad S_9^\beta=(4\beta:-1:2),\quad S_{10}^\beta=(4:-\beta:4).
\]

\begin{figure}[h!]
	\begin{tikzpicture}

	\def\colorap{blue} 
	\def\coloram{teal}  
	\def\colorbp{red} 
	\def\colorbm{orange} 
	
\begin{scope}[shift={(-4,0)}]
		\draw[dashed, color=gray!50] (0,-2) -- (0,5) ;
		\draw[dashed, color=gray!50] (-2,0) -- (5,0) ;
		\draw[dashed, color=gray!50] (-2,4.5) -- (2.5,4.5) to[out=0,in=90] (4.5,2.5) -- (4.5,-2);

		\draw (-2,-1) -- (3,-1) to[out=0,in=-120] (4.5,0) -- (4.75,0.5);
		\draw (-2,2) -- (3,2) to[out=0,in=100] (4.5,0) -- (4.6,-0.5);
		\draw (-2,1) -- (3,1) to[out=0,in=120] (4.5,0) -- (4.8,-0.5);
		\draw (1,-2) -- (1,3) to[out=90,in=-30] (0,4.5) -- (-0.5,4.75);
		\draw (2,-2) -- (2,3) to[out=90,in=-10] (0,4.5) -- (-0.5,4.55);
		\draw (-2,-2) -- (3,3);
		\draw (-2,1) -- (3,-1.5);
		\draw (1.5,3) -- (-1,-2);

		\draw[color=\colorap] (-2,1.41) -- (3,1.41) to[out=0,in=110] (4.5,0) -- (4.7,-0.5);
		\draw[color=\colorap] (0.71,-2) -- (0.71,3) to[out=90,in=-40] (0,4.5) -- (-0.5,4.85);
		\draw[color=\colorap] (2,2.82) -- (-1.41,-2);


		\draw[color=\colorbp] (-2,-0.71) -- (3,-0.71) to[out=0,in=-130] (4.5,0) -- (4.85,0.5);
		\draw[color=\colorbp] (1.41,-2) -- (1.41,3) to[out=90,in=-20] (0,4.5) -- (-0.5,4.65);
		\draw[color=\colorbp] (-2,1.41) -- (3,-2.12);



		 \node at (0,0) {$\bullet$};
		 \node[right] at (0.15,0) {$V_3$};
		 \node at (0,4.5) {$\bullet$};
		 \node[above] at (0,4.5) {$V_1$};
		 \node at (4.5,0) {$\bullet$};
		 \node[right] at (4.5,0) {$V_2$};
	 
		 \node at (1,1) {$\bullet$}; 
		 \node at (2,2) {$\bullet$}; 
		 \node at (1,2) {$\bullet$}; 
		 \node at (2,-1) {$\bullet$}; 
	 
		 \node[text=\colorap] at (0.71,1.41) {$\bullet$}; 
		 \node[text=\colorap] at (1,1.41) {$\bullet$}; 
		 \node[text=\colorap] at (0.71,1) {$\bullet$};  
	 
		 \node[text=\colorbp] at (1.41,-0.71) {$\bullet$}; 
		 \node[text=\colorbp] at (1.41,-1) {$\bullet$}; 
		 \node[text=\colorbp] at (1,-0.71) {$\bullet$}; 
	\end{scope}

	\begin{scope}[shift={(4,0)}]
	
		\draw[dashed, color=gray!50] (0,-2) -- (0,5) ;
		\draw[dashed, color=gray!50] (-2,0) -- (5,0) ;
		\draw[dashed, color=gray!50] (-2,4.5) -- (2.5,4.5) to[out=0,in=90] (4.5,2.5) -- (4.5,-2);

		\draw (-2,-1) -- (3,-1) to[out=0,in=-120] (4.5,0) -- (4.75,0.5);
		\draw (-2,2) -- (3,2) to[out=0,in=100] (4.5,0) -- (4.6,-0.5);
		\draw (-2,1) -- (3,1) to[out=0,in=120] (4.5,0) -- (4.8,-0.5);
		\draw (1,-2) -- (1,3) to[out=90,in=-30] (0,4.5) -- (-0.5,4.75);
		\draw (2,-2) -- (2,3) to[out=90,in=-10] (0,4.5) -- (-0.5,4.55);
		\draw (-2,-2) -- (3,3);
		\draw (-2,1) -- (3,-1.5);
		\draw (1.5,3) -- (-1,-2);


		\draw[color=\coloram] (-2,-1.41) -- (3,-1.41) to[out=0,in=-100] (4.5,0) -- (4.55,0.5);
		\draw[color=\coloram] (-0.71,-2) -- (-0.71,3) to[out=90,in=-140] (0,4.5) -- (0.5,4.8);
		\draw[color=\coloram] (1.41,-2) -- (-2.12,3);

		\draw[color=\colorbp] (-2,-0.71) -- (3,-0.71) to[out=0,in=-130] (4.5,0) -- (4.85,0.5);
		\draw[color=\colorbp] (1.41,-2) -- (1.41,3) to[out=90,in=-20] (0,4.5) -- (-0.5,4.65);
		\draw[color=\colorbp] (-2,1.41) -- (3,-2.12);



		 \node at (0,0) {$\bullet$};
		 \node[right] at (0.15,0) {$V_3$};
		 \node at (0,4.5) {$\bullet$};
		 \node[above] at (0,4.5) {$V_1$};
		 \node at (4.5,0) {$\bullet$};
		 \node[right] at (4.5,0) {$V_2$};
	 
		 \node at (1,1) {$\bullet$}; 
		 \node at (2,2) {$\bullet$}; 
		 \node at (1,2) {$\bullet$}; 
		 \node at (2,-1) {$\bullet$}; 
	 
		 \node[text=\coloram] at (-0.71,-1.41) {$\bullet$}; 
		 \node[text=\coloram] at (1,-1.41) {$\bullet$}; 
		 \node[text=\coloram] at (-0.71,1) {$\bullet$}; 
	 
		 \node[text=\colorbp] at (1.41,-0.71) {$\bullet$}; 
		 \node[text=\colorbp] at (1.41,-1) {$\bullet$}; 
		 \node[text=\colorbp] at (1,-0.71) {$\bullet$}; 
	\end{scope}
	\end{tikzpicture}
	\caption{The $(3,2)$-configurations $\C^1_{1,1}$ and $\C^1_{1,-1}$. In black, the common points $\V$, $\S$ and their lines. In color, the points and lines corresponding to each configuration for {\color{blue}{$\alpha=1$}}, {\color{teal}{$\alpha=-1$}} and {\color{red}{$\beta=1$}}.}\label{fig:dual_ZP13_1pt5}
\end{figure}
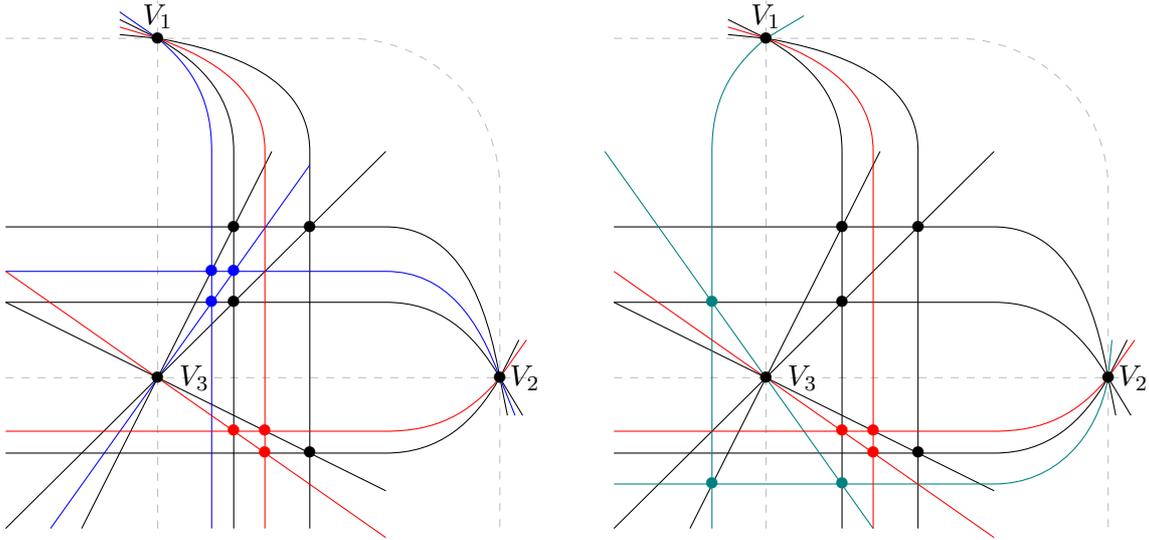

\begin{propo}\label{propo:combi_ZP_1pt5}
	For any $\alpha,\beta\in\set{-1,1}$, the $(3,2)$-configuration $\C^1_{\alpha,\beta}$ is stable and has the following combinatorics
	\begin{equation*}
		\begin{array}{c}
			\left\{\ 
				\big\{V_1,S_1,S_3,S_6^\alpha,S_{10}^\beta\big\},\ \big\{V_1,S_2,S_4\big\},\ \big\{V_1,S_5^\alpha,S_7^\alpha\big\},\ \big\{V_1,S_8^\beta,S_9^\beta\big\},  
				\right. \\[0.5em]
				\big\{V_2,S_1,S_7^\alpha\big\},\ \big\{V_2,S_2,S_3\big\},\ \big\{V_2,S_4,S_9^\beta\big\},\ \big\{V_2,S_5^\alpha,S_6^\alpha\big\},\ \big\{V_2,S_8^\beta,S_{10}^\beta\big\},  \\[0.5em]
				\left. 
				\big\{V_3,S_1,S_2\big\},\ \big\{V_3,S_3,S_5^\alpha\big\},\ \big\{V_3,S_4,S_8^\beta\big\},\ \big\{V_3,S_6^\alpha,S_7^\alpha\big\},\ \big\{V_3,S_9^\beta,S_{10}^\beta\big\}\
			\right\}.
		\ \end{array}
	\end{equation*}
\end{propo}

\begin{proof}
	The description of combinatorics follows from the construction of $\C^1_{\alpha,\beta}$. Note that $\V^1$ is combinatorially characterized by the following properties: $V_1$ is the only point of $\C^1_{\alpha,\beta}$ possessing one~5-point line and three~3-point line, while $V_2$ and $V_3$ are the only one having five~3-point lines. Then, any automorphism of the combinatorics fixes $V_1$ and $\{V_2,V_3\}$, implying that $\C^1_{\alpha,\beta}$ is stable.
\end{proof}

Denote by $\A_1^{\alpha,\beta}$ the respective dual arrangements of $\C^1_{\alpha,\beta}$. Computing $\tau$ for the previous four configurations, we obtain the following result.

\begin{thm}
	Let $\alpha,\beta,\alpha',\beta'\in\set{-1,1}$ be such that $\alpha\beta\neq\alpha'\beta'$. There is no homeomorphism between $(\PC^2,\A_1^{\alpha,\beta})$ and $(\PC^2,\A_1^{\alpha',\beta'})$.
\end{thm}

\begin{proof}
	For any $\alpha,\beta\in\{-1,1\}$, the (3,2)-configuration $\C^1_{\alpha,\beta}$ is stable by Proposition~\ref{propo:combi_ZP_1pt5}. Using a similar argument as in the proof of Theorem~\ref{thm:homeo_ZP13}, we have $\tau(\C^1_{1,1})=\tau(\C^1_{-1,-1})=0$ mod $2$ and $\tau(\C^1_{1,-1})=\tau(\C^1_{-1,1})=1$ mod $2$. The result follows by Corollary~\ref{cor:invariant}.
\end{proof}

\begin{cor}
	If $\alpha,\beta,\alpha',\beta'\in\set{1,-1}$ are such that $\alpha\beta\neq\alpha'\beta'$ then $\A_1^{\alpha,\beta}$ and $\A_1^{\alpha',\beta'}$ form a Zariski pair.
\end{cor}

	We can repeat the same procedure by moving the point $S_4$ on the intersection of $(V_1,S_2)$ and $(V_2,S_1)$ instead of moving the point $S_3$, but the configuration obtained with this modification is $\C^1_{\alpha,\beta}$ up to symmetry.
	
\subsection{With 2 points of multiplicity 5}\mbox{}

We can perform the above procedure by moving both $S_3$ and $S_4$. We fix the first point on the intersection of $(V_1,S_1)$ and $(V_2,S_2)$, and thus $S_4$ as the intersection point of $(V_1,S_2)$ and $(V_2,S_1)$. In this way, we create two~$5$-point lines: $V_1,S_1,S_3,S_6^\alpha,S_{10}^\beta$ in a first time, and $V_2,S_1,S_4,S_7^\alpha,S_9^\beta$ in a second one. The configurations obtained are in Figure~\ref{fig:dual_ZP13_2pts5} (for $(\alpha,\beta)=(1,1)$ and $(\alpha,\beta)=(1,-1)$). They are defined by $\C^2_{\alpha,\beta}=(\V^2,\S_{\alpha,\beta}^2,\L_{\alpha,\beta}^2,\pl)$, where $\V^2=\V=\set{V_1,V_2,V_3}$, with
\[
	V_1=(0:1:0),\quad V_2=(1:0:0),\quad V_3=(0:0:1),
\]
and $\S_{\alpha,\beta}^2=\set{S_1,S_2,S_3,S_4}\sqcup\set{S_5^\alpha,S_6^\alpha,S_7^\alpha}\sqcup\set{S_8^\beta,S_9^\beta,S_{10}^\beta}$ with
\[
	S_1=(1:1:1),\quad S_2=(4:4:1),\quad S_3=(1:4:1),\quad S_4=(4:1:1),
\]
\[
	S_5^\alpha=(1:4:2\alpha),\quad S_6^\alpha=(1:2\alpha :1),\quad S_7^\alpha=(\alpha:2:2),
\]
\[
	S_8^\beta=(4:1:2\beta),\quad S_9^\beta=(2\beta:1:1),\quad S_{10}^\beta=(2:\beta:2).
\]

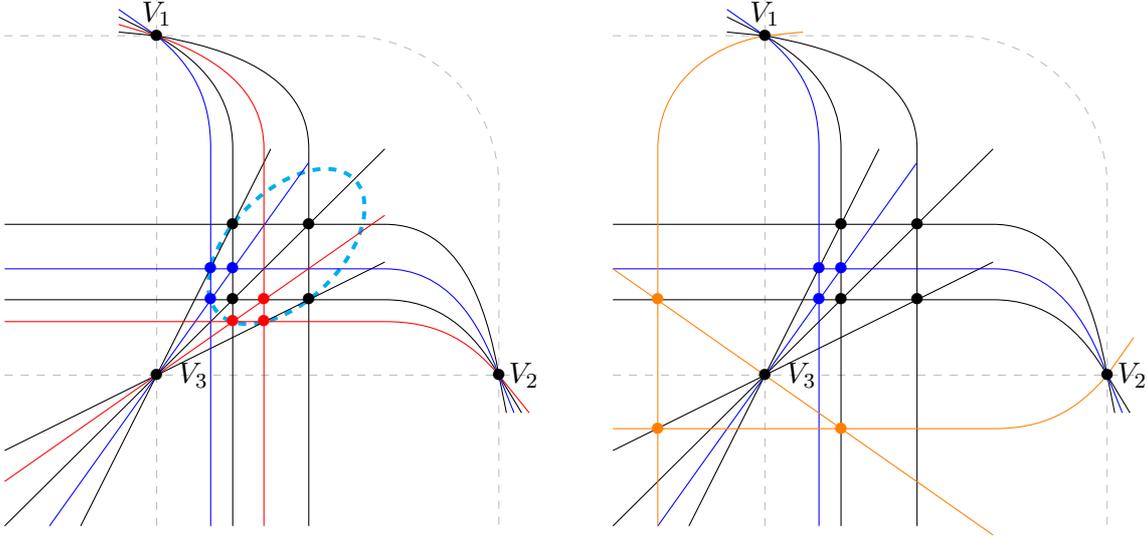
\begin{figure}[h!]
	\begin{tikzpicture}

	\def\colorap{blue} 
	\def\coloram{teal}  
	\def\colorbp{red} 
	\def\colorbm{orange} 
	
\begin{scope}[shift={(-4,0)}]
	\tikzset{
	    myConic/.style = {color=cyan, dashed, smooth, samples=200, line width=1.5}
	}
	\draw[domain=sqrt(2)/2-0.025:2*sqrt(2)-0.09062][myConic] plot[variable=\x] ({\x}, { (1/4*sqrt(2)*(sqrt(2)*\x + sqrt(2) - sqrt(-6*\x*\x + 6*\x*(sqrt(2) + 2)
- 10*sqrt(2) + 3) + 1) });
	\draw[domain=sqrt(2)/2-0.025:2*sqrt(2)-0.09062][myConic] plot[variable=\x] ({\x}, { (1/4*sqrt(2)*(sqrt(2)*\x + sqrt(2) + sqrt(-6*\x*\x + 6*\x*(sqrt(2) + 2)
- 10*sqrt(2) + 3) + 1) });

		\draw[dashed, color=gray!50] (0,-2) -- (0,5) ;
		\draw[dashed, color=gray!50] (-2,0) -- (5,0) ;
		\draw[dashed, color=gray!50] (-2,4.5) -- (2.5,4.5) to[out=0,in=90] (4.5,2.5) -- (4.5,-2);

		\draw (-2,2) -- (3,2) to[out=0,in=100] (4.5,0) -- (4.6,-0.5);
		\draw (-2,1) -- (3,1) to[out=0,in=120] (4.5,0) -- (4.8,-0.5);
		\draw (1,-2) -- (1,3) to[out=90,in=-30] (0,4.5) -- (-0.5,4.75);
		\draw (2,-2) -- (2,3) to[out=90,in=-10] (0,4.5) -- (-0.5,4.55);
		\draw (-2,-2) -- (3,3);
		\draw (-2,-1) -- (3,1.5);
		\draw (1.5,3) -- (-1,-2);

		\draw[color=\colorap] (-2,1.41) -- (3,1.41) to[out=0,in=110] (4.5,0) -- (4.7,-0.5);
		\draw[color=\colorap] (0.71,-2) -- (0.71,3) to[out=90,in=-40] (0,4.5) -- (-0.5,4.85);
		\draw[color=\colorap] (2,2.82) -- (-1.41,-2);


		\draw[color=\colorbp] (-2,0.71) -- (3,0.71) to[out=0,in=130] (4.5,0) -- (4.9,-0.5);
		\draw[color=\colorbp] (1.41,-2) -- (1.41,3) to[out=90,in=-20] (0,4.5) -- (-0.5,4.65);
		\draw[color=\colorbp] (-2,-1.41) -- (3,2.12);



		 \node at (0,0) {$\bullet$};
		 \node[right] at (0.15,0) {$V_3$};
		 \node at (0,4.5) {$\bullet$};
		 \node[above] at (0,4.5) {$V_1$};
		 \node at (4.5,0) {$\bullet$};
		 \node[right] at (4.5,0) {$V_2$};
	 
		 \node at (1,1) {$\bullet$}; 
		 \node at (2,2) {$\bullet$}; 
		 \node at (1,2) {$\bullet$}; 
		 \node at (2,1) {$\bullet$}; 
	 
		 \node[text=\colorap] at (0.71,1.41) {$\bullet$}; 
		 \node[text=\colorap] at (1,1.41) {$\bullet$}; 
		 \node[text=\colorap] at (0.71,1) {$\bullet$};  
	 
		 \node[text=\colorbp] at (1.41,0.71) {$\bullet$}; 
		 \node[text=\colorbp] at (1.41,1) {$\bullet$}; 
		 \node[text=\colorbp] at (1,0.71) {$\bullet$}; 
	\end{scope}

	\begin{scope}[shift={(4,0)}]
		\draw[dashed, color=gray!50] (0,-2) -- (0,5) ;
		\draw[dashed, color=gray!50] (-2,0) -- (5,0) ;
		\draw[dashed, color=gray!50] (-2,4.5) -- (2.5,4.5) to[out=0,in=90] (4.5,2.5) -- (4.5,-2);

		\draw (-2,2) -- (3,2) to[out=0,in=100] (4.5,0) -- (4.6,-0.5);
		\draw (-2,1) -- (3,1) to[out=0,in=120] (4.5,0) -- (4.8,-0.5);
		\draw (1,-2) -- (1,3) to[out=90,in=-30] (0,4.5) -- (-0.5,4.75);
		\draw (2,-2) -- (2,3) to[out=90,in=-10] (0,4.5) -- (-0.5,4.55);
		\draw (-2,-2) -- (3,3);
		\draw (-2,-1) -- (3,1.5);
		\draw (1.5,3) -- (-1,-2);

		\draw[color=\colorap] (-2,1.41) -- (3,1.41) to[out=0,in=110] (4.5,0) -- (4.7,-0.5);
		\draw[color=\colorap] (0.71,-2) -- (0.71,3) to[out=90,in=-40] (0,4.5) -- (-0.5,4.85);
		\draw[color=\colorap] (2,2.82) -- (-1.41,-2);



		\draw[color=\colorbm] (-2,-0.71) -- (3,-0.71) to[out=0,in=-130] (4.5,0) -- (4.85,0.5);
		\draw[color=\colorbm] (-1.41,-2) -- (-1.41,3) to[out=90,in=-170] (0,4.5) -- (0.5,4.55);
		\draw[color=\colorbm] (-2,1.41) -- (3,-2.12);


		 \node at (0,0) {$\bullet$};
		 \node[right] at (0.15,0) {$V_3$};
		 \node at (0,4.5) {$\bullet$};
		 \node[above] at (0,4.5) {$V_1$};
		 \node at (4.5,0) {$\bullet$};
		 \node[right] at (4.5,0) {$V_2$};
	 
		 \node at (1,1) {$\bullet$}; 
		 \node at (2,2) {$\bullet$}; 
		 \node at (1,2) {$\bullet$}; 
		 \node at (2,1) {$\bullet$}; 
	 
		 \node[text=\colorap] at (0.71,1.41) {$\bullet$}; 
		 \node[text=\colorap] at (1,1.41) {$\bullet$}; 
		 \node[text=\colorap] at (0.71,1) {$\bullet$};  
	 
		 \node[text=\colorbm] at (-1.41,-0.71) {$\bullet$}; 
		 \node[text=\colorbm] at (-1.41,1) {$\bullet$}; 
		 \node[text=\colorbm] at (1,-0.71) {$\bullet$}; 
	\end{scope}
	\end{tikzpicture}
	\caption{The $(3,2)$-configurations $\C^2_{1,1}$ and $\C^2_{1,-1}$, as well as the conic joining the six points $S_3,S_4,S_5^+,S_7^+,S_8^+,S_{10}^+$. In black, the common points $\V^2$, $\S^2$ and their lines. In color, the points and lines corresponding to each configuration for {\color{blue}{$\alpha=1$}}, {\color{red}{$\beta=1$}}, {\color{orange}{$\beta=-1$}}. }\label{fig:dual_ZP13_2pts5}
\end{figure}

\begin{propo}\label{propo:combi_ZP_2pts5}
	For any $\alpha,\beta\in\set{-1,1}$, the $(3,2)$-configuration $\C^2_{\alpha,\beta}$ is stable and has the following combinatorics
	\begin{equation*}
		\begin{array}{c}
			\left\{\ 
				\big\{V_1,S_1,S_3,S_6^\alpha,S_{10}^\beta\big\},\ \big\{V_1,S_2,S_4\big\},\ \big\{V_1,S_5^\alpha,S_7^\alpha\big\},\ \big\{V_1,S_8^\beta,S_9^\beta\big\},  
				\right. \\[0.5em]
				\quad \big\{V_2,S_1,S_4,S_7^\alpha,S_9^\beta\big\},\ \big\{V_2,S_2,S_3\big\},\ \big\{V_2,S_5^\alpha,S_6^\alpha\big\},\ \big\{V_2,S_8^\beta,S_{10}^\beta\big\},  \\[0.5em]
				\left. \qquad
				\big\{V_3,S_1,S_2\big\},\ \big\{V_3,S_3,S_5^\alpha\big\},\ \big\{V_3,S_4,S_8^\beta\big\},\ \big\{V_3,S_6^\alpha,S_7^\alpha\big\},\ \big\{V_3,S_9^\beta,S_{10}^\beta\big\}\
			\right\}.
		\ \end{array}
	\end{equation*}
\end{propo}

We denote by $\A_2^{\alpha,\beta}$ the respective dual arrangements of $\C^2_{\alpha,\beta}$. Computing the chamber weight $\tau$ for these four configurations, we obtain the following return.

\begin{thm}\label{thm:ZP13-2pts5}
	Let $\alpha,\beta,\alpha',\beta'\in\set{-1,1}$ be such that $\alpha\beta\neq\alpha'\beta'$. There is no homeomorphism between $(\PC^2,\A_2^{\alpha,\beta})$ and $(\PC^2,\A_2^{\alpha',\beta'})$.
\end{thm}

\begin{cor}
	If $\alpha,\beta,\alpha',\beta'\in\set{-1,1}$ are such that $\alpha\beta\neq\alpha'\beta'$ then $\A_2^{\alpha,\beta}$ and $\A_2^{\alpha',\beta'}$ form a Zariski pair.
\end{cor}

\begin{rmk}\label{rmk:deg_geom_prop}~
	\begin{enumerate}
	\item The configurations $\C^2_{\alpha,\beta}$ reveal a geometric distinction of the Zariski pair. We have $\alpha=\beta$ (i.e. the value $\tau(\C_{\alpha,\beta}^2)=0$) if and only if the six points $S_3, S_4, S_5^\alpha, S_7^\alpha, S_8^\beta$ and $S_{10}^\beta$ are contained in a conic.
	
	\item The characterization of the value $\tau(\C_{\alpha,\beta}^2)$ by the concurrence of the lines $(V_3,S_1,S_2)$, $(V_2,S_5,S_6)$ and $(V_1,S_8,S_9)$ showed in Proposition~\ref{propo:pts_align} holds for this degeneration.
	\end{enumerate}
\end{rmk}

We prove in the following section that the above geometrical properties characterize topologically the connected components of the moduli space of the associated dual arrangements.

\bigskip
\section{Moduli spaces and geometrical characterizations}\label{sec:MS}
In this section, we are interested in the moduli space of realizations of the Zariski pair giving by $\A_2^{\alpha,\beta}$, presented in Section~\ref{sec:deg}. The others can be computed in a similar manner.

\subsection{Realizations of $\A_2^{\alpha,\beta}$}\mbox{}

\subsubsection{Computation of the moduli space}

The \emph{moduli space} $\Sigma_\A$ of an arrangement $\A$ (composed of $n$ lines) is defined by:
\begin{equation*}
	\Sigma_\A = \{ \B \in (\PCc^2)^n\, \mid\, \B\sim\A\} / \PGL_3(\CC),
\end{equation*}
where $\B\sim\A$ means that $\A$ and $\B$ have the same combinatorics.

\begin{rmk}
	Giving an ordering on the set of lines of the combinatorics, we can define accordingly the \emph{ordered moduli space}.
\end{rmk}

\begin{thm}\label{thm:moduli}
	The moduli space $\Sigma_2$ of $\A_2^{\alpha,\beta}$ is formed by the arrangements $\M^{\kappa_1,\kappa_2}_\gamma$ composed of the lines
	\[
		L_1: y=0,\quad L_2: x=0,\quad L_3: z=0,\quad 
		L_4: x+y+z=0,	
	\]
	\[
		L_5: \gamma x+\gamma y+z=0,\quad	
		L_6: x + \gamma y+z=0,\quad	
		L_7: \gamma x + y + z=0,	
	\]
	\[
		L_8: \kappa_1^{-1} x + \kappa_1 y + z=0,\quad 
		L_9: x +\kappa_1 y + z=0,\quad	
		L_{10}: \kappa_1^{-1} x + y + z=0,	
	\]
	\[
		L_{11}: \kappa_2 x + \kappa_2^{-1} y +  z=0,\quad 
		L_{12}: \kappa_2 x + y +z=0,\quad	
		L_{13}: x + \kappa_2^{-1} y + z=0.	
	\]
	where $\kappa_1^2=\kappa_2^2=\gamma\in\CC^*$ satisfying $\kappa_1^3,\kappa_2^3\neq 1$, $\kappa_1\kappa_2\neq 1$ and:
	\begin{multicols}{2}
	\begin{enumerate}
		\item if $\kappa_1=\kappa_2$:
		\begin{itemize}
			\item $2\kappa_1^2 + \kappa_1 + 1\neq0$,
			\item $2\kappa_1^2 + 2\kappa_1 + 1\neq0$,
			\item $\kappa_1\neq -1/2$,
			\item $\kappa_1^3 + 3\kappa_1^2 + 2\kappa_1 + 1\neq0$,
			\item $\kappa_1^3 + 2\kappa_1^2 + \kappa_1 + 1\neq0$;
		\end{itemize}
		
		\item if $\kappa_1=-\kappa_2$:
		\begin{itemize}
			\item $\kappa_1^3 + \kappa_1^2 + 1\neq0$,
			\item $\kappa_1^3 - \kappa_1^2 - 1\neq0$,
			\item $2\kappa_1^2 + 1\neq0$,
			\item $\kappa_1^3 + \kappa_1 - 1\neq0$,
			\item $\kappa_1^3 + \kappa_1 + 1\neq0$.
		\end{itemize}
	\end{enumerate}
	\end{multicols}
\end{thm}

\begin{proof}
	Note that we are identifying
	\[
		L_i=V_i^\bullet,\ i=1,2,3\quad \text{and}\quad L_{j+3}=S_j^\bullet,\ j=1,\ldots,10.
	\]
	Using the action of $\PGL_3(\CC)$, we fix the lines $L_1,\dots,L_4$ as in the statement. From the concurrence conditions given by the combinatorics in Proposition~\ref{propo:combi_ZP_2pts5}, we obtain the equations of the lines $L_5,\dots,L_{13}$ and these depend on $\gamma,\kappa_1,\kappa_2\in\CC\setminus\{0,1\}$. The details of these computations are omitted. Note that the relation $\gamma=\kappa_1^2=\kappa_2^2$ comes from the two triple points $\{L_3,L_6,L_8\}$ and $\{L_3,L_7,L_{11}\}$. 
	
	We deduce from the equations of the lines above that we obtain different collapses between lines when $\kappa_1^3=1$, $\kappa_2^3=1$ or $\kappa_1\kappa_2=1$. Finally, we avoid extra concurrences between lines when we remove the values of $\kappa_1$ and $\kappa_2$ listed in the statement, which correspond respectively to the following incidences:
	\begin{enumerate}
		\item if $\kappa_1=\kappa_2$,
		\[
				\{L_5,L_8,L_{11}\},\ \{L_5,L_8,L_{12}\},\ \{L_5,L_9,L_{12}\},
		\]
		\[
		 \{L_5,L_8,L_{13}\}\ \text{and}\ \{L_5,L_{10},L_{11}\},\ \{L_6,L_{10},L_{11}\}\ \text{and}\ \{L_7,L_8,L_{13}\};
		\]
		
		\item if $\kappa_1=-\kappa_2$,
		\[
				\{L_5,L_8,L_{13}\},\ \{L_5,L_{10},L_{11}\},\ \{L_5,L_{10},L_{13}\},\ \{L_6,L_{10},L_{11}\},\ \{L_7,L_8,L_{13}\}.
		\]
	\end{enumerate}
	
\end{proof}

\subsubsection{Realization over a number field}

A Zariski pair is said to be \emph{arithmetic} if the equations of the arrangements are Galois-conjugated in a number field (as the example in~\cite{Guerville:ZP,ACGM:ZP}). Such pairs are particular since they cannot be distinguished by algebraic arguments.

\begin{cor}
	There is no arithmetic Zariski pair with the combinatorics of $\A_2^{\alpha,\beta}$.
\end{cor}

Even if there is no Galois-conjugated realization, we observe a Rybnikov-like construction. Using Theorem~\ref{thm:moduli} for $\kappa_1^2=\kappa_2^2=2$, we can show that there exist $(3,2)$-configurations having the same combinatorics than $\A_2^{\alpha,\beta}$ and defined over the number field $\QQ(\sqrt{2})$. Indeed, taking $\gamma=2$, we obtain the configuration given by
\[
	V_1=(0:1:0),\quad V_2=(1:0:0),\quad V_3=(0:0:1),
\]
\[
	S_1=(1:1:1),\quad S_2=(2:2:1),\quad S_3=(1:2:1),\quad S_4=(2:1:1),
\]
\[
	S_5^\alpha=(1:2:\alpha\sqrt{2}),\quad S_6^\alpha=(1:\alpha\sqrt{2}:1),\quad S_7^\alpha=(\alpha\sqrt{2}:2:2),
\]
\[
	S_8^\beta=(2:1:\beta\sqrt{2}),\quad S_9^\beta=(\beta\sqrt{2}:1:1),\quad S_{10}^\beta=(2:\beta\sqrt{2}:2).
\]

In~\cite{Rybnikov}, Rybnikov constructed the first known example of Zariski pair by gluing in two different ways two MacLane arrangements, which are complex conjugated but not complexified real arrangements. It is worth noticing that $\A_2^{\alpha,\beta}$ are of the same \emph{semi-arithmetic} nature.

Consider $\A^{\alpha}$ and $\B^{\beta}$ the complexified dual arrangements of the points sets $\V\sqcup\S\sqcup\S_\alpha$ and $\V\sqcup\S\sqcup\S_\beta$ respectively. Observe that $s(\A^\alpha)=\B^\alpha$, where $s:\PC^2\rightarrow\PC^2$ is the symmetry with respect to the line $x-y=0$. Then, we obtain that $\A_2^{\alpha,\beta}$ is the a particular gluing of two copy of the $\A^\alpha$ arrangement. More precisely, if $\sigma$ is the non-trivial Galois automorphism of $\QQ(\sqrt{2})$, we have
\begin{equation*}
	\A_2^{\alpha,\beta}=\varphi_\alpha(\A^{1})\cup \psi_\beta(\A^{1}),
\end{equation*}
where $\varphi_\alpha=\sigma^{(1+\alpha)/2}$ and $\psi_\beta=s \circ \sigma^{(1+\beta)/2}$. Thereby, $\sigma$ mimics the complex conjugation of Rybnikov's example, and $s$ the gluing.

\begin{rmk}\mbox{}
	\begin{enumerate}
		\item The combinatorics of $\A^{\alpha,\beta}$ and $\A_1^{\alpha,\beta}$ also admit a similar type of realizations in $\QQ(\sqrt{2})$ and Rybnikov-like constructions.
		\item The construction above is similar to the one given by the first author in~\cite{Guerville:multiplicativity}.
	\end{enumerate}
\end{rmk}

\subsection{Topology of the moduli space}\mbox{}

From Theorem~\ref{thm:moduli}, we can obtain additional results on the topology of the moduli space. In the first place, we compute the number of connected components of the space.

\begin{propo}
	The moduli space $\Sigma_2$ of $\A_2^{\alpha,\beta}$ is formed by two connected components $\Sigma_2^{0}$ and $\Sigma_2^{1}$. The first is characterized by the relation $\kappa_1=\kappa_2$, and the second by $\kappa_1=-\kappa_2$.
\end{propo}

\begin{proof}
	We denote by $\R$ the subspace of $\CC^2$ defined by
	\begin{equation*}
		(\kappa_1,\kappa_2) \in \R \Longleftrightarrow \M^{\kappa_1,\kappa_2}_\gamma \in \Sigma_2. 
	\end{equation*}
	Since $\gamma$ is determined by the values of $\kappa_1$ and $\kappa_2$ then $\R\simeq\Sigma_2$. Using the first conditions described in Theorem~\ref{thm:moduli}, we get $0=\kappa_1^2-\kappa_2^2=(\kappa_1-\kappa_2)(\kappa_1+\kappa_2)$ and $\kappa_1,\kappa_2\neq 0$. In particular,
	\[
		\R\subset (\L_1\cup\L_2)\setminus 0,
	\]
	where $\L_1: \kappa_1-\kappa_2=0$ and $\L_2: \kappa_1+\kappa_2=0$ are two lines in $\CC^2$ intersecting at the origin. Since all the conditions defining $\R$ are algebraic, it has at most two connected components. Indeed, applying the inequality conditions in Theorem~\ref{thm:moduli} at each component contained respectively in $\L_1$ and $\L_2$, we obtain that $\R$ is the disjoint union of two affine spaces $\CC_p$ and $\CC_q$, which are isomorphic to $\CC$ punctured with a finite collection of points.
\end{proof}

The previous result implies that $\A_2^{1,1}$ and $\A_2^{-1,-1}$ (resp.  $\A_2^{-1,1}$ and $\A_2^{1,-1}$) have the same topology. In particular, the dual arrangements which are in $\Sigma_2^0$ correspond to those for which $\tau(\C_{\alpha,\beta}^2)=0$, and they belong to $\Sigma_2^1$ if $\tau(\C_{\alpha,\beta}^2)=1$. We can also remark that these arrangements are path-connected in each component as complex arrangements but not as complexified real arrangements. In general, for any $(3,2)$-configuration $\C$ with the same combinatorics of $\C_{\alpha,\beta}^2$, we have
\[
	\tau(\C)=0\Longleftrightarrow \A^\C\in\Sigma_2^0.
\]

In addition, these two connected components are also geometrically characterized as follows. As in the original Zariski example of two sextics~\cite{zariski}, the Zariski pair formed by $\A_2^{1,1}$ can be distinguished by a conic in two different ways (see Figure~\ref{fig:ZP13_2pts5}). Denote by $P_{i_1,\ldots,i_k}$ the singular point defined by $L_{i_1}, \ldots, L_{i_k}\in\M^{\kappa_1,\kappa_2}_\gamma$.

\begin{propo}\label{propo:geom_char}
	For any $\M^{\kappa_1,\kappa_2}_\gamma\in\Sigma_2$, the following are equivalent:
	\begin{enumerate}
		\item $\M^{\kappa_1,\kappa_2}_\gamma\in\Sigma_2^{0}$.
		\item The six lines $L_6, L_7, L_8, L_{10}, L_{11}, L_{13}$ are tangent to a smooth conic.
		\item The six triple points $P_{1,8,10}$, $P_{1,11,12}$, $P_{2,8,9}$, $P_{2,11,12}$, $P_{3,9,10}$ and $P_{3,12,13}$ are contained in a smooth conic.
		\item The three triple points $P_{1,11,12}$, $P_{2,8,9}$ and $P_{3,4,5}$ are aligned.
	\end{enumerate}
\end{propo}

\begin{rmk}\label{ref:conic_dual}
The second statement in Proposition~\ref{propo:geom_char} can be rephrased in terms of dual configurations. That is, the six points $L_6^*, L_7^*, L_8^*, L_{10}^*, L_{11}^*$ and $L_{13}^*$ belongs to a smooth conic if and only if $\M^{\kappa_1,\kappa_2}_\gamma\in\Sigma_2^{0}$.
\end{rmk}

\begin{proof}
	We prove separately each of the geometrical characterizations of the arrangements in $\Sigma_2^{0}$.
	
	For the first one, in order to simplify the proof, we consider the dual point of view as in Remark~\ref{ref:conic_dual}, and then prove that $L_6^*, L_7^*, L_8^*, L_{10}^*, L_{11}^*$ and $L_{13}^*$ are on the same conic for any arrangement $\M^{\kappa_1,\kappa_2}_\gamma$ in one connected component and not in the other one. The conic defined by the five points $L_6^*, L_7^*, L_8^*, L_{10}^*, L_{11}^*$ is given by
	\begin{equation*}
		\det\left(
			\begin{matrix}
				X^2 & Y^2 & Z^2 & XY & XZ & YZ \\			
				1 & \gamma^2 & 1 & \gamma & 1 & \gamma \\				
				\gamma^2 & 1 & 1 & \gamma & \gamma & 1 \\		
				\kappa_1^{-2} & \kappa_1^2 & 1 & 1 & \kappa_1^{-1} & \kappa_1 \\
				\kappa_1^{-2}	& 1 & 1 & \kappa_1^{-1} & \kappa_1^{-1} & 1 \\
				\kappa_2^2 & \kappa_2^{-2} & 1 & 1 & \kappa_2 & \kappa_2^{-1}						
			\end{matrix}
		\right) = 0.
	\end{equation*}
	Substituting $X$, $Y$ and $Z$ by the coordinates of $L_{13}^*$, we obtain that if $\kappa_1=\kappa_2$ then this determinant vanishes, while if $\kappa_1=-\kappa_2$ then it takes the value
	\begin{equation*}
		\pm 2\kappa_1^{-5}(\kappa_1^2+1)(\kappa_1^2-\kappa_1+1)(\kappa_1^2+\kappa_1+1)(\kappa_1-1)^4(\kappa_1+1)^4,
	\end{equation*}
	which is non-zero if $\M^{\kappa_1,\kappa_2}_\gamma\in\Sigma_2^{0}$.
	
	Using analogous arguments over the six triple points $P_{1,8,10}$, $P_{1,11,12}$, $P_{2,8,9}$, $P_{2,11,12}$, $P_{3,9,10}$ and $P_{3,12,13}$, we obtain the second geometrical characterization.
	
	Finally, computing the coordinates of the three triple points $P_{1,11,12}$, $P_{2,8,9}$ and $P_{3,4,5}$, we obtain
	\[
	P_{1,11,12}=(1:0:-\kappa_2),\quad P_{2,8,9}=(0:1:-\kappa_1),\quad P_{3,4,5}=(1:-1:0).
	\]
	These points are collinear if and only if
	\[
	\det\begin{pmatrix}
	1 & 0 & 1\\
	0 & 1 & -1\\
	-\kappa_2 & -\kappa_1 & 0
	\end{pmatrix}
	=\kappa_1-\kappa_2=0.
	\]
	That is if and only if $\kappa_1=\kappa_2$.
\end{proof}

\begin{rmk}\mbox{}
  \begin{enumerate}
    \item The previous geometrical properties were already observed for the configuration $\C_{\alpha,\beta}^2$ in Remark~\ref{rmk:deg_geom_prop}. Proposition~\ref{propo:geom_char} extends this characterization to the entire connected components.
    
    \item It can be checked that there exist more smooth conics tangent to six lines (resp. containing six points) characterizing the connected components of $\Sigma_2$.
   \end{enumerate}
 \end{rmk}

\begin{figure}[h!]
	\begin{tikzpicture}
  \begin{scope}[scale=0.6,shift={(-6,0)}]
    \def\colorV{black}
    \def\colorS{black}
    \def\colorSa{black}
    \def\colorSb{black}
    
    \tikzset{
        myConic1/.style = {color=cyan, dashed, smooth, samples=200, line width=1.5}
    }
    
    \draw[domain=-1/3:2][myConic1] plot[variable=\x] ({\x}, { (50 - 11*\x + 2*sqrt(190)*sqrt(2 + 5*\x - 3*\x*\x))/49 });
    \draw[domain=-1/3:2][myConic1] plot[variable=\x] ({\x}, { (50 - 11*\x - 2*sqrt(190)*sqrt(2 + 5*\x - 3*\x*\x))/49 });
    
    \tikzset{
        myConic2/.style = {color=orange, dashed, smooth, samples=200, line width=1.5}
    }
    \draw[domain=-1.43:2.098][myConic2] plot[variable=\x] ({\x}, { (1 - \x - sqrt(9 + 2*\x - 3*\x*\x))/2 });
    \draw[domain=-1.43:2.098][myConic2] plot[variable=\x] ({\x}, { (1 - \x + sqrt(9 + 2*\x - 3*\x*\x))/2 });
    
    
    \coordinate (P1) at (0, -3);
    \coordinate (Q1) at (0, 13/2);
    \coordinate (P2) at (-3, 0);
    \coordinate (Q2) at (13/2, 0);
    \coordinate (P3) at (-3, 4);
    \coordinate (Q3) at (4, -3);
    \coordinate (P4) at (-3, 8/3);
    \coordinate (Q4) at (8/3, -3);
    \coordinate (P5) at (-1/3, -3);
    \coordinate (Q5) at (-1/3, 13/2);
    \coordinate (P6) at (-3, -1/3);
    \coordinate (Q6) at (13/2, -1/3);
    \coordinate (P7) at (-5/2, -3);
    \coordinate (Q7) at (9/4, 13/2);
    \coordinate (P8) at (-1, -3);
    \coordinate (Q8) at (-1, 13/2);
    \coordinate (P9) at (-3, 2);
    \coordinate (Q9) at (13/2, 2);
    \coordinate (P10) at (-3, -5/2);
    \coordinate (Q10) at (13/2, 9/4);
    \coordinate (P11) at (-3, -1);
    \coordinate (Q11) at (13/2, -1);
    \coordinate (P12) at (2, -3);
    \coordinate (Q12) at (2, 13/2);
    \draw[color=\colorV, line width=1] (P1)--(Q1) node[pos=-.05, xshift=.2cm] {$L_1$};
    \draw[color=\colorV, line width=1] (P2)--(Q2) node[pos=-.05, yshift=.15cm] {$L_2$};
    \draw[color=\colorV, line width=1] (P3)--(Q3) node[pos=-.05] {$L_3$};
    \draw[color=\colorS, line width=1] (P4)--(Q4) node[pos=-.075] {$L_5$};
    \draw[color=\colorS, line width=1] (P5)--(Q5) node[pos=-.05, xshift=-.05cm] {$L_6$};
    \draw[color=\colorS, line width=1] (P6)--(Q6) node[pos=-.05, yshift=-.05cm] {$L_7$};
    \draw[color=\colorSa, line width=1] (P7)--(Q7) node[pos=-.05] {$L_8$};
    \draw[color=\colorSa, line width=1] (P8)--(Q8) node[pos=-.05, xshift=-.1cm] {$L_9$};
    \draw[color=\colorSa, line width=1] (P9)--(Q9) node[pos=-.05]
    {$L_{10}$};
    \draw[color=\colorSb, line width=1] (P10)--(Q10) node[pos=-.05]
    {$L_{11}$};
    \draw[color=\colorSb, line width=1] (P11)--(Q11) node[pos=-.05, yshift=-.05cm]
    {$L_{12}$};
    \draw[color=\colorSb, line width=1] (P12)--(Q12) node[pos=-.05]
    {$L_{13}$};
    
    \node at (1,-5) {(A) A representative of $\Sigma_2^{0}$ and the conics.}; 
    
  \end{scope}

  \begin{scope}[scale=0.95,shift={(6,0)}]
  	\def\colorV{black}
    \def\colorS{black}
    \def\colorSa{black}
    \def\colorSb{black}
    
    \coordinate (P1) at (0, -2);
    \coordinate (Q1) at (0, 4);
    \coordinate (P2) at (-4, 0);
    \coordinate (Q2) at (2, 0);
    \coordinate (P3) at (2, -1);
    \coordinate (Q3) at (-3, 4);
    \coordinate (P4) at (-4, 11/3);
    \coordinate (Q4) at (5/3, -2);
    \coordinate (P5) at (-1/3, -2);
    \coordinate (Q5) at (-1/3, 4);
    \coordinate (P6) at (-4, -1/3);
    \coordinate (Q6) at (2, -1/3);
    \coordinate (P7) at (-2, -2);
    \coordinate (Q7) at (1, 4);
    \coordinate (P8) at (-1, -2);
    \coordinate (Q8) at (-1, 4);
    \coordinate (P9) at (-4, 2);
    \coordinate (Q9) at (2, 2);
    \coordinate (P10) at (-4, 7/3);
    \coordinate (Q10) at (2, -2/3);
    \coordinate (P11) at (-4, 1/3);
    \coordinate (Q11) at (2, 1/3);
    \coordinate (P12) at (2/3, -2);
    \coordinate (Q12) at (2/3, 4);
    \draw[color=\colorV, line width=1] (P1)--(Q1) node[pos=-.05, xshift=.15cm] {$L_1$};
    \draw[color=\colorV, line width=1] (P2)--(Q2) node[pos=-.05] {$L_2$};
    \draw[color=\colorV, line width=1] (P3)--(Q3) node[pos=-.05] {$L_3$};
    \draw[color=\colorS, line width=1] (P4)--(Q4) node[pos=-.05] {$L_5$};
    \draw[color=\colorS, line width=1] (P5)--(Q5) node[pos=-.05, xshift=-0.05cm] {$L_6$};
    \draw[color=\colorS, line width=1] (P6)--(Q6) node[pos=-.05, yshift=-.05cm] {$L_7$};
    \draw[color=\colorSa, line width=1] (P7)--(Q7) node[pos=-.05] {$L_8$};
    \draw[color=\colorSa, line width=1] (P8)--(Q8) node[pos=-.05] {$L_9$};
    \draw[color=\colorSa, line width=1] (P9)--(Q9) node[pos=-.05]
    {$L_{10}$};
    \draw[color=\colorSb, line width=1] (P10)--(Q10) node[pos=-.05]
    {$L_{11}$};
    \draw[color=\colorSb, line width=1] (P11)--(Q11) node[pos=-.05, yshift=.05cm]
    {$L_{12}$};
    \draw[color=\colorSb, line width=1] (P12)--(Q12) node[pos=-.05]
    {$L_{13}$};
    
    \node at (-1,-3.15) {(B) A representative of $\Sigma_2^{1}$.};
    
  \end{scope}
	\end{tikzpicture}
	\caption{Real representatives of the two connected components of $\Sigma_2$, considering $L_4$ as the line at the infinity.}\label{fig:ZP13_2pts5}
\end{figure}
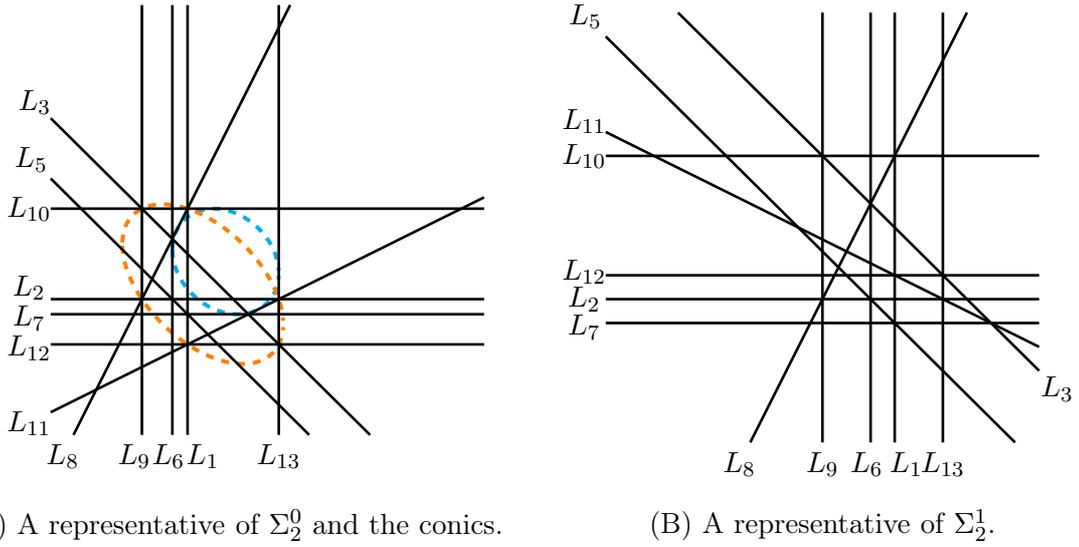

%
%
%

\bigskip
\section{Other examples of Zariski pairs}\label{sec:list}
Applying a similar construction and arguments as in Sections~\ref{sec:ZP} and~\ref{sec:deg}, we construct examples and degenerations of Zariski pairs with 15 and 17 lines using $(3,2)$-configurations. Any of these examples admits realizations by using equations with rational coefficients. However, in order to present coherent equations between the initial new pairs and their degenerations (together with simpler equations), the following examples are given over the number field $\QQ(\sqrt{2})$.

\subsection{Examples with 15 lines}\label{subsec:15lines}\mbox{}

Consider the planar $(3,2)$-configurations $\D_{\alpha,\beta}=(\V,\S_{\alpha,\beta},\L_{\alpha,\beta},\pl)$, for $\alpha,\beta\in\{-1,1\}$, defined by $\V=\set{V_1,V_2,V_3}$ with
\begin{equation*}
	\begin{array}{l p{0.5cm} l p{0.5cm} l}
		V_1=(0:0:1), && V_2=(1:0:0), && V_3=(0:1:0),
	\end{array}
\end{equation*}
and $\S_{\alpha,\beta}=\set{S_1,S_2,S_3,S_4,S_5,S_6}\sqcup\set{S_7^\alpha,S_8^\alpha,S_9^\alpha}\sqcup\set{S_{10}^\beta,S_{11}^\beta,S_{12}^\beta}$ where
\begin{equation*}
	\begin{array}{l p{0.5cm} l p{0.5cm} l}
		S_1=(1:1:1), && S_2=(4:1:1), && S_3=(4:1:4), \\
		S_4=(2:2:1), && S_5=(4:-2:1), && S_6=(-1:1:4),\\
		S_7^\alpha=(4:-2:\alpha\sqrt{2}), && S_8^\alpha=(2\alpha\sqrt{2}:-2:1), && S_9^\alpha=(2:-\alpha\sqrt{2}:1),\\
		S_{10}^\beta(-1:1:\beta\sqrt{2}), &&	S_{11}^\beta(-1:2\beta\sqrt{2}:4), &&	S_{12}^\beta(-\beta\sqrt{2}:4:2).
	\end{array}
\end{equation*}

The configurations $\D_{1,1}$ and $\D_{1,-1}$ are as in Figure~\ref{fig:dual_ZP15}. Ordered from left to right and from top to bottom, the surrounding points are given by
\begin{itemize}
	\item $S_{12},\ S_4,\ S_1,\ S_2,\ S_{10},\ S_{11},\ S_6,\ S_3,\ S_9,\ S_7,\ S_8,\ S_5 $, for $\D_{1,1}$, (see~Figure~\ref{fig:dual_ZP15}--(A));
	\item $S_{12},\ S_4,\ S_1,\ S_2,\ S_6,\ S_3,\ S_{11},\ S_{10},\ S_9,\ S_7,\ S_8,\ S_5$, for $\D_{1,-1}$, (see~Figure~\ref{fig:dual_ZP15}--(B)).
\end{itemize}

\pagebreak
\begin{thm}\label{thm:ZP15}
	Let $\M^{\alpha,\beta}$ be the dual arrangement of the configuration $\D_{\alpha,\beta}$. For any $\alpha,\beta\in\set{-1,1}$, we have:
	\begin{enumerate}
		\item The combinatorics of the configuration $\D_{\alpha,\beta}$ is given by
			\begin{equation*}			
				\begin{array}{l}
				\Big\{\ 
					\big\{ V_1, S_1, S_4\big\},\ \big\{ V_1, S_2, S_3\big\},\ \big\{V_1, S_5,S_7^\alpha\big\},\ \big\{V_1,S_6, S_{10}^\beta\big\},\ \big\{V_1, S_8^\alpha, S_9^\beta\big\},\ \big\{V_1, S_{11}^\beta, S_{12}^\beta\big\},   \\[0.5em]
					\big\{V_2, S_1, S_2\big\},\ \big\{V_2, S_3, S_6\big\},\ \big\{V_2, S_4, S_{12}^\beta\big\},\ \big\{V_2, S_5, S_8^\alpha\big\},\ \big\{V_2, S_7^\alpha, S_9^\alpha\big\},\ \big\{V_2, S_{10}^\beta, S_{11}^\beta\big\}, \\[0.5em]
					\big\{V_3, S_1, S_3\big\},\ \big\{V_3, S_2, S_5\big\},\ \big\{V_3, S_4, S_9^\alpha\big\},\ \big\{V_3, S_6, S_{11}^\beta\big\},\ \big\{V_3, S_7^\alpha, S_8^\alpha\big\},\ \big\{V_3, S_{10}^\beta, S_{12}^\beta\big\}\ \Big\}.
				\end{array}
			\end{equation*}
			
		\item $\D_{\alpha,\beta}$ is stable.
		
		\item If $\alpha',\beta'\in\set{-1,1}$ are such that $\alpha\beta\neq\alpha'\beta'$, then there is no homeomorphism between $(\PC^2,\M^{\alpha,\beta})$ and $(\PC^2,\M^{\alpha',\beta'})$.
	\end{enumerate}
\end{thm}

\begin{cor}\label{cor:ZP15}
	Let $\alpha,\beta,\alpha',\beta'\in\set{-1,1}$ be such that $\alpha\beta\neq\alpha'\beta'$. Then the arrangements $\M^{\alpha,\beta}$ and $\M^{\alpha',\beta'}$ form a Zariski pair.
\end{cor}

\begin{rmk}
	Following~\cite{PapadimaSuciu:doubles_triples} (as well as~\cite{Libgober:doubles_triples} and~\cite{Dimca:Monodromy}), we can conclude that the Betti numbers of the Milnor fiber as well as the eigenvalues and the characteristic polynomial of the monodromy do not differ  in the arrangements above (thus, they are combinatorially determined), since they contain only double and triple points. 
\end{rmk}

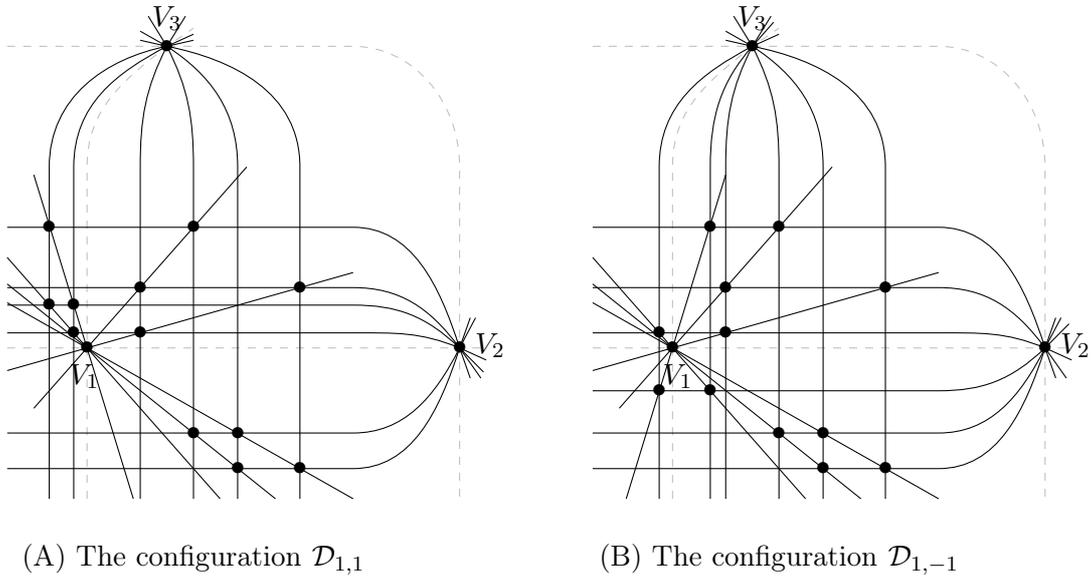
\begin{figure}[h!]
	\begin{tikzpicture}
	\begin{scope}[xscale=0.7,yscale=0.8]
		\begin{scope}[shift={(-5.5,0)}]
			\draw[dashed, color=gray!50] (0,-2.5) -- (0,3) to[out=90,in=-145] (1.5,5) -- (2,5.3);
			\draw[dashed, color=gray!50] (-1.5,0) -- (7.5,0);
			\draw[dashed, color=gray!50] (-1.5,5) -- (5,5) to[out=0,in=90] (7,3) -- (7,-2.5);
			
			\draw (1,-2.5) -- (1,3) to[out=90,in=-120] (1.5,5) -- (1.85,5.5);
			\draw (2,-2.5) -- (2,3) to[out=90,in=-60] (1.5,5) -- (1.15,5.5);
			\draw (4,-2.5) -- (4,3) to[out=90,in=-10] (1.5,5) -- (1,5.1);
			\draw (-0.25,-2.5) -- (-0.25,3) to[out=90,in=-155] (1.5,5) -- (2,5.2);	
			\draw (2.83,-2.5) -- (2.83,3) to[out=90,in=-30] (1.5,5) -- (1,5.25);	
			\draw (-0.71,-2.5) -- (-0.71,3) to[out=90,in=-170] (1.5,5) -- (2,5.1);
			
			\draw (-1.5,0.25) -- (5,0.25) to[out=0,in=160] (7,0) -- (7.5,-0.2);
			\draw (-1.5,1) -- (5,1) to[out=0,in=130] (7,0) -- (7.4,-0.5);
			\draw (-1.5,2) -- (5,2) to[out=0,in=110] (7,0) -- (7.2,-0.5);
			\draw (-1.5,-2) -- (5,-2) to[out=0,in=250] (7,0) -- (7.2,0.5);
			\draw (-1.5,-1.41) -- (5,-1.41) to[out=0,in=240] (7,0) -- (7.3,0.5);
			\draw (-1.5,0.71) -- (5,0.71) to[out=0,in=140] (7,0) -- (7.45,-0.4);
				
			
			\draw(-1,-1) -- (3,3);
			\draw(-1.5,-0.375) -- (5,1.25);
			\draw(-1.5,1.5) -- (2.5,-2.5);
			\draw(-1.5,0.75) -- (5,-2.5);
			\draw(-1.5,1.06) -- (3.54,-2.5);
			\draw(-1,2.87) -- (0.87,-2.5);
			
		
			\node at (0,0) {$\bullet$};
			\node[below] at (-0.05,-0.1) {$V_1$};
			\node at (7,0) {$\bullet$};
			\node[right] at (7.1,0.05) {$V_2$};
			\node at (1.5,5) {$\bullet$};
			\node[above] at (1.5,5.1) {$V_3$};
			
			\node at (1,1) {$\bullet$}; 
			\node at (4,1) {$\bullet$}; 
			\node at (1,0.25) {$\bullet$}; 
			\node at (2,2) {$\bullet$}; 
			\node at (4,-2) {$\bullet$}; 
			\node at (2.83,-1.41) {$\bullet$}; 
			\node at (2.83,-2) {$\bullet$}; 
			\node at (2,-1.41) {$\bullet$}; 
			
			\node at (-0.25,0.25) {$\bullet$}; 
			\node at (-0.71,0.71) {$\bullet$}; 
			\node at (-0.25,0.71) {$\bullet$}; 
			\node at (-0.71,2) {$\bullet$}; 
	\node at (2,-3.5) {(A) The configuration $\D_{1,1}$};
		\end{scope}
		
		\begin{scope}[shift={(5.5,0)}]
			\draw[dashed, color=gray!50] (0,-2.5) -- (0,3) to[out=90,in=-145] (1.5,5) -- (2,5.3);
			\draw[dashed, color=gray!50] (-1.5,0) -- (7.5,0);
			\draw[dashed, color=gray!50] (-1.5,5) -- (5,5) to[out=0,in=90] (7,3) -- (7,-2.5);
			
			\draw (1,-2.5) -- (1,3) to[out=90,in=-120] (1.5,5) -- (1.85,5.5);
			\draw (2,-2.5) -- (2,3) to[out=90,in=-60] (1.5,5) -- (1.15,5.5);
			\draw (4,-2.5) -- (4,3) to[out=90,in=-10] (1.5,5) -- (1,5.1);
			\draw (-0.25,-2.5) -- (-0.25,3) to[out=90,in=-155] (1.5,5) -- (2,5.2);	
			\draw (2.83,-2.5) -- (2.83,3) to[out=90,in=-30] (1.5,5) -- (1,5.25);	
			\draw  (0.71,-2.5) -- (0.71,3) to[out=90,in=-135] (1.5,5) -- (1.9,5.4);
			
			\draw (-1.5,0.25) -- (5,0.25) to[out=0,in=160] (7,0) -- (7.5,-0.2);
			\draw (-1.5,1) -- (5,1) to[out=0,in=130] (7,0) -- (7.4,-0.5);
			\draw (-1.5,2) -- (5,2) to[out=0,in=110] (7,0) -- (7.2,-0.5);
			\draw (-1.5,-2) -- (5,-2) to[out=0,in=250] (7,0) -- (7.2,0.5);
			\draw (-1.5,-1.41) -- (5,-1.41) to[out=0,in=240] (7,0) -- (7.3,0.5);
			\draw (-1.5,-0.71) -- (5,-0.71) to[out=0,in=220] (7,0) -- (7.45,0.4);
				
			
			\draw(-1,-1) -- (3,3);
			\draw(-1.5,-0.375) -- (5,1.25);
			\draw(-1.5,1.5) -- (2.5,-2.5);
			\draw(-1.5,0.75) -- (5,-2.5);
			\draw(-1.5,1.06) -- (3.54,-2.5);
			\draw(1,2.87) -- (-0.87,-2.5);
			
		
			\node at (0,0) {$\bullet$};
			\node[below] at (0.1,-0.1) {$V_1$};
			\node at (7,0) {$\bullet$};
			\node[right] at (7.1,0.05) {$V_2$};
			\node at (1.5,5) {$\bullet$};
			\node[above] at (1.5,5.1) {$V_3$};
			
			\node at (1,1) {$\bullet$}; 
			\node at (4,1) {$\bullet$}; 
			\node at (1,0.25) {$\bullet$}; 
			\node at (2,2) {$\bullet$}; 
			\node at (4,-2) {$\bullet$}; 
			\node at (2.83,-1.41) {$\bullet$}; 
			\node at (2.83,-2) {$\bullet$}; 
			\node at (2,-1.41) {$\bullet$}; 
			
			\node at (-0.25,0.25) {$\bullet$}; 
			\node at (0.71,-0.71) {$\bullet$}; 
			\node at (-0.25,-0.71) {$\bullet$}; 
			\node at (0.71,2) {$\bullet$}; 
	\node at (2,-3.5) {(B) The configuration $\D_{1,-1}$};
		\end{scope}
	\end{scope}
\end{tikzpicture}

	\caption{The $(3,2)$-configurations $\D_{1,1}$ and $\D_{1,-1}$ with chamber weight 0 and 1, respectively.}\label{fig:dual_ZP15}
\end{figure}

We can also construct degenerations of the pairs $(\M^{\alpha,\beta},\M^{\alpha',\beta'})$ which form Zariski pairs too, as in Section~\ref{sec:deg}. We give two different pairs: the first one, denoted by $\D^1_{\alpha,\beta}$, admits one point of multiplicity~5 and the second one, denoted by $\D^2_{\alpha,\beta}$, admits two points of multiplicity~5. These two constructions are summarized in the Table~\ref{tab:ZP15_1pt5} and Table~\ref{tab:ZP15_2pts5}.

\renewcommand{\arraystretch}{1.5}
\begin{table}[h!]
  \centering
	\begin{tabular}{|l|l l|}
		\hline
		\multicolumn{3}{|c|}{ Configurations $\D^1_{\alpha,\beta}$}\\[2pt]
		\hline
		Created alignment & \multicolumn{2}{c|}{ $\set{V_2,S_4, S_5, S_8^\alpha, S_{12}^\beta}$ } \\[4pt]
		\hline
		\multirow{2}*{Modified points} 
					& $S_5 = (4:2:1)$ 										& $S_7^\alpha = (4:2:\alpha\sqrt{2})$ \\
					& $S_8^\alpha = (2\alpha\sqrt{2}:2:1)$ 		& $S_9^\alpha = (2:\alpha\sqrt{2}:1)$  \\[4pt]
		\hline
		Chamber weight & \multicolumn{2}{c|}{$\tau(\D^1_{\alpha,\beta})=0\Longleftrightarrow \alpha\beta=1$}\\[2pt]
		\hline
	\end{tabular}
	\vspace{0.25cm}
	\caption{\label{tab:ZP15_1pt5} Description of the configurations $\D^1_{\alpha,\beta}$.}
\end{table}

\begin{table}[h!]
  \centering
	\begin{tabular}{|l|l l l|}
		\hline
		\multicolumn{4}{|c|}{ Configurations $\D^2_{\alpha,\beta}$}\\[2pt]
		\hline
		Created alignments & \multicolumn{3}{c|}{ $\set{V_2,S_4, S_5, S_8^\alpha, S_{12}^\beta}$ and $\set{V_3,S_4, S_6, S_9^\alpha, S_{11}^\beta}$ } \\[4pt]
		\hline
		\multirow{3}*{Modified points} 
				& $S_5 = (4:2:1)$ & $S_7^\alpha = (4:2:\alpha\sqrt{2})$ & $S_8^\alpha = (2\alpha\sqrt{2}:2:1)$\\
				& $S_9^\alpha = (2:\alpha\sqrt{2}:1)$ & $S_6 = (8:1:4)$ & $S_{10}^\beta = (8:1:\beta\sqrt{2})$  \\
				& \multicolumn{3}{c|}{ $S_{11}^\beta = (4:\beta\sqrt{2}:2) \quad\quad S_{12}^\beta = (4\beta\sqrt{2}:2:1)$ }\\[4pt]
		\hline
		Chamber weight & \multicolumn{3}{c|}{$\tau(\D^2_{\alpha,\beta})=0\Longleftrightarrow \alpha\beta=1$}\\[2pt]
		\hline
	\end{tabular}
	\vspace{0.25cm}
	\caption{\label{tab:ZP15_2pts5} Description of the configurations $\D^2_{\alpha,\beta}$.}
\end{table}
\vspace{0.5cm}

\subsection{Examples with 17 lines}\label{subsec:17lines}\mbox{}

We create Zariski pairs with 17 lines based on a family of arrangements composed of 9 arrangements (determined by 3 parameters $\alpha,\beta,\gamma\in\{-1,1\}$) and denoted by $\N^{\alpha,\beta,\gamma}$. More precisely, they are dual arrangements of the $(3,2)$-configurations $\E_{\alpha,\beta,\gamma}=(\V,\S_{\alpha,\beta,\gamma},\L_{\alpha,\beta,\gamma},\pl)$ defined by $\V=\{V_1,V_2,V_3\}$ with $V_1=(0:0:1)$, $V_2=(1:0:0)$, $V_3=(0:1:0)$, and $\S_{\alpha,\beta,\gamma}=\set{S_1,S_2,S_3,S_4,S_5}\sqcup\set{S_6^\alpha,S_7^\alpha,S_8^\alpha}\sqcup\set{S_9^\beta,S_{10}^\beta,S_{11}^\beta}\sqcup\set{S_{12}^\gamma,S_{13}^\gamma,S_{14}^\gamma}$ where
\begin{equation*}
	\begin{array}{l p{0.5cm} l p{0.5cm} l}
		\multicolumn{5}{c}{S_1=(1:2:1), \quad\quad\quad\quad S_2=(2:1:1)} \\
		S_3=(2:-4:1), && S_4=(-4:2:1), && S_5=(5:10:-2),\\
		S_6^\alpha=(2:-4:\alpha\sqrt{2}), && S_7^\alpha=(\alpha\sqrt{2}:-4:1), && S_8^\alpha=(1:-2\alpha\sqrt{2}:1),\\
		S_9^\beta=(-4:2:\beta\sqrt{2}), && S_{10}^\beta=(-4:\beta\sqrt{2}:1), && S_{11}^\beta=(-2\beta\sqrt{2}:1:1),\\
		S_{12}^\gamma=(5:5\gamma:-2), &&	S_{13}^\gamma=(5\gamma:5:-1), &&	S_{14}^\gamma=(10:5:-2\gamma).
	\end{array}
\end{equation*}

\pagebreak
\begin{thm}
	Let $\N^{\alpha,\beta,\gamma}$ be the dual arrangements of the configuration $\E_{\alpha,\beta,\gamma}$.  For any $\alpha,\beta,\gamma\in\set{-1,1}$, we have:
	\begin{enumerate}
		\item The configuration $\E_{\alpha,\beta,\gamma}$ has the following combinatorics
			\begin{equation*}
				\qquad\begin{array}{l}
					\Big\{ \big\{V_1,S_1,S_5\big\}, \big\{V_1,S_2,S_{14}^\gamma\big\}, \big\{V_1,S_3,S_6^\alpha\big\}, \big\{V_1,S_4,S_9^\beta\big\}, \big\{V_1,S_7^\alpha,S_8^\alpha\big\},\\[0.5em] \big\{V_1,S_{10}^\beta,S_{11}^\beta\big\}, \big\{V_1,S_{12}^\gamma,S_{13}^\gamma\big\},
					\big\{V_2,S_1,S_4\big\}, \big\{V_2,S_2,S_{11}^\beta\big\}, \big\{V_2,S_3,S_7^\alpha\big\},\\[0.5em] \big\{V_2,S_5,S_{13}^\gamma\big\}, \big\{V_2,S_6^\alpha,S_8^\alpha\big\}, \big\{V_2,S_9^\beta,S_{10}^\beta\big\}, \big\{V_2,S_{12}^\gamma,S_{14}^\gamma\big\},
					\big\{V_3,S_1,S_8^\alpha\big\},\\[0.5em]
					\big\{V_3,S_2,S_3\big\}, \big\{V_3,S_4,S_{10}^\beta\big\}, \big\{V_3,S_5,S_{12}^\gamma\big\}, \big\{V_3,S_6^\alpha,S_7^\alpha\big\}, \big\{V_3,S_9^\beta,S_{11}^\beta\big\}, \big\{V_3,S_{13}^\gamma,S_{14}^\gamma\big\} \Big\}
				\end{array}
			\end{equation*}
		
		\item $\E_{\alpha,\beta,\gamma}$ is stable.
		
		\item If $\alpha',\beta',\gamma'\in\set{-1,1}$ are such that $\alpha\beta\gamma\neq\alpha'\beta'\gamma'$, then there is no homeomorphism between $(\PC^2,\N^{\alpha,\beta,\gamma})$ and $(\PC^2,\N^{\alpha',\beta',\gamma'})$.
	\end{enumerate}
\end{thm}

As example, the configurations $\E_{1,1,1}$ and $\E_{-1,-1,-1}$ are depicted in Figure~\ref{fig:dual_ZP17}. From left to right and from top to bottom, the surrounding points are given by:
\begin{itemize}
	\item $S_4, S_1, S_{10}^\beta, S_9^\beta, S_{11}^\beta, S_2, S_{14}^\gamma, S_{12}^\gamma, S_8^\alpha, S_6^\alpha, S_7^\alpha, S_3, S_{13}^\gamma,S_5$ in $\E_{1,1,1}$, (see~Figure~\ref{fig:dual_ZP17}--(A));
	\item $S_6^\alpha, S_8^\alpha, S_{12}^\gamma, S_{14}^\gamma, S_4, S_1, S_2, S_{11}^\beta, S_{10}^\beta, S_9^\beta, S_7^\alpha, S_3, S_5, S_{13}^\gamma$ in $\E_{-1,-1,-1}$, (see~Figure~\ref{fig:dual_ZP17}--(B)).
\end{itemize}

\begin{figure}[h!]
	\begin{tikzpicture}
 \begin{scope}[shift={(-3.5,0)},scale=0.45]
	\pgfmathsetmacro{\sr}{1.41}
	\pgfmathsetmacro{\a}{1}
	\pgfmathsetmacro{\b}{1}
	\pgfmathsetmacro{\c}{1}
	
	\draw[dashed, color=gray!50] (0,-6) -- (0,4) to[out=90,in=-60] (-1,7) -- (-1.5,7.7);
	\draw (-4,-6) -- (-4,4) to[out=90,in=-160] (-1,7) -- (-0.25,7.2);
	\draw (-2.5,-6) -- (-2.5,4) to[out=90,in=-130] (-1,7) -- (-0.4,7.65);
	\draw (1,-6) -- (1,4) to[out=90,in=-35] (-1,7) -- (-1.6,7.4);
	\draw (2,-6) -- (2,4) to[out=90,in=-15] (-1,7) -- (-1.7,7.2);
	\draw (\a*\sr,-6) -- (\a*\sr5,4) to[out=90,in=-25] (-1,7) -- (-1.65,7.3);
	\draw (-2*\b*\sr,-6) -- (-2*\b*\sr,4) to[out=90,in=-140] (-1,7) -- (-0.3,7.5);
	\draw (-\c*5,-6) -- (-\c*5,4) to[out=90,in=-175] (-1,7) -- (-0.2,7.05);
	
	\draw[dashed, color=gray!50] (-5.5-\c/2,0) -- (5-\c,0) to[out=0,in=150] (8-\c,-1) -- (8-\c+0.5,-1.5);
	\draw (-5.5-\c/2,-5) -- (5-\c,-5) to[out=0,in=-100] (8-\c,-1) -- (8-\c+0.07,-0.5);
	\draw (-5.5-\c/2,-4) -- (5-\c,-4) to[out=0,in=-110] (8-\c,-1) -- (8-\c+0.15,-0.5);
	\draw (-5.5-\c/2,-\c*2.5) -- (5-\c,-\c*2.5) to[out=0,in=170+\c*60] (8-\c,-1) -- (8.25-\c,\c/2-1);
	\draw (-5.5-\c/2,-2*\a*\sr) -- (5-\c,-2*\a*\sr) to[out=0,in=-120] (8-\c,-1) -- (8.3-\c,-0.5);
	\draw (-5.5-\c/2,1) -- (5-\c,1) to[out=0,in=120] (8-\c,-1) -- (8.25-\c,-1.5);
	\draw (-5.5-\c/2,2) -- (5-\c,2) to[out=0,in=100] (8-\c,-1) -- (8.1-\c,-1.5);
	\draw (-5.5-\c/2,\b*\sr) -- (5-\c,\b*\sr) to[out=0,in=110] (8-\c,-1) -- (8.15-\c,-1.5);
	
	\draw[dashed, color=gray!50] (-5,7) -- (5,7) to[out=0,in=90] (7,5) --(7,-6);
	\draw (-3,-6) -- (2,4);
	\draw (-5.5-\c/2,-5.5/2-\c/4) -- (5-\c,5/2-\c/2);
	\draw (-5-\c,-5*\c-1) -- (5-\c,5*\c-1);
	\draw (-4.5,2.25) -- (4,-2);
	\draw (3,-6) -- (-2,4);
	\draw (3*\a*\sr/2,-6) -- (-\a*\sr,4);
	\draw  (-5,5*\b*\sr/4) --  (4,-\b*\sr);
	
	\node at (0,0) {$\bullet$}; 
		\node[left] at (-0.5,-0.05) {$V_1$}; 
		\node[right] at (8.2-\c,-1) {$V_2$}; 
		\node[above] at (-1,7.2) {$V_3$}; 
	\node at (1,2) {$\bullet$}; 
	\node at (2,1) {$\bullet$}; 
	\node at (2,-4) {$\bullet$}; 
	\node at (-4,2) {$\bullet$}; 
	\node at (-2.5,-5) {$\bullet$}; 
	\node at (8-\c,-1) {$\bullet$};
	\node at (-1,7) {$\bullet$};
	\node at (\a*\sr, -2*\a*\sr) {$\bullet$}; 
	\node at (\a*\sr,-4) {$\bullet$}; 
	\node at (1,-2*\a*\sr) {$\bullet$}; 
	\node at (-2*\b*\sr,\b*\sr) {$\bullet$}; 
	\node at (-4,\b*\sr) {$\bullet$}; 
	\node at (-2*\b*\sr,1) {$\bullet$}; 
	\node at (-2.5,-\c*2.5) {$\bullet$}; 
	\node at (-\c*5,-5) {$\bullet$}; 
	\node at (-\c*5,-\c*2.5) {$\bullet$}; 
	
	\node at (0,-7) {(A) The configuration $\E_{1,1,1}$};
	
 \end{scope}

 \begin{scope}[shift={(3.5,0)},scale=0.45]
	\pgfmathsetmacro{\sr}{1.41}
	\pgfmathsetmacro{\a}{-1}
	\pgfmathsetmacro{\b}{-1}
	\pgfmathsetmacro{\c}{-1}
	
	\draw[dashed, color=gray!50] (0,-6) -- (0,4) to[out=90,in=-60] (-1,7) -- (-1.5,7.7);
	\draw (-4,-6) -- (-4,4) to[out=90,in=-160] (-1,7) -- (-0.25,7.2);
	\draw (-2.5,-6) -- (-2.5,4) to[out=90,in=-130] (-1,7) -- (-0.4,7.65);
	\draw (1,-6) -- (1,4) to[out=90,in=-35] (-1,7) -- (-1.6,7.4);
	\draw (2,-6) -- (2,4) to[out=90,in=-15] (-1,7) -- (-1.7,7.2);
	\draw (\a*\sr,-6) -- (\a*\sr5,4) to[out=90,in=-110] (-1,7) -- (-0.8,7.6);
	\draw (-2*\b*\sr,-6) -- (-2*\b*\sr,4) to[out=90,in=-10] (-1,7) -- (-1.7,7.1);
	\draw (-\c*5,-6) -- (-\c*5,4) to[out=90,in=-05] (-1,7) -- (-1.7,7.05);
	
	\draw[dashed, color=gray!50] (-5.5-\c/2,0) -- (5-\c,0) to[out=0,in=150] (8-\c,-1) -- (8-\c+0.5,-1.5);
	\draw (-5.5-\c/2,-5) -- (5-\c,-5) to[out=0,in=-100] (8-\c,-1) -- (8-\c+0.07,-0.5);
	\draw (-5.5-\c/2,-4) -- (5-\c,-4) to[out=0,in=-110] (8-\c,-1) -- (8-\c+0.15,-0.5);
	\draw (-5.5-\c/2,-\c*2.5) -- (5-\c,-\c*2.5) to[out=0,in=100] (8-\c,-1) -- (8.175-\c,-1.5);
	\draw (-5.5-\c/2,-2*\a*\sr) -- (5-\c,-2*\a*\sr) to[out=0,in=95] (8-\c,-1) -- (8.05-\c,-1.5);
	\draw (-5.5-\c/2,1) -- (5-\c,1) to[out=0,in=120] (8-\c,-1) -- (8.25-\c,-1.5);
	\draw (-5.5-\c/2,2) -- (5-\c,2) to[out=0,in=110] (8-\c,-1) -- (8.1-\c,-1.5);
	\draw (-5.5-\c/2,\b*\sr) -- (5-\c,\b*\sr) to[out=0,in=-155] (8-\c,-1) -- (8.4-\c,-0.75);
	
	\draw[dashed, color=gray!50] (-5,7) -- (7,7) to[out=0,in=90] (9,5) --(9,-6);
	\draw (-3,-6) -- (2,4);
	\draw (-5.5-\c/2,-5.5/2-\c/4) -- (5-\c,5/2-\c/2);
	\draw (-5-\c,-5*\c-1) -- (5-\c,5*\c-1);
	\draw (-4.5,2.25) -- (4,-2);
	\draw (3,-6) -- (-2,4);
	\draw (3*\a*\sr/2,-6) -- (-\a*\sr,4);
	\draw  (-5,5*\b*\sr/4) --  (4,-\b*\sr);
	
	\node at (0,0) {$\bullet$}; 
	\node at (0,0) {$\bullet$}; 
		\node[left] at (-0.35,0.05) {$V_1$}; 
		\node[right] at (8.2-\c,-1) {$V_2$}; 
		\node[above] at (-1,7.2) {$V_3$}; 
	\node at (1,2) {$\bullet$}; 
	\node at (2,1) {$\bullet$}; 
	\node at (2,-4) {$\bullet$}; 
	\node at (-4,2) {$\bullet$}; 
	\node at (-2.5,-5) {$\bullet$}; 
	\node at (8-\c,-1) {$\bullet$};
	\node at (-1,7) {$\bullet$};
	\node at (\a*\sr, -2*\a*\sr) {$\bullet$}; 
	\node at (\a*\sr,-4) {$\bullet$}; 
	\node at (1,-2*\a*\sr) {$\bullet$}; 
	\node at (-2*\b*\sr,\b*\sr) {$\bullet$}; 
	\node at (-4,\b*\sr) {$\bullet$}; 
	\node at (-2*\b*\sr,1) {$\bullet$}; 
	\node at (-2.5,-\c*2.5) {$\bullet$}; 
	\node at (-\c*5,-5) {$\bullet$}; 
	\node at (-\c*5,-\c*2.5) {$\bullet$}; 
	
	\node at (2,-7) {(B) The configuration $\E_{-1,-1,-1}$};
 \end{scope}
\end{tikzpicture}
	\caption{The $(3,2)$-configurations $\E_{1,1,1}$ and $\E_{-1,-1,-1}$ with chamber weight 0 and 1, respectively.}\label{fig:dual_ZP17}
\end{figure}
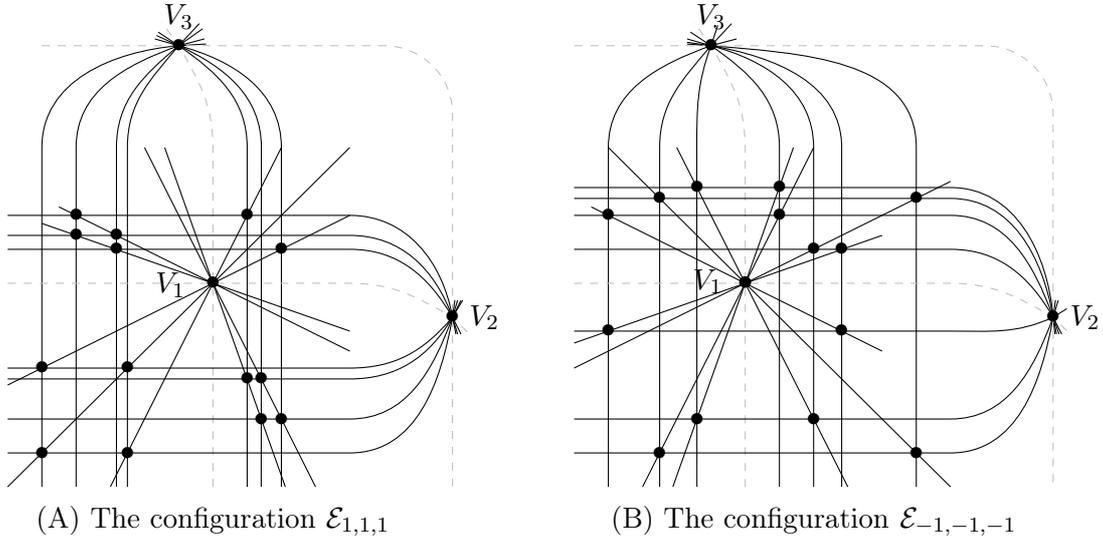

From the pairs constructed with $\N^{\alpha,\beta,\gamma}$, we can derive three other families of Zariski pairs. They are obtained from the arrangements $\N^{\alpha,\beta,\gamma}$ by creating singular points of multiplicity~5. The constructions are summarized in Tables~\ref{tab:ZP17_1pt5},~\ref{tab:ZP17_2pts5} and~\ref{tab:ZP17_2pts5_diago}

\vspace{0.5cm}

\begin{table}[h!]
  \centering
	\begin{tabular}{|l|l l|}
		\hline
		\multicolumn{3}{|c|}{ Configurations $\E^1_{\alpha,\beta,\gamma}$}\\[2pt]
		\hline
		Created alignments & \multicolumn{2}{c|}{ $\set{V_2,S_3,S_5,S_{7}^\alpha,S_{13}^\gamma}$ } \\[4pt]
		\hline
		\multirow{2}*{Modified points}
				& $S_3 = (2:-5:1)$ & $S_6^\alpha = (2:-5:\alpha\sqrt{2})$ \\
				& $S_7^\alpha = (\alpha\sqrt{2}:-5:1)$ & $S_8^\alpha = (2:-5\alpha\sqrt{2}:2)$ \\
		\hline
		Chamber weight & \multicolumn{2}{c|}{$\tau(\E^1_{\alpha,\beta,\gamma})=0\Longleftrightarrow \alpha\beta\gamma=1$}\\[2pt]
		\hline
	\end{tabular}
	\vspace{0.25cm}
	\caption{\label{tab:ZP17_1pt5} Description of the configurations $\E^1_{\alpha,\beta,\gamma}$.}
\end{table}

\begin{table}[h!]
  \centering
	\begin{tabular}{|l|l l l|}
		\hline
		\multicolumn{4}{|c|}{ Configurations $\E^2_{\alpha,\beta,\gamma}$, with $\gamma=1$}\\[2pt]
		\hline
		Created alignments & \multicolumn{3}{c|}{ $\set{V_2,S_3,S_5,S_{7}^\gamma,S_{13}^\gamma}$ and $\set{V_3,S_4,S_{10}^\beta,S_{13}^\gamma,S_{14}^\gamma}$} \\[4pt]
		\hline
		\multirow{3}*{Modified points}
				& \multicolumn{3}{c|}{ $S_3 = (2:-5:1)$ and $S_4=(-5:2:1)$} \\				
				& $S_6^\alpha = (2:-5:\alpha\sqrt{2})$ & $S_7^\alpha = (\alpha\sqrt{2}:-5:1)$ & $S_8^\alpha = (2:-5\alpha\sqrt{2}:2)$ \\			
				& $S_9^\beta = (-5:2:\beta\sqrt{2})$ & $S_{10}^\beta = (-5:\beta\sqrt{2}:1)$ & $S_{11}^\beta = (-5\beta\sqrt{2}:2:2)$ \\
		\hline
		Chamber weight & \multicolumn{3}{c|}{$\tau(\E^2_{\alpha,\beta,\gamma})=0\Longleftrightarrow \alpha\beta=1$}\\[2pt]
		\hline
	\end{tabular}
	\vspace{0.25cm}
	\caption{\label{tab:ZP17_2pts5} Description of the configurations $\E^2_{\alpha,\beta,1}$.}
\end{table}

\begin{table}[h!]
  \centering
	\begin{tabular}{|l|l l l|}
		\hline
		\multicolumn{4}{|c|}{ Configurations $\E^3_{\alpha,\beta,\gamma}$}\\[2pt]
		\hline
		Created alignments & \multicolumn{3}{c|}{ $\set{V_1,S_1, S_3, S_6^\alpha, S_{11}^\beta}$ and $\set{V_1,S_2, S_4, S_9^\beta, S_{14}^\gamma}$ } \\[4pt]
		\hline
		\multirow{3}*{Modified points}
				& \multicolumn{3}{c|}{$S_3 = (2:4:1)$ \quad\quad $S_4 = (4:2:1)$}\\
				& $S_6^\alpha = (2:4:\alpha\sqrt{2})$ & $S_7^\alpha = (\alpha\sqrt{2}:4:1)$ & $S_8^\alpha = (1:2\alpha\sqrt{2}:1)$ \\
				& $S_9^\beta = (4:2:\beta\sqrt{2})$ & $S_{10}^\beta = (4:\beta\sqrt{2}:1)$ & $S_{11}^\beta = (2\beta\sqrt{2}:1:1)$ \\
		\hline
		Chamber weight & \multicolumn{3}{c|}{$\tau(\E^3_{\alpha,\beta,\gamma})=0\Longleftrightarrow \alpha\beta\gamma=1$}\\[2pt]
				\hline
	\end{tabular}
	\vspace{0.25cm}
	\caption{\label{tab:ZP17_2pts5_diago} Description of the configurations $\E^3_{\alpha,\beta,\gamma}$.}
\end{table}

\bigskip
\section{Relation with $\mathcal{I}$-invariant and proof of the main results}\label{sec:proof}
The main idea of the proof is that the chamber weight $\tau$ is a dual version of the $\I$-invariant for complexified real arrangements. Since we are considering $(3,m)$-configurations, we only need to define the $\I$-invariant for a triangular inner-cyclic arrangement. In this way, we first introduce fundamental notions on the $\I$-invariant. Then, we prove our main results from Section~\ref{sec:configuration} by relating this object with $(3,m)$-configurations. Finally, we study the Pappus and non-Pappus arrangements to point out the possible use of this point of view in the study of the characteristic varieties.

\subsection{The $\I$-invariant}\mbox{}

\subsubsection{Definition}

The $\I$-invariant was first introduced by Artal, Florens and the first author in~\cite{AFG:invariant}, and generalized for algebraic plane curves by Meilhan and the first author in~\cite{GM:invariant}. It can be defined as an adaptation of the linking number of knots. In~\cite{AFG:invariant}, the idea is to consider a cycle $\gamma$ in the boundary of a regular neighborhood of the arrangement and then consider it as a cycle in the complement. To eliminate the indeterminacies in the definition of the cycle, we consider the image of the cycle by a particular character $\xi$. In the other hand, following~\cite{GM:invariant}, the cycle $\gamma$ is considered as contained in the arrangement and then, it is viewed as an element of the complement of the sub-arrangement (denoted $\A_\gamma^c$) which does not intersect $\gamma$. The indeterminacies are resolved using a quotient of $\HH_1(\PC^2\setminus\A_\gamma^c)$ by a suitable sub-group. We use here a combination of these two approaches by assuming that the cycle $\gamma$ is contained in the arrangement as in~\cite{GM:invariant}, but we will consider a character $\xi$ as in~\cite{AFG:invariant}.

We consider $\A=\set{D_1,\dots,D_n}$ a line arrangement in $\PC^2$ such that $D_1$, $D_2$ and $D_3$ are in generic position and, for $i,j\in\set{1,2,3}$, $D_i\cap D_j$ is a singular point of multiplicity 2 in $\A$. Let $\gamma$ be the image of the embedding of $S^1$ in the ``triangle'' formed by $D_1$, $D_2$ and $D_3$, such that $\gamma$ has no self-intersection, it passes exactly once by each singular point of the triangle and it does not traverse other singular point of $\A$. We take an orientation of $\gamma$ such that this path passes through $D_1$, $D_2$ and $D_3$ in this cyclic order. We denote by $\A_\gamma^c$ the sub arrangement of $\A$ composed of the lines $D_4,\dots,D_n$. The sub-arrangement $D_1\cup D_2 \cup D_3$ is called the \emph{support} of $\gamma$. Observe that, by construction, $\gamma$ is contained in $\PC^2\setminus\A_\gamma^c$.

Let $\xi:\HH_1(\PC^2\setminus\A)\rightarrow \CC^*$ be a character on the first integral homology group of the complement of $\A$, such that $\xi(\m_{D_i})= 1$ if and only if $i=1,2,3$, where $\m_D$ is the homology class of a meridian of the line $D\in\A$. We denote by $\tilde{\xi}$ the restriction of $\xi$ to $\HH_1(\PC^2\setminus\A_\gamma^c)$. The triple $(\A,\xi,\gamma)$ is called a \emph{triangular inner-cyclic arrangement} if for all singular point $P$ of $\A$ contained in $D_1\cup D_2\cup D_3$, we have $\prod\limits_{D\ni P}\xi(\m_D) =1$.

\begin{rmk}
	The previous definition of triangular inner-cyclic arrangements is slightly stronger than the one in~\cite{AFG:invariant}. Indeed, no line of $\A$ is mapped to~$1$ by the character~$\xi$ except $D_1,D_2$ and $D_3$.
\end{rmk}

\begin{thm}[\cite{AFG:invariant}]
	Let $(\A,\xi,\gamma)$ be a triangular inner-cyclic arrangement. Then, the value
	\begin{equation*}
		\I(\A,\xi,\gamma) = \tilde{\xi} ([\gamma])
	\end{equation*}
	is invariant by homeomorphism of the pair $(\PC^2,\A)$ respecting the order and the orientation of~$\A$.
\end{thm}

\begin{rmk}\mbox{}
	\begin{enumerate}
		\item The condition ``respecting the orientation'' means that the homeomorphism of the pair must respect the global orientation of $\PC^2$ and the local orientation of the meridian around the line of $\A$. 
		\item By~\cite[Corollary 2.6.]{AFG:invariant}, if $\A$ is a complexified real arrangement then
		\[
		    \I(\A,\xi,\gamma)=\pm 1.
		\]
	\end{enumerate}
\end{rmk}

\subsubsection{Computation}

The computation of the $\I$-invariant is described in~\cite[Section 4]{AFG:invariant}. We explain here how to compute it in the particular case of a triangular inner-cyclic arrangement which is also complexified real, since it is the case of the DPA of any $(3,m)$-configuration.

In the general case, the $\I$-invariant is computed from the \emph{wiring diagram} of $\A$ (for further details see~\cite[Section~4]{AFG:invariant}). For a complexified real arrangement $\A$, the real picture $\A\cap\PR^2$ is a wiring diagram of $\A$. Up to a projective transformation, we may assume that $D_1$, $D_2$ and $D_3$ are respectively $z=0$, $x+y=0$ and $x-y=0$, and that no line of $\A$ is of the form $y+az=0$ for $a\in\RR$. In the following, $D_1$ is considered as the line at infinity and $\CC^2$ (resp. $\RR^2$) refers to $\PC^2\setminus D_1$ (resp. $\PR^2\setminus D_1$). The trace of the lines $D_2,\dots,D_n$ in $\RR^2$ are denoted by $D'_2,\dots,D'_n$ respectively. Let $H$ be the half-plane of $\RR^2$ defined by $H=\set{(x,y)\in\RR^2\mid x<0}$. In order to compute the invariant, let us denote by $\D_2$ (resp. $\D_3$)  the set of lines defined by $\set{D\in\A \mid D'\cap D'_2 \in H,\ \Sl(D')<\Sl(D'_2)}$ (resp. $\set{D\in\A \mid D'\cap D'_3 \in H,\ \Sl(D')<\Sl(D'_3)}$), where $\Sl(D')$ is the slope of the affine real line $D'$; and by $\D_1$ the set of lines $D$ in $\A$ such that their associated real affine line $D'$ has a slope between -1 and 1.

\begin{rmk}\label{rmk:D2D3}
    We have $\D_2\subset\D_3$ and $\D_3\setminus\D_2=\D_1\cap\D_3$. Indeed, if a line $D_j$ intersects $D'_2$ in $H$ with a slope smaller than $-1=\Sl(D'_2)$, then this line also intersects $D'_3$ in $H$ with a slope smaller than $1=\Sl(D'_3)$ since the intersection point of $D'_2$ and $D'_3$ is on the boundary of the half-plane $H$. For example, it is the case of $D'_5$ in Figure~\ref{fig:Ceva7}.
\end{rmk}

\begin{propo}\label{propo:computation}
	Let $(\A,\xi,\gamma)$ be a triangular inner-cyclic arrangement. If $\A$ is a complexified real arrangement, then we have
	\begin{equation*}
		\I(\A,\xi,\gamma) = \xi\Big( \sum\limits_{D\in\D_1\cap\D_3} \m_D \Big) \inv.
	\end{equation*}
\end{propo}

\begin{proof}
The proof is based on an adaptation of the proof of Proposition 4.3 in~\cite{AFG:invariant} to the case of triangular inner-cyclic arrangement which are also real complexified. First, we prove that
\begin{equation}\label{eq:computation}
	\I(\A,\xi,\gamma) = \xi\Big(  \sum\limits_{D\in\D_2} \m_D - \sum\limits_{D'\in\D_3} \m_{D'} \Big),
\end{equation}
which is an adaptation to our case of Lemma 4.2 in~\cite{AFG:invariant}.

We denote as in~\cite[Section 4]{AFG:invariant} $\beta_{0,u}$ the braid obtained from the singular braid defined by $\big( D'_1\cup\cdots\cup D'_n \big) \cap H$, by exchanging the singular points by the corresponding local oriented half-twist. Also denote $\beta_{u,0}$ the mirror braid of $\beta_{0,u}$ reversing the orientation. Let $a_{j,s}(\beta)$ be the number (counting with sign) of crossings of the $j^\text{th}$ string over the $s^\text{th}$ string in $\beta$. Lemma 4.2 in~\cite{AFG:invariant} states that there exists a cycle $\gamma$ such that
\begin{equation*}
	[\gamma]=\sum\limits_{j=4}\limits^n a_{j,2}(\beta_{0,u})\cdot \m_{D_j} + \sum\limits_{j=4}\limits^n a_{j,3}(\beta_{u,0})\cdot \m_{D_j} \in \HH_1(\PC^2\setminus \A_\gamma^c).
\end{equation*}
In our case, the strings over-crossing the string associated to $D'_2$ (resp. $D'_3$) correspond to the lines intersecting $D'_2$ (resp. $D'_3$) in $H$ with smaller slope than $D'_2$ (resp. $D'_3$), that is, exactly the set $\D_2$ (resp. $\D_3$). Each string over-crosses at most once the string associated with $D'_2$ (resp. $D'_3$) since they come from the intersection of the corresponding line with $D_2$ (resp. $D_3$). Finally, the minus sign before the second sum in Equation~(\ref{eq:computation}) is due to the fact that $\beta_{u,0}$ is the mirror of $\beta_{0,u}$ and that the orientation of the crossings are inverted.

To obtain the result, the inclusion of $\D_2$ in $\D_3$ implies that for any $j=4,\ldots,n$ if $a_{j,3}(\beta_{u,0})\neq 0$ then $a_{j,2}(\beta_{0,u})=-a_{j,3}(\beta_{u,0})$. Thus, all terms of the first sum are canceled by terms of the second one in Equation~(\ref{eq:computation}). As $\D_3\setminus\D_2=\D_1\cap\D_3$, the result follows.
\end{proof}

\begin{example}\label{ex:invariant_computation}
	Let $\A$ be the dual arrangement of the $(3,2)$-configuration in Figure~\ref{fig:config_examples}~(A). Its affine real drawn is given in Figure~\ref{fig:Ceva7}, together with the half-plan $H$. We consider the character $\xi$ sending $\m_{D_i}$ to $-1$, for $i=4,\dots,7$. Then $(\A,\xi,\gamma)$ is a triangular inner-cyclic arrangement. In order to compute $\I(\A,\xi,\gamma)$, we need to compute the image of $\gamma$ in $\HH_1(\PC^2\setminus\A_\gamma^c)$. We have $\D_1=\set{D'_6,D'_7}$, $\D_2=\set{D'_5}$ and $\D_3=\set{D'_5,D'_7}$, therefore $[\gamma]=-\m_{D_7}$. Finally, $\I(\A,\xi,\gamma)=\xi(-\m_{D_7}) = -1$.
\end{example}

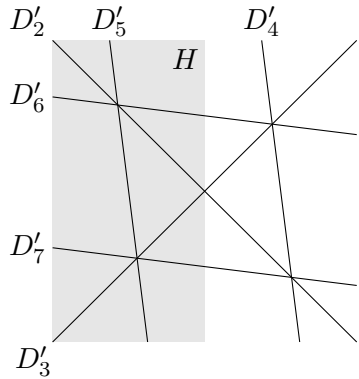
\begin{figure}[h]
	\begin{tikzpicture}
		\begin{scope}[scale=1]
		\fill[color=gray!20] (-2,2) -- (0,2) -- (0,-2) -- (-2,-2);
		
		\draw (-1.25,2) -- (-0.75,-2);
		\draw (0.75,2) -- (1.25,-2);
		\draw (-2,1.25) --(2,0.75);
		\draw (-2,-0.75) --(2,-1.25);
		\draw (-2,-2) -- (2,2);
		\draw (-2,2) -- (2,-2);
		
		\node at (0.75,2.25) {$D'_4$};
		\node at (-1.25,2.25) {$D'_5$};
		\node at (-2.33,2.25) {$D'_2$};
		\node at (-2.33,1.25) {$D'_6$};
		\node at (-2.33,-0.75) {$D'_7$};
		\node at (-2.25,-2.25) {$D'_3$};		
		\node at (-0.25,1.75) {$H$};		
	\end{scope}
\end{tikzpicture}
	\caption{Dual arrangement of the configuration given in Figure~\ref{fig:config_examples}~(A)\label{fig:Ceva7}}
\end{figure}

\subsection{Relation with the chamber weight and proof of the topological invariance}\mbox{}

\subsubsection{$(3,m)$-configuration vs triangular inner-cyclic arrangement}

We show how to relate the $\I$-invariant with $\tau(\C)$. In all this section we consider $\C=(\V,\S,\L,\pl)$ a planar $(3,m)$-configuration, $(\A^\V,\A^\S,\xi)$ its DPA, and $\gamma$ a cycle contained in $\A^\V$. 

\begin{propo}
	The triplet $(\A^\C,\xi,\gamma)$ is a triangular inner-cyclic arrangement.
\end{propo}

\begin{proof}
	The aim here is to prove that the support of the DPA is also the support of the cycle~$\gamma$.
	
	Since the configuration $\C$ is planar, the arrangement $\A^\V$ is generic. By Definition~\ref{def:configuration}~(1), for $i,j\in\set{1,2,3}$, the singular points $V_i^\bullet\cap V_j^\bullet$ are of multiplicity 2 in $\A^\C$. The singular points of $\A^\C$ contained in $\A^\V$ are exactly the (complexified) dual of the lines $L\in\L$ (see Definition~\ref{def:configuration}), we denote them by $L^\bullet$. Finally, Definition~\ref{def:configuration}~(2) implies that for all $L^\bullet$ we have
	\begin{equation*}
		\prod\limits_{S^\bullet\ni L^\bullet} \xi(\m_{S^\bullet}) = 
		\exp\left(\frac{2 \pi i }{m}\sum\limits_{S\in L\cap\S}\pl(S)\right) = 1.\qedhere
	\end{equation*}
\end{proof}

\begin{rmk}
	The converse is also true. That is, if $(\A,\xi,\gamma)$ is a triangular inner-cyclic complexified real arrangement then its dual configuration of points admits a structure of $(3,m)$-configuration.
\end{rmk}

\begin{thm}\label{thm:relation}
	Let $\C$ be a $(3,m)$-configuration and $(\A^\C,\xi,\gamma)$ its associated triangular inner-cyclic arrangement. The chamber weight of $\C$ and the $\I$-invariant of $(\A^\C,\xi,\gamma)$ are related by the formula
	\begin{equation*}
		\I(\A^\C,\xi,\gamma) = \exp\left(\frac{-2 \pi i}{m}\tau(\C)\right).
	\end{equation*}
\end{thm}

\begin{proof}
	By Proposition~\ref{propo:computation}, we need to show that the set $\D_1\cap \D_3$ corresponds to dualizing the set $\S\cap\ch_i$, for a chamber in $\set{\ch_1,\ch_2,\ch_3,\ch_4}$. Recall that
	\begin{equation*}
		\D_1\cap\D_3 = \set{ D \in \A\ \mid\ D'\cap D_3 \in H,\ -1<\Sl(D')<1}.
	\end{equation*}
	A line $D$ such that $D'\cap D'_3$ lies in $H$ corresponds to a dual point $P$, i.e. $P^\bullet=D$, is contained in $\mathfrak{c}_3$, one of the two cones of $\PR^2$ defined by the lines $(V_1,V_3)$ and $(V_2,V_3)$. In the same way, a line $D$ such that $-1<\Sl(D')<1$ corresponds to a dual point contained in $\mathfrak{c}_1$, one of the two cones of $\PR^2$ defined by the lines $(V_1,V_2)$ and $(V_1,V_3)$. 
	
	The intersection of these two cones $\mathfrak{c}_1$ and $\mathfrak{c}_3$ is exactly one of the chambers $\ch_i$ (see Figure~\ref{fig:chamber}). In addition, the lines in $\D_1\cap\D_3$ are lines of $\A$, then their dual points are in $\V\sqcup\S$. But they are distinct from $V_i^\bullet$, for $i=1,2,3$, thus they are in $\S$. Finally, we obtain that the dual of $\D_1\cap\D_3$ is $\S\cap\ch_i$. 
	
	Since $\tau(\C)$ is defined as the sum of the plumbing of the surrounding-points contained in a chamber $\ch_i$, then the statement follows from the definition of the character $\xi$ on the DPA. 
\end{proof}

\begin{figure}[h!]
\begin{tikzpicture}
 \fill[pattern color=gray!50, pattern=north west lines] (-1,0) -- (-2,0) -- (-2,-1);
 \fill[pattern color=gray!50, pattern=north west lines] (-1,0) -- (1,2) -- (2,0);
 \fill[pattern color=gray!50, pattern=north east lines] (1,0) -- (2,0) -- (2,-1);
 \fill[pattern color=gray!50, pattern=north east lines] (1,0) -- (-1,2) -- (-2,0);
 \node at (-1,0) {$\bullet$};
 \node at (1,0) {$\bullet$};
 \node at (0,1) {$\bullet$}; 
 \node[below right] at (-1.1,0) {$V_1$};
 \node[below left] at (1.1,0) {$V_3$};
 \node[above] at (0,1.1) {$V_2$};
 \node at (0,0.33) {$\ch_i$};
 \node at (1,1) {$\mathfrak{c}_1$};
 \node at (-1,1) {$\mathfrak{c}_3$};
 \draw (-2,-1) -- (1,2);
 \draw (-2,0) -- (2,0);
 \draw (-1,2) -- (2,-1);
\end{tikzpicture}
	\caption{Intersection of two sectors.\label{fig:chamber}}
\end{figure}
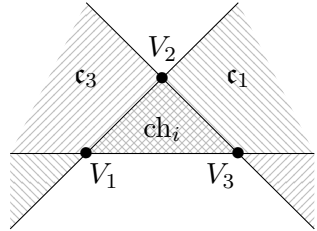

\subsubsection{An alternative proof of Proposition~\ref{propo:tau_constant}}

The above relation gives an alternative topological proof of the fact that the chamber weight $\tau(\C)$ is well defined.


	Indeed, in the proof of Theorem~\ref{thm:relation}, we showed that for a particular choice of a cycle $\gamma$, the $\I$-invariant is obtained by the chamber weight of a given chamber. In fact, this chamber is determined by the orientation of $\gamma$ and the line of its support which is taken as line at infinity in our construction of $\D_1$, $\D_2$ and $\D_3$.  As the value of the $\I$-invariant does not depend on the choice of $\gamma$ and the line at infinity which determines $\D_1$, $\D_2$ and $\D_3$, we deduce that all chamber weights are equal.

\begin{rmk}
	As a consequence of Proposition~\ref{propo:tau_constant}, we only consider configurations with a plumbing in $\ZZ/m\ZZ$ with $m$ even. Thus, the character $\xi$ of the DPA also have an even order. Since $\I(\A^\C,\xi,\gamma)^{-1}=\I(\A^\C,\xi,\gamma^{-1})$, the order of the character $\xi$ implies that $\I$ is independent of the orientation of $\gamma$ in the case of a DPA.
\end{rmk}

\subsubsection{Proofs of Theorem~\ref{thm:invariant} and of Corollary~\ref{cor:invariant}}

From Theorem~\ref{thm:relation} and \cite[Theorem 2.1]{AFG:invariant}, the value  $\tau(\C)$ is an invariant of the ordered and oriented topology of $\A^\C$. Since the arrangement considered are complexified real arrangement, then they are invariant by complex conjugation (i.e. $\overline{\A^\C}=\A^\C$), we can thus apply~\cite[Corollary 2.6]{AFG:invariant}. Remark that this implies that if $m$ is even then $\tau(\C)$ is $0$ or $m/2$ in $\ZZ/m\ZZ$; and if $m$ is odd then $\tau(\C)$ is zero, as it was already proven in Proposition~\ref{propo:tau_constant}. Therefore by~\cite[Theorem 4.19]{ACCM:real_ZP} (or equivalently~\cite[Lemma 4.12]{Guerville:ZP}), we can remove the orientation condition and obtain Theorem~\ref{thm:invariant}.

In order to obtain Corollary~\ref{cor:invariant}, it is enough to observe that the value of $\tau(\C)$ is invariant by any automorphism of the combinatorics of $\C$ when the configuration is stable and uniform. Indeed, in such a case, $\#(\S\cap \ch_i)$ is unchanged by any automorphism of the combinatorics since the configuration is stable. As the plumbing is constant on $\S$, we have
\begin{equation*}
	\tau(\C) = \sum\limits_{S\in\S\cap \ch_i} \pl(S) = \zeta \cdot \#(\S\cap \ch_i),
\end{equation*}
where $\zeta=\pl(S)$ for any $S\in\S$.

\subsection{Study of characteristic varieties}\mbox{}

For the definition of the quasi-projective part of the characteristic varieties and how to compute it from the $\I$-invariant and the combinatorics, we refer to~\cite[Section 3]{AFG:invariant}. In~\cite{artal:char_var}, Artal proves that the characteristic varieties are not determined by the weak combinatorics, i.e. by the number of lines, the number of singular points on each line and their multiplicities. More precisely, he gives a negative answer to Orlik's conjecture which states that
\begin{equation*}
	b_1(\F_\A) = b_1(\PC^2\setminus\A) \quad \text{ and } \quad b_2(\F_\A)=b_2(\PC^2\setminus\A)+ (m-1) \chi(\PC^2\setminus\A),
\end{equation*}
where $\F_\A$ is the Milnor fiber of $\A$. In order to obtain this result, he computes the corresponding Betti numbers for two classical examples of arrangements.

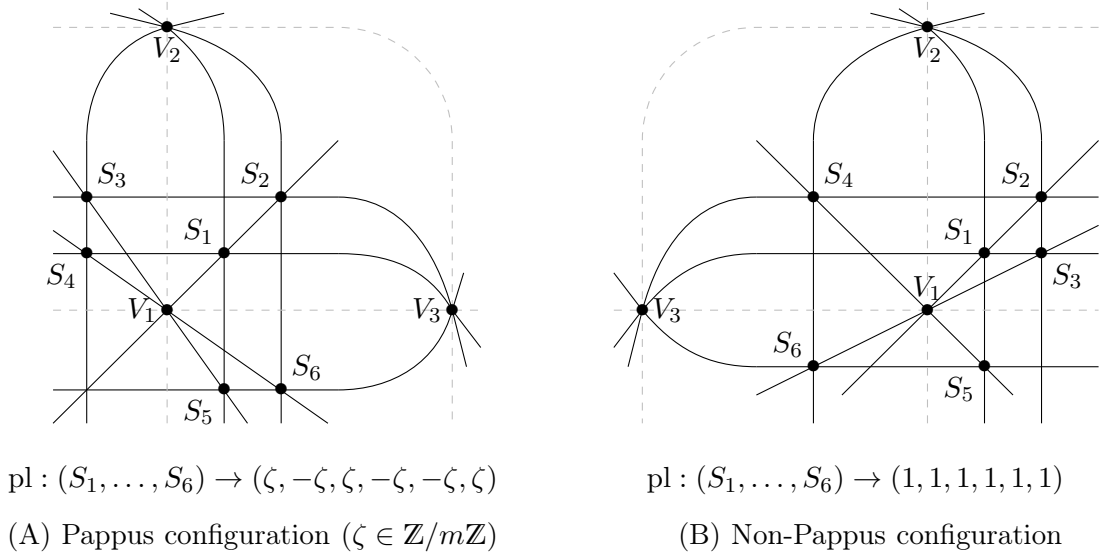
\begin{figure}[h!]
	\begin{tikzpicture}
		\begin{scope}[shift={(-5,0)},scale=0.75]
	 \draw (-1.41,-2) -- (-1.41,3);
	 \draw (1,-2) -- (1,3);
	 \draw (2,-2) -- (2,3);
	 
 	\draw (-1.41,3) to[out=90, in=-165] (0,5) -- (1,5.25);
	\draw (1,3) to[out=90, in=-40] (0,5) -- (-0.5,5.33);
	\draw (2,3) to[out=90, in=-15] (0,5) -- (-1,5.25);
	
	 \draw (-2,-1.41) --(3,-1.41);
	 \draw (-2,1) --(3,1);
	 \draw (-2,2) --(3,2);
	 
	\draw (3,2) to[out=0, in=105] (5,0) -- (5.25,-1);
	\draw (3,1) to[out=0, in=120] (5,0) -- (5.5,-0.66);
	\draw (3,-1.41) to[out=0, in=-105] (5,0) -- (5.2,0.66);
	 
	 \draw (-2,-2) -- (3,3);
	 \draw (-2,2.82) -- (1.41,-2);
	 \draw (-2,1.41) -- (2.82,-2);
	 
	 \draw[color=gray!50, dashed] (0,-2) -- (0,5.5);
	 \draw[color=gray!50, dashed] (-2,0) -- (5.5,0);
	 \draw[color=gray!50, dashed] (-2,5) -- (3,5) to[out=0,in=90] (5,3) --(5,-2);
	 
	 \node at (0,0) {$\bullet$};
	 \node[left] at (0,0) {$V_1$};
	 \node at (0,5) {$\bullet$};
	 \node[below] at (0,5) {$V_2$};
	 \node at (5,0) {$\bullet$};
	 \node[left] at (5,0) {$V_3$};
	 
	 \node at (1,1) {$\bullet$};
	 \node[above left] at (1,1) {$S_1$};
	 \node at (2,2) {$\bullet$};
	 \node[above left] at (2,2) {$S_2$};
	 \node at (-1.41,2) {$\bullet$};
	 \node[above right] at (-1.41,2) {$S_3$};
	 \node at (-1.41,1) {$\bullet$};
	 \node[below left] at (-1.41,1) {$S_4$};
	 \node at (1,-1.41) {$\bullet$};
	 \node[below left] at (1,-1.41) {$S_5$};
	 \node at (2,-1.41) {$\bullet$};
	 \node[above right] at (2,-1.41) {$S_6$};
	 
 	\node at (1.5,-3) {$\pl:(S_1,\dots,S_6) \rightarrow (\zeta,-\zeta,\zeta,-\zeta,-\zeta,\zeta)$};
	\node at (1.5,-4) {(A) Pappus configuration ($\zeta\in\ZZ/m\ZZ$)};
		\end{scope}
		\begin{scope}[shift={(5,0)},scale=0.75]
	\draw (-2,-2) -- (-2,3);
	\draw (1,-2) -- (1,3);
	\draw (2,-2) -- (2,3);
	
	\draw (-2,3) to[out=90, in=-165] (0,5) -- (1,5.25);
	\draw (1,3) to[out=90, in=-40] (0,5) -- (-0.5,5.33);
	\draw (2,3) to[out=90, in=-15] (0,5) -- (-1,5.25);
	
	\draw (-3,-1) -- (3,-1);
	\draw (-3,1) -- (3,1);
	\draw (-3,2) -- (3,2);

	\draw (-3,2) to[out=-180, in=75] (-5,0) -- (-5.25,-1);
	\draw (-3,1) to[out=-180, in=50] (-5,0) -- (-5.5,-0.66);
	\draw (-3,-1) to[out=-180, in=-50] (-5,0) -- (-5.5,0.66);
	
	\draw (-3,3) -- (1.5,-1.5);
	\draw (-3,-1.5) -- (3,1.5);
	\draw (-1.5,-1.5) --( 3,3);
	
	 \draw[color=gray!50, dashed] (0,-2) -- (0,5.5);
	 \draw[color=gray!50, dashed] (-5.5,0) -- (3,0);
	 \draw[color=gray!50, dashed] (3,5) -- (-3,5) to[out=180,in=90] (-5,3) -- (-5,-2);
	
	\node at (0,0) {$\bullet$};
	\node[above] at (0,0) {$V_1$};
	\node at (0,5) {$\bullet$};
	\node[below] at (0,5) {$V_2$};
	\node at (-5,0) {$\bullet$};
	\node[right] at (-5,0) {$V_3$};
	
	\node at (1,1) {$\bullet$};
	\node[above left] at (1,1) {$S_1$};
	\node at (2,2) {$\bullet$};
	\node[above left] at (2,2) {$S_2$};
	\node at (2,1) {$\bullet$};
	\node[below right] at (2,1) {$S_3$};
	\node at (-2,2) {$\bullet$};
	\node[above right] at (-2,2) {$S_4$};
	\node at (1,-1) {$\bullet$};
	\node[below left] at (1,-1) {$S_5$};
	\node at (-2,-1) {$\bullet$};
	\node[above left] at (-2,-1) {$S_6$};
	
	\node at (-1.25,-3) {$\pl:(S_1,\dots,S_6)\rightarrow (1,1,1,1,1,1)$};
	\node at (-1,-4) {(B) Non-Pappus configuration};
		\end{scope}
	\end{tikzpicture}
	\caption{Pappus and non-Pappus $(3,2)$-configurations with chamber weight 0 and 1, respectively.\label{fig:Pappus_nonPappus}}
\end{figure}

Combining the results in~\cite{artal:position} with the fact that a triplet $(\A^\V,\A^\S,\xi)$ is the DPA of a $(3,m)$-configuration $\C$ together with the relation between the $\I$-invariant and the chamber weight (Theorem~\ref{thm:relation}), we obtain that the quasi-projective depth of $\xi$ is given by
\begin{equation*}
	\depth(\xi) = \corank
		\begin{pmatrix}
			1-l_1 & 1 & 1 \\
			1 & 1-l_2 & j \\
			1 & j & 1-l_3
		\end{pmatrix}
		, \text{ with }
		j = \left\{
			\begin{array}{ll}
				1 & \text{if }\tau(\C)=0 \\
				-1 & \text{if }\tau(\C)=m/2 \\
			\end{array}
		\right. ,
\end{equation*}
and $l_i$ is the number of lines in $\L$ passing through $V_i$, for $i=1,2,3$. Using this formula, we have an alternative proof of the result of Artal~\cite{artal:char_var}. Indeed, if we consider the $(3,2)$-configurations given in Figure~\ref{fig:Pappus_nonPappus} (called the Pappus and the non-Pappus configurations), we obtain that: 
\begin{itemize}
	\item For any character $\xi$ of the Pappus arrangement, we have $\depth(\xi)>0$,
	\item If $(\A^\V,\A^\S,\xi)$ is the DPA of the non-Pappus configuration, then $\depth(\xi)=0$.
\end{itemize}
The Pappus and the non-Pappus arrangements have the same weak combinatorics, that is, they have 9 lines, 9 triple points and 3 of them on each line. But they have not the same combinatorics. These arrangements are two of the three different possible realizations for this weak combinatorics (called \emph{configurations of type $(9_3)$} in~\cite[Sec.~2.2]{Grunbaum}). In fact, these two are the unique realizations which are triangular inner-cyclic arrangements. Thus, we obtain the following result.

\begin{propo}\label{propo:char_var}
	The quasi-projective part of the characteristic varieties of a real line arrangement is not determined by the weak combinatorics.
\end{propo}

Using the same argument for the $(3,m)$-configurations of our Zariski pairs, we can compute the projective depth of the character defined by each DPA and obtain the following result.
\begin{propo}
The quasi-projective part of the characteristic variety related to the dual arrangements of the configurations $\C_{\alpha,\beta}$, $\C_{\alpha,\beta}^1$, $\C_{\alpha,\beta}^2$, $\D_{\alpha,\beta}$, $\D_{\alpha,\beta}^1$, $\D_{\alpha,\beta}^2$, $\E_{\alpha,\beta,\gamma}$, $\E_{\alpha,\beta,\gamma}^1$, $\E_{\alpha,\beta,\gamma}^2$ and $\E_{\alpha,\beta,\gamma}^3$ are combinatorially determined.
\end{propo}

In the case of triangular inner-cyclic arrangements, the matrix in the previous formula giving $\depth(\xi)$ becomes a strictly diagonal dominant matrix when $l_1$, $l_2$ and $l_3$ are greater or equal to 4 (i.e. when we have at least 4 singular points lying in each line of the support). In such case, the matrix becomes full-rank and $\depth(\xi)$ is 0.  Almost all DPA of $(3,m)$-configuration have this property. In particular, it is the case of all our previous examples of Zariski pairs.

\subsection{Conclusion}\mbox{}

We have shown how the use of the $(3,m)$-configurations and the chamber weight can be efficient to study the topology of line arrangements. First, to detect new examples of Zariski pairs as done in Section~\ref{sec:ZP} and~\ref{sec:list}, but also to study the problem of the combinatoriality of characteristic varieties, as show with Proposition~\ref{propo:char_var}. Nevertheless, the following general question is still open:
\begin{center}
	\smallskip
		{ \it  Are the characteristic varieties combinatorially determined? }
	\smallskip
\end{center}

We could contemplate the use of this new approach, with the $(t,m)$-configurations and the chamber weight, to attempt to obtain an answer for this question.

\newpage
\appendix
\section{The $(3,2)$-configurations $\C_{\alpha,\beta}$ and their duals}\label{sec:appendix_ZP13}

In Section~\ref{sec:ZP}, we introduce four the $(3,m)$-configurations $\C_{\alpha,\beta}$ and their dual arrangements $\A^{\alpha,\beta}$, for $\alpha,\beta\in\{-1,1\}$, which composed Zariski pairs of 13 lines. We picture in this appendix detailed real pictures of both the configurations and dual arrangements. The dual line arrangements are pictured here up to rotation of their associated configurations, in order to obtain a more presentable figures.

\vfill
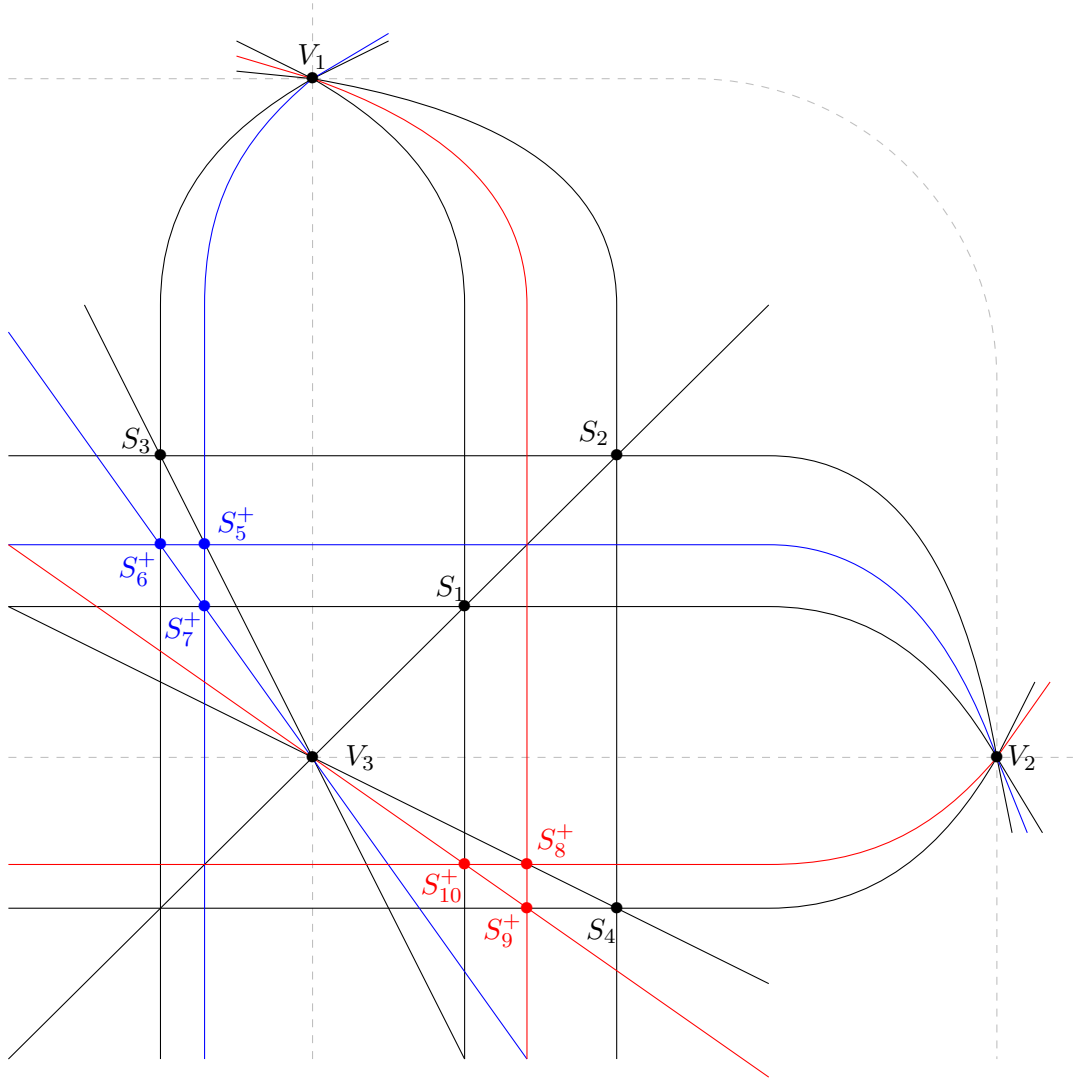
\begin{figure}[h]
	\centering	
	\begin{tikzpicture}[scale=2]
	\def\colorap{blue} 
	\def\coloram{teal}  
	\def\colorbp{red} 
	\def\colorbm{orange} 

		\draw[dashed, color=gray!50] (0,-2) -- (0,5) ;
		\draw[dashed, color=gray!50] (-2,0) -- (5,0) ;
		\draw[dashed, color=gray!50] (-2,4.5) -- (2.5,4.5) to[out=0,in=90] (4.5,2.5) -- (4.5,-2);

		\draw (-2,-1) -- (3,-1) to[out=0,in=-120] (4.5,0) -- (4.75,0.5);
		\draw (-2,2) -- (3,2) to[out=0,in=100] (4.5,0) -- (4.6,-0.5);
		\draw (-2,1) -- (3,1) to[out=0,in=120] (4.5,0) -- (4.8,-0.5);
		\draw (-1,-2) -- (-1,3) to[out=90,in=-150] (0,4.5) -- (0.5,4.75);
		\draw (1,-2) -- (1,3) to[out=90,in=-30] (0,4.5) -- (-0.5,4.75);
		\draw (2,-2) -- (2,3) to[out=90,in=-10] (0,4.5) -- (-0.5,4.55);
		\draw (-2,-2) -- (3,3);
		\draw (-2,1) -- (3,-1.5);
		\draw (-1.5,3) -- (1,-2);

		\draw[color=\colorap] (-2,1.41) -- (3,1.41) to[out=0,in=110] (4.5,0) -- (4.7,-0.5);
		\draw[color=\colorap] (-0.71,-2) -- (-0.71,3) to[out=90,in=-140] (0,4.5) -- (0.5,4.8);
		\draw[color=\colorap] (-2,2.82) -- (1.41,-2);


		\draw[color=\colorbp] (-2,-0.71) -- (3,-0.71) to[out=0,in=-130] (4.5,0) -- (4.85,0.5);
		\draw[color=\colorbp] (1.41,-2) -- (1.41,3) to[out=90,in=-20] (0,4.5) -- (-0.5,4.65);
		\draw[color=\colorbp] (-2,1.41) -- (3,-2.12);



		 \node at (0,0) {$\bullet$};
		 \node[right] at (0.15,0) {$V_3$};
		 \node at (0,4.5) {$\bullet$};
		 \node[above] at (0,4.5) {$V_1$};
		 \node at (4.5,0) {$\bullet$};
		 \node[right] at (4.5,0) {$V_2$};
	 
		 \node at (1,1) {$\bullet$}; \node at (0.91,1.12) {$S_1$};
		 \node at (2,2) {$\bullet$}; \node at (1.85,2.15) {$S_2$};
		 \node at (-1,2) {$\bullet$}; \node at (-1.16,2.1) {$S_3$};
		 \node at (2,-1) {$\bullet$}; \node at (1.9,-1.15) {$S_4$};
	 
		 \node[text=\colorap] at (-0.71,1.41) {$\bullet$}; \node[text=\colorap] at (-0.5,1.55) {$S_5^{+}$};
		 \node[text=\colorap] at (-1,1.41) {$\bullet$}; \node[text=\colorap] at (-1.15,1.25) {$S_6^{+}$};
		 \node[text=\colorap] at (-0.71,1) {$\bullet$};  \node[text=\colorap] at (-0.85,0.85) {$S_7^{+}$};
	 
		 \node[text=\colorbp] at (1.41,-0.71) {$\bullet$}; \node[text=\colorbp] at (1.6,-0.55) {$S_8^{+}$};
		 \node[text=\colorbp] at (1.41,-1) {$\bullet$}; \node[text=\colorbp] at (1.25,-1.15) {$S_9^{+}$};
		 \node[text=\colorbp] at (1,-0.71) {$\bullet$}; \node[text=\colorbp] at (0.85,-0.85) {$S_{10}^{+}$};
\end{tikzpicture}
	\caption{The $(3,2)$-configuration $\C_{1,1}$.\label{fig:dual_ZP_pp}}
\end{figure}
\vfill

\newpage\mbox{}

\vfill
\begin{figure}[h]
	\centering	
	\begin{tikzpicture}[scale=0.35]
\def\colorV{black}
\def\colorS{black}
\def\colorSa{blue}
\def\colorSb{red}
\coordinate (P1) at (-1, -20);
\coordinate (Q1) at (-1, 22);
\coordinate (P2) at (1, -20);
\coordinate (Q2) at (1, 22);
\coordinate (P3) at (-20, 0);
\coordinate (Q3) at (20, 0);
\coordinate (P4) at (-20, 1);
\coordinate (Q4) at (20, 1);
\coordinate (P5) at (-20, 4);
\coordinate (Q5) at (20, 4);
\coordinate (P6) at (17, -20);
\coordinate (Q6) at (-11, 22);
\coordinate (P7) at (-17, -20);
\coordinate (Q7) at (11, 22);
\coordinate (P8) at (-20, 71/4);
\coordinate (Q8) at (20, -49/4);
\coordinate (P9) at (49/5, -20);
\coordinate (Q9) at (-7, 22);
\coordinate (P10) at (89/5, -20);
\coordinate (Q10) at (-79/5, 22);
\coordinate (P11) at (-20, -49/4);
\coordinate (Q11) at (20, 71/4);
\coordinate (P12) at (-49/5, -20);
\coordinate (Q12) at (7, 22);
\coordinate (P13) at (-89/5, -20);
\coordinate (Q13) at (79/5, 22);
\draw[color=\colorV, line width=1] (P1)--(Q1) node[pos=-.05] {$V^\bullet_1$};
\draw[color=\colorV, line width=1] (P2)--(Q2) node[pos=-.05] {$V^\bullet_2$};
\draw[color=\colorV, line width=1] (P3)--(Q3) node[pos=-.05, yshift=-1] {$V^\bullet_3$};
\draw[color=\colorS, line width=1] (P4)--(Q4) node[pos=-.05, yshift=1] {$S^\bullet_1$};
\draw[color=\colorS, line width=1] (P5)--(Q5) node[pos=-.05] {$S^\bullet_2$};
\draw[color=\colorS, line width=1] (P6)--(Q6) node[pos=-.05, xshift=-1] {$S^\bullet_3$};
\draw[color=\colorS, line width=1] (P7)--(Q7) node[pos=-.05, xshift=1] {$S^\bullet_4$};
\draw[color=\colorSa, line width=1] (P8)--(Q8) node[pos=-.05] {$S^\bullet_5$};
\draw[color=\colorSa, line width=1] (P9)--(Q9) node[pos=-.05] {$S^\bullet_6$};
\draw[color=\colorSa, line width=1] (P10)--(Q10) node[pos=-.05, xshift=1] {$S^\bullet_7$};
\draw[color=\colorSb, line width=1] (P11)--(Q11) node[pos=-.05] {$S^\bullet_8$};
\draw[color=\colorSb, line width=1] (P12)--(Q12) node[pos=-.05] {$S^\bullet_9$};
\draw[color=\colorSb, line width=1] (P13)--(Q13) node[pos=-.05, xshift=-1] {$S^\bullet_{10}$};
\end{tikzpicture}
	\caption{The arrangement $\A^{1,1}$.\label{fig:ZP_pp}}
\end{figure}
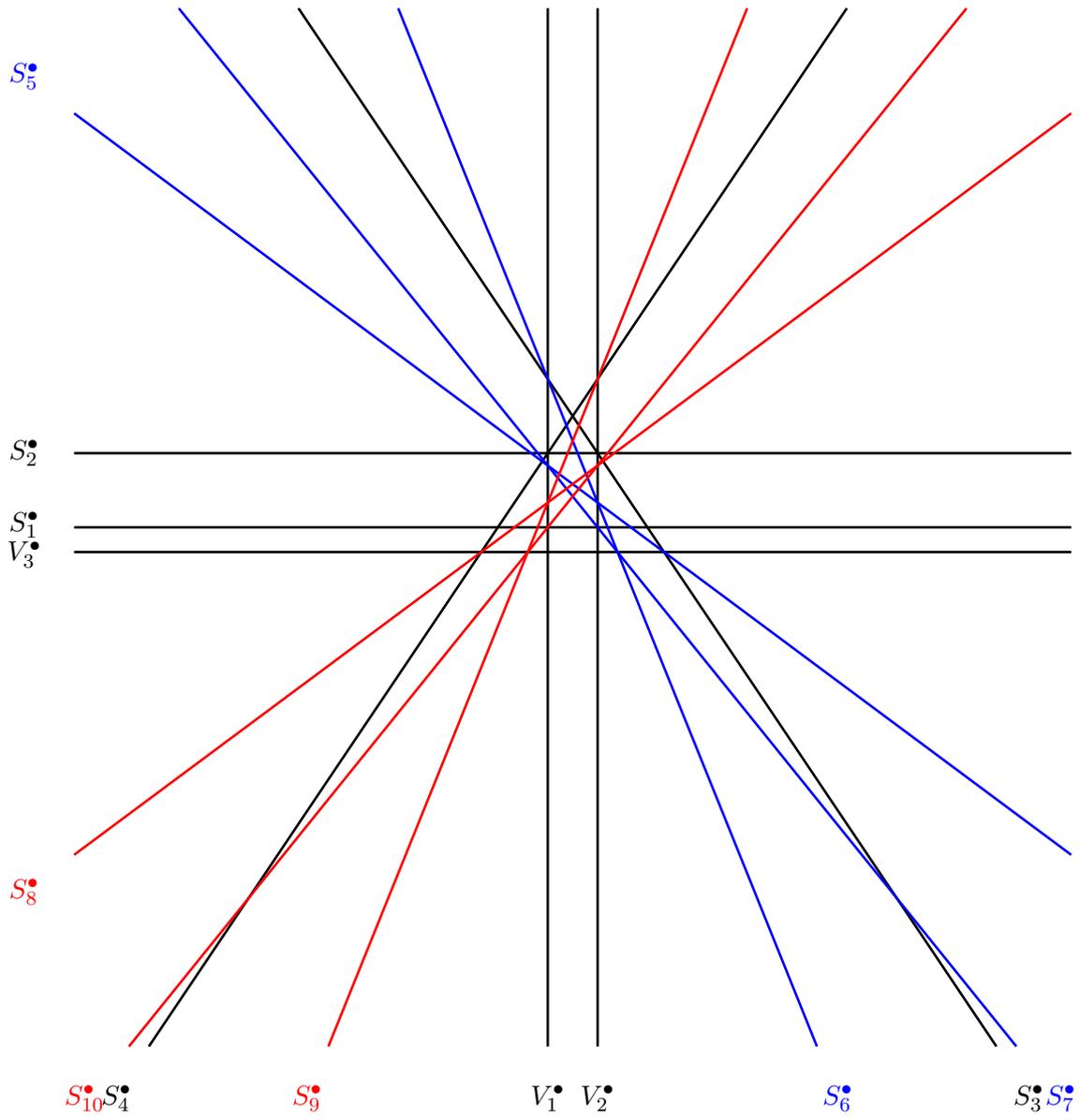
\vfill

\newpage\mbox{}

\vfill
\begin{figure}[h]
	\begin{tikzpicture}[scale=2]
	\def\colorap{blue} 
	\def\coloram{teal}  
	\def\colorbp{red} 
	\def\colorbm{orange} 

		\draw[dashed, color=gray!50] (0,-2) -- (0,5) ;
		\draw[dashed, color=gray!50] (-2,0) -- (5,0) ;
		\draw[dashed, color=gray!50] (-2,4.5) -- (2.5,4.5) to[out=0,in=90] (4.5,2.5) -- (4.5,-2);

		\draw (-2,-1) -- (3,-1) to[out=0,in=-120] (4.5,0) -- (4.75,0.5);
		\draw (-2,2) -- (3,2) to[out=0,in=100] (4.5,0) -- (4.6,-0.5);
		\draw (-2,1) -- (3,1) to[out=0,in=120] (4.5,0) -- (4.8,-0.5);
		\draw (-1,-2) -- (-1,3) to[out=90,in=-150] (0,4.5) -- (0.5,4.75);
		\draw (1,-2) -- (1,3) to[out=90,in=-30] (0,4.5) -- (-0.5,4.75);
		\draw (2,-2) -- (2,3) to[out=90,in=-10] (0,4.5) -- (-0.5,4.55);
		\draw (-2,-2) -- (3,3);
		\draw (-2,1) -- (3,-1.5);
		\draw (-1.5,3) -- (1,-2);

		\draw[color=\colorap] (-2,1.41) -- (3,1.41) to[out=0,in=110] (4.5,0) -- (4.7,-0.5);
		\draw[color=\colorap] (-0.71,-2) -- (-0.71,3) to[out=90,in=-140] (0,4.5) -- (0.5,4.8);
		\draw[color=\colorap] (-2,2.82) -- (1.41,-2);



		\draw[color=\colorbm] (-2,0.71) -- (3,0.71) to[out=0,in=130] (4.5,0) -- (4.9,-0.5);
		\draw[color=\colorbm] (-1.41,-2) -- (-1.41,3) to[out=90,in=-170] (0,4.5) -- (0.5,4.55);
		\draw[color=\colorbm] (-2,-1.41) -- (3,2.12);


		 \node at (0,0) {$\bullet$};
		 \node[right] at (0.15,0) {$V_3$};
		 \node at (0,4.5) {$\bullet$};
		 \node[above] at (0,4.5) {$V_1$};
		 \node at (4.5,0) {$\bullet$};
		 \node[right] at (4.5,0) {$V_2$};
	 
		 \node at (1,1) {$\bullet$}; \node at (0.91,1.12) {$S_1$};
		 \node at (2,2) {$\bullet$}; \node at (1.85,2.15) {$S_2$};
		 \node at (-1,2) {$\bullet$}; \node at (-1.16,2.1) {$S_3$};
		 \node at (2,-1) {$\bullet$}; \node at (1.9,-1.15) {$S_4$};
	 
		 \node[text=\colorap] at (-0.71,1.41) {$\bullet$}; \node[text=\colorap] at (-0.5,1.55) {$S_5^{+}$};
		 \node[text=\colorap] at (-1,1.41) {$\bullet$}; \node[text=\colorap] at (-1.15,1.25) {$S_6^{+}$};
		 \node[text=\colorap] at (-0.71,1) {$\bullet$};  \node[text=\colorap] at (-0.85,0.85) {$S_7^{+}$};
	 
		 \node[text=\colorbm] at (-1.41,0.71) {$\bullet$}; \node[text=\colorbm] at (-1.6,0.55) {$S_8^{-}$};
		 \node[text=\colorbm] at (-1.41,-1) {$\bullet$}; \node[text=\colorbm] at (-1.55,-0.85) {$S_9^{-}$};
		 \node[text=\colorbm] at (1,0.71) {$\bullet$}; \node[text=\colorbm] at (1.2,0.55) {$S_{10}^{-}$};
\end{tikzpicture}
	\caption{The $(3,2)$-configuration $\C_{1,-1}$.\label{fig:dual_ZP_pm}}
\end{figure}
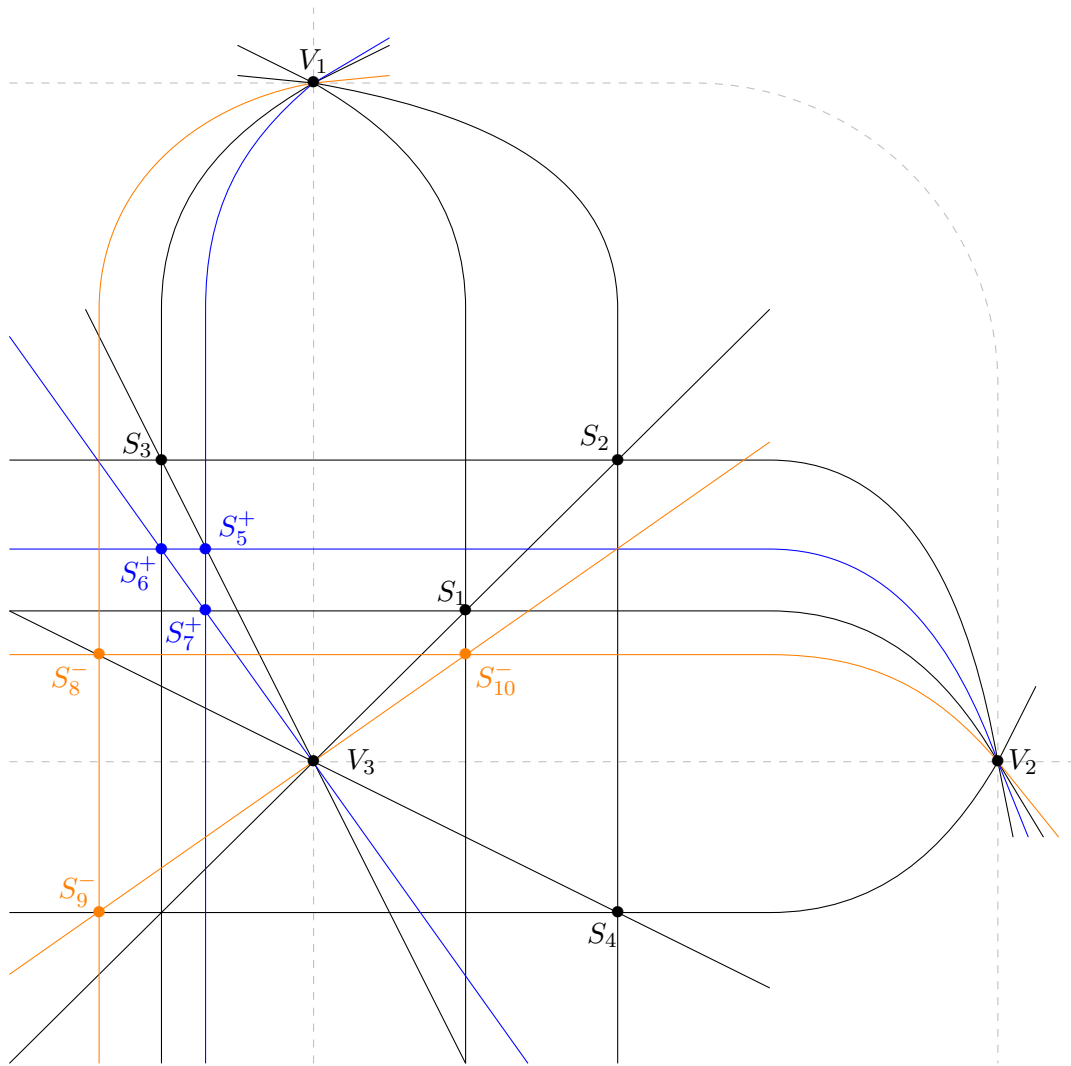
\vfill

\newpage\mbox{}

\vfill
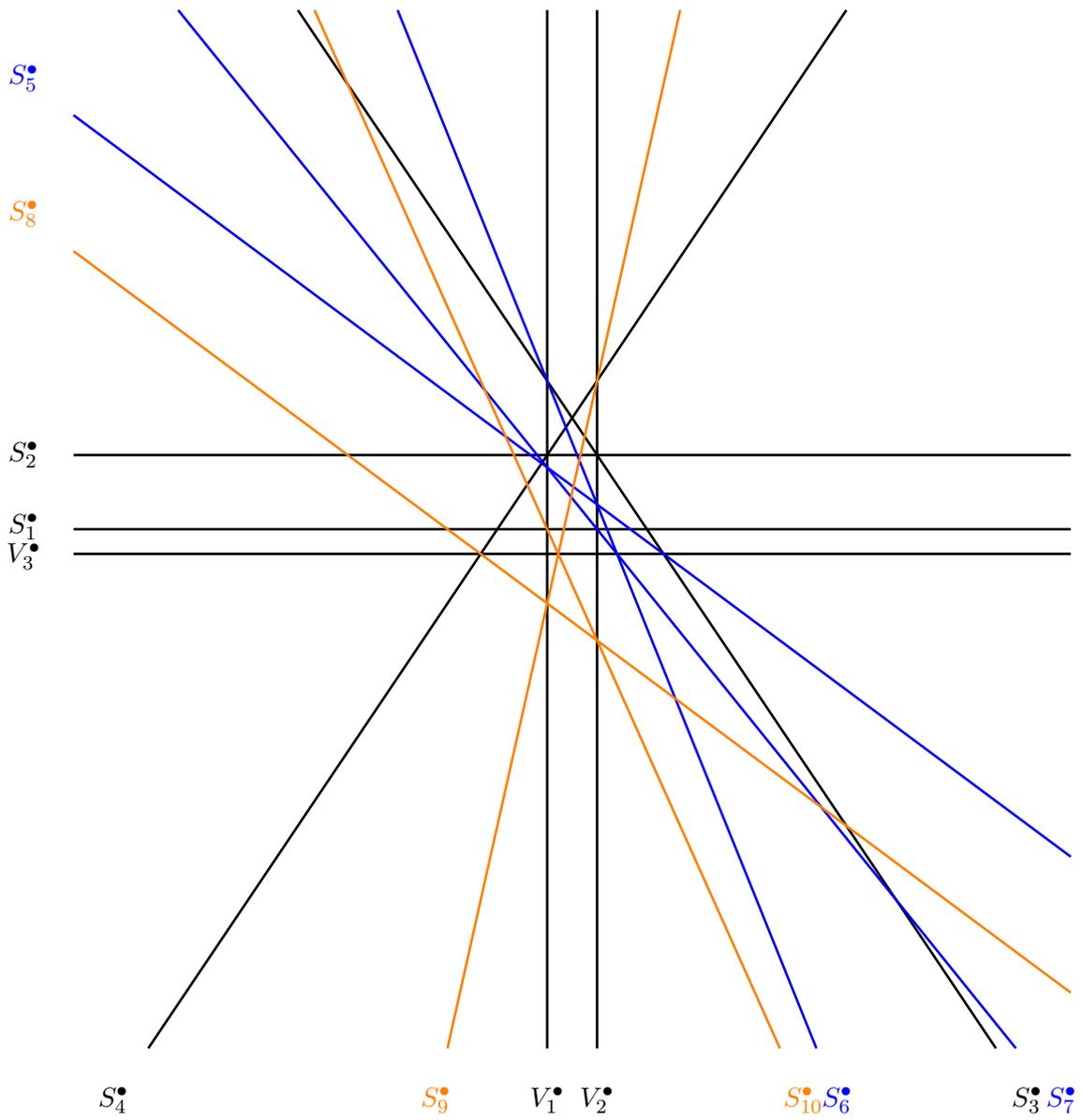
\begin{figure}[h]
	\centering	
	\begin{tikzpicture}[scale=0.35]
\def\colorV{black}
\def\colorS{black}
\def\colorSa{blue}
\def\colorSb{orange}
\coordinate (P1) at (-1, -20);
\coordinate (Q1) at (-1, 22);
\coordinate (P2) at (1, -20);
\coordinate (Q2) at (1, 22);
\coordinate (P3) at (-20, 0);
\coordinate (Q3) at (20, 0);
\coordinate (P4) at (-20, 1);
\coordinate (Q4) at (20, 1);
\coordinate (P5) at (-20, 4);
\coordinate (Q5) at (20, 4);
\coordinate (P6) at (17, -20);
\coordinate (Q6) at (-11, 22);
\coordinate (P7) at (-17, -20);
\coordinate (Q7) at (11, 22);
\coordinate (P8) at (-20, 71/4);
\coordinate (Q8) at (20, -49/4);
\coordinate (P9) at (49/5, -20);
\coordinate (Q9) at (-7, 22);
\coordinate (P10) at (89/5, -20);
\coordinate (Q10) at (-79/5, 22);
\coordinate (P11) at (-20, 49/4);
\coordinate (Q11) at (20, -71/4);
\coordinate (P12) at (-5, -20);
\coordinate (Q12) at (13/3, 22);
\coordinate (P13) at (25/3, -20);
\coordinate (Q13) at (-31/3, 22);
\draw[color=\colorV, line width=1] (P1)--(Q1) node[pos=-.05] {$V^\bullet_1$};
\draw[color=\colorV, line width=1] (P2)--(Q2) node[pos=-.05] {$V^\bullet_2$};
\draw[color=\colorV, line width=1] (P3)--(Q3) node[pos=-.05, yshift=-1] {$V^\bullet_3$};
\draw[color=\colorS, line width=1] (P4)--(Q4) node[pos=-.05, yshift=1] {$S^\bullet_1$};
\draw[color=\colorS, line width=1] (P5)--(Q5) node[pos=-.05] {$S^\bullet_2$};
\draw[color=\colorS, line width=1] (P6)--(Q6) node[pos=-.05, xshift=-1.5] {$S^\bullet_3$};
\draw[color=\colorS, line width=1] (P7)--(Q7) node[pos=-.05] {$S^\bullet_4$};
\draw[color=\colorSa, line width=1] (P8)--(Q8) node[pos=-.05] {$S^\bullet_5$};
\draw[color=\colorSa, line width=1] (P9)--(Q9) node[pos=-.05] {$S^\bullet_6$};
\draw[color=\colorSa, line width=1] (P10)--(Q10) node[pos=-.05, xshift=1.5] {$S^\bullet_7$};
\draw[color=\colorSb, line width=1] (P11)--(Q11) node[pos=-.05] {$S^\bullet_8$};
\draw[color=\colorSb, line width=1] (P12)--(Q12) node[pos=-.05] {$S^\bullet_9$};
\draw[color=\colorSb, line width=1] (P13)--(Q13) node[pos=-.05] {$S^\bullet_{10}$};
\end{tikzpicture}
	\caption{The arrangement $\A^{1,-1}$.\label{fig:ZP_pm}}
\end{figure}
\vfill

\newpage\mbox{}

\vfill
\begin{figure}[h]
	\begin{tikzpicture}[scale=2]
	\def\colorap{blue} 
	\def\coloram{teal}  
	\def\colorbp{red} 
	\def\colorbm{orange} 

		\draw[dashed, color=gray!50] (0,-2) -- (0,5) ;
		\draw[dashed, color=gray!50] (-2,0) -- (5,0) ;
		\draw[dashed, color=gray!50] (-2,4.5) -- (2.5,4.5) to[out=0,in=90] (4.5,2.5) -- (4.5,-2);

		\draw (-2,-1) -- (3,-1) to[out=0,in=-120] (4.5,0) -- (4.75,0.5);
		\draw (-2,2) -- (3,2) to[out=0,in=100] (4.5,0) -- (4.6,-0.5);
		\draw (-2,1) -- (3,1) to[out=0,in=120] (4.5,0) -- (4.8,-0.5);
		\draw (-1,-2) -- (-1,3) to[out=90,in=-150] (0,4.5) -- (0.5,4.75);
		\draw (1,-2) -- (1,3) to[out=90,in=-30] (0,4.5) -- (-0.5,4.75);
		\draw (2,-2) -- (2,3) to[out=90,in=-10] (0,4.5) -- (-0.5,4.55);
		\draw (-2,-2) -- (3,3);
		\draw (-2,1) -- (3,-1.5);
		\draw (-1.5,3) -- (1,-2);


		\draw[color=\coloram] (-2,-1.41) -- (3,-1.41) to[out=0,in=-100] (4.5,0) -- (4.55,0.5);
		\draw[color=\coloram] (0.71,-2) -- (0.71,3) to[out=90,in=-40] (0,4.5) -- (-0.5,4.85);
		\draw[color=\coloram] (-1.41,-2) -- (2.12,3);

		\draw[color=\colorbp] (-2,-0.71) -- (3,-0.71) to[out=0,in=-130] (4.5,0) -- (4.85,0.5);
		\draw[color=\colorbp] (1.41,-2) -- (1.41,3) to[out=90,in=-20] (0,4.5) -- (-0.5,4.65);
		\draw[color=\colorbp] (-2,1.41) -- (3,-2.12);



		 \node at (0,0) {$\bullet$};
		 \node[right] at (0.15,0) {$V_3$};
		 \node at (0,4.5) {$\bullet$};
		 \node[above] at (0,4.5) {$V_1$};
		 \node at (4.5,0) {$\bullet$};
		 \node[right] at (4.5,0) {$V_2$};
	 
		 \node at (1,1) {$\bullet$}; \node at (0.91,1.12) {$S_1$};
		 \node at (2,2) {$\bullet$}; \node at (1.85,2.15) {$S_2$};
		 \node at (-1,2) {$\bullet$}; \node at (-1.16,2.1) {$S_3$};
		 \node at (2,-1) {$\bullet$}; \node at (1.9,-1.15) {$S_4$};
	 
		 \node[text=\coloram] at (0.71,-1.41) {$\bullet$}; \node[text=\coloram] at (0.55,-1.55) {$S_5^{-}$};
		 \node[text=\coloram] at (-1,-1.41) {$\bullet$}; \node[text=\coloram] at (-0.85,-1.55) {$S_6^{-}$};
		 \node[text=\coloram] at (0.71,1) {$\bullet$}; \node[text=\coloram] at (0.55,1.15) {$S_7^{-}$};
	 
		 \node[text=\colorbp] at (1.41,-0.71) {$\bullet$}; \node[text=\colorbp] at (1.6,-0.55) {$S_8^{+}$};
		 \node[text=\colorbp] at (1.41,-1) {$\bullet$}; \node[text=\colorbp] at (1.25,-1.15) {$S_9^{+}$};
		 \node[text=\colorbp] at (1,-0.71) {$\bullet$}; \node[text=\colorbp] at (0.85,-0.85) {$S_{10}^{+}$};
\end{tikzpicture}
	\caption{The $(3,2)$-configuration $\C_{-1,1}$.\label{fig:dual_ZP_mp}}
\end{figure}
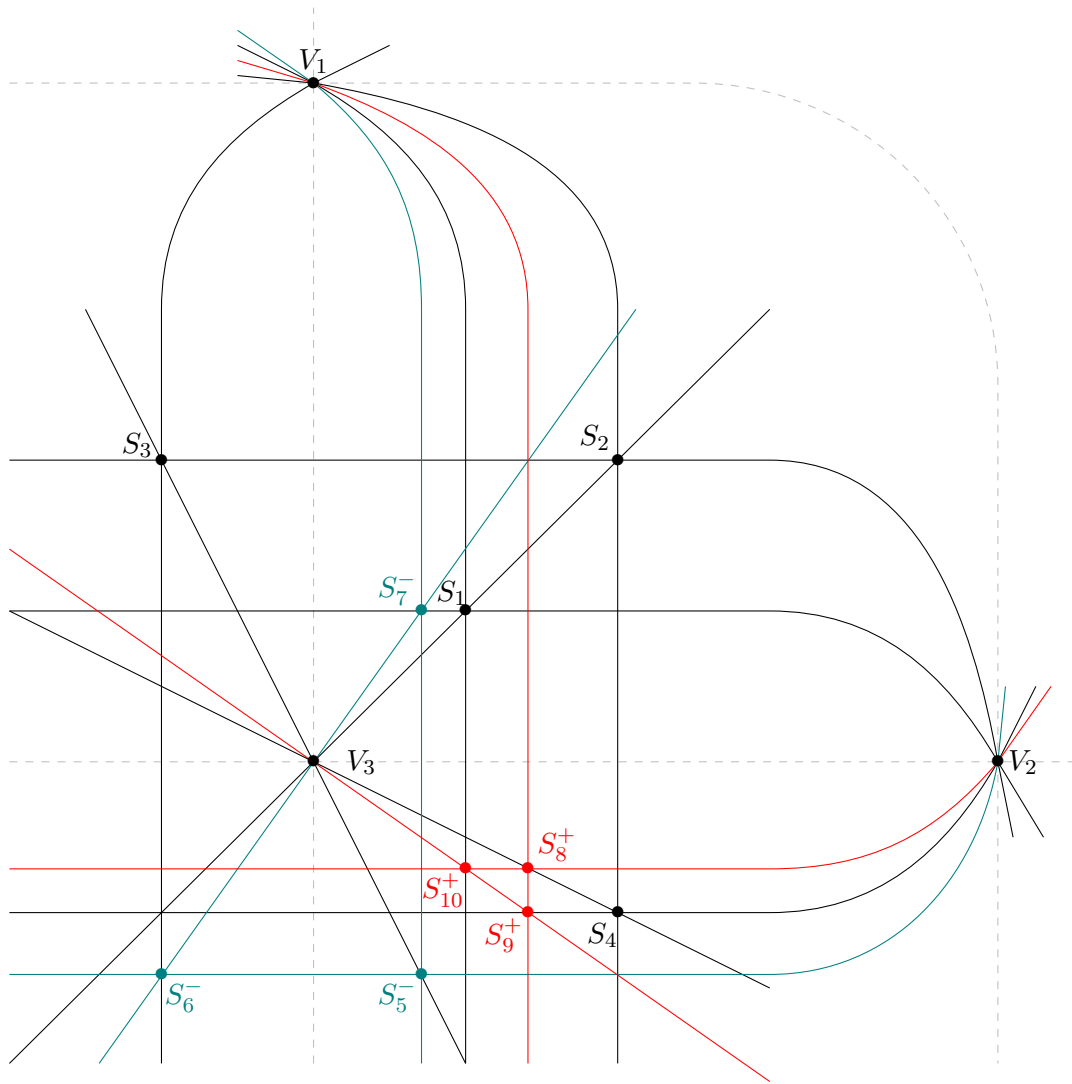
\vfill

\newpage\mbox{}

\vfill
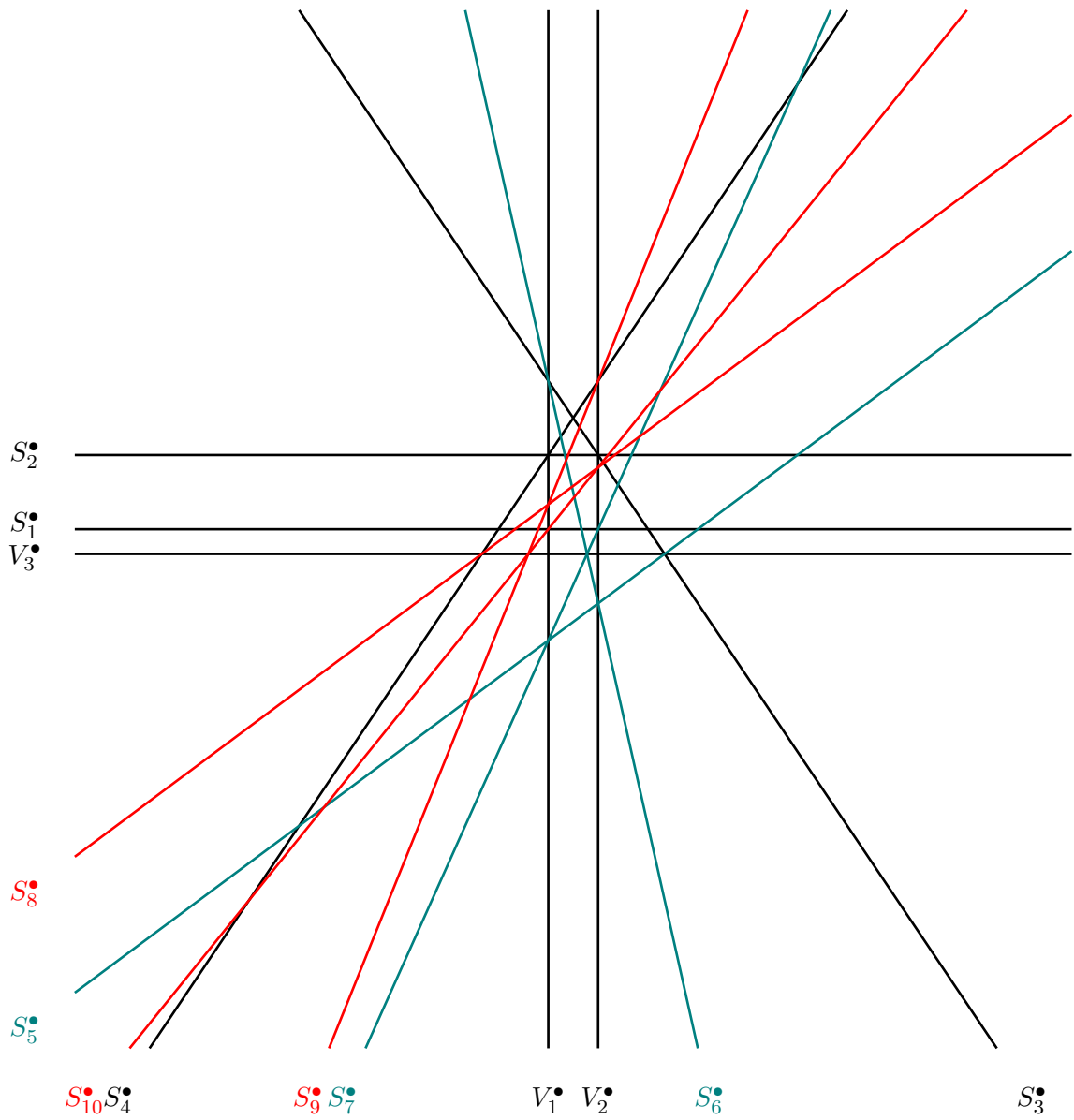
\begin{figure}[h]
	\centering	
	\begin{tikzpicture}[scale=0.35]
\def\colorV{black}
\def\colorS{black}
\def\colorSa{teal}
\def\colorSb{red}
\coordinate (P1) at (-1, -20);
\coordinate (Q1) at (-1, 22);
\coordinate (P2) at (1, -20);
\coordinate (Q2) at (1, 22);
\coordinate (P3) at (-20, 0);
\coordinate (Q3) at (20, 0);
\coordinate (P4) at (-20, 1);
\coordinate (Q4) at (20, 1);
\coordinate (P5) at (-20, 4);
\coordinate (Q5) at (20, 4);
\coordinate (P6) at (17, -20);
\coordinate (Q6) at (-11, 22);
\coordinate (P7) at (-17, -20);
\coordinate (Q7) at (11, 22);
\coordinate (P8) at (-20, -71/4);
\coordinate (Q8) at (20, 49/4);
\coordinate (P9) at (5, -20);
\coordinate (Q9) at (-13/3, 22);
\coordinate (P10) at (-25/3, -20);
\coordinate (Q10) at (31/3, 22);
\coordinate (P11) at (-20, -49/4);
\coordinate (Q11) at (20, 71/4);
\coordinate (P12) at (-49/5, -20);
\coordinate (Q12) at (7, 22);
\coordinate (P13) at (-89/5, -20);
\coordinate (Q13) at (79/5, 22);
\draw[color=\colorV, line width=1] (P1)--(Q1) node[pos=-.05] {$V^\bullet_1$};
\draw[color=\colorV, line width=1] (P2)--(Q2) node[pos=-.05] {$V^\bullet_2$};
\draw[color=\colorV, line width=1] (P3)--(Q3) node[pos=-.05, yshift=-1.5] {$V^\bullet_3$};
\draw[color=\colorS, line width=1] (P4)--(Q4) node[pos=-.05, yshift=1.5] {$S^\bullet_1$};
\draw[color=\colorS, line width=1] (P5)--(Q5) node[pos=-.05] {$S^\bullet_2$};
\draw[color=\colorS, line width=1] (P6)--(Q6) node[pos=-.05] {$S^\bullet_3$};
\draw[color=\colorS, line width=1] (P7)--(Q7) node[pos=-.05, xshift=1.5] {$S^\bullet_4$};
\draw[color=\colorSa, line width=1] (P8)--(Q8) node[pos=-.05] {$S^\bullet_5$};
\draw[color=\colorSa, line width=1] (P9)--(Q9) node[pos=-.05] {$S^\bullet_6$};
\draw[color=\colorSa, line width=1] (P10)--(Q10) node[pos=-.05] {$S^\bullet_7$};
\draw[color=\colorSb, line width=1] (P11)--(Q11) node[pos=-.05] {$S^\bullet_8$};
\draw[color=\colorSb, line width=1] (P12)--(Q12) node[pos=-.05] {$S^\bullet_9$};
\draw[color=\colorSb, line width=1] (P13)--(Q13) node[pos=-.05, xshift=-1.5] {$S^\bullet_{10}$};
\end{tikzpicture}
	\caption{The arrangement $\A^{-1,1}$.\label{fig:ZP_mp}}
\end{figure}
\vfill

\newpage\mbox{}

\vfill
\begin{figure}[h]
	\begin{tikzpicture}[scale=2]
	\def\colorap{blue} 
	\def\coloram{teal}  
	\def\colorbp{red} 
	\def\colorbm{orange} 

		\draw[dashed, color=gray!50] (0,-2) -- (0,5) ;
		\draw[dashed, color=gray!50] (-2,0) -- (5,0) ;
		\draw[dashed, color=gray!50] (-2,4.5) -- (2.5,4.5) to[out=0,in=90] (4.5,2.5) -- (4.5,-2);

		\draw (-2,-1) -- (3,-1) to[out=0,in=-120] (4.5,0) -- (4.75,0.5);
		\draw (-2,2) -- (3,2) to[out=0,in=100] (4.5,0) -- (4.6,-0.5);
		\draw (-2,1) -- (3,1) to[out=0,in=120] (4.5,0) -- (4.8,-0.5);
		\draw (-1,-2) -- (-1,3) to[out=90,in=-150] (0,4.5) -- (0.5,4.75);
		\draw (1,-2) -- (1,3) to[out=90,in=-30] (0,4.5) -- (-0.5,4.75);
		\draw (2,-2) -- (2,3) to[out=90,in=-10] (0,4.5) -- (-0.5,4.55);
		\draw (-2,-2) -- (3,3);
		\draw (-2,1) -- (3,-1.5);
		\draw (-1.5,3) -- (1,-2);


		\draw[color=\coloram] (-2,-1.41) -- (3,-1.41) to[out=0,in=-100] (4.5,0) -- (4.55,0.5);
		\draw[color=\coloram] (0.71,-2) -- (0.71,3) to[out=90,in=-40] (0,4.5) -- (-0.5,4.85);
		\draw[color=\coloram] (-1.41,-2) -- (2.12,3);


		\draw[color=\colorbm] (-2,0.71) -- (3,0.71) to[out=0,in=130] (4.5,0) -- (4.9,-0.5);
		\draw[color=\colorbm] (-1.41,-2) -- (-1.41,3) to[out=90,in=-170] (0,4.5) -- (0.5,4.55);
		\draw[color=\colorbm] (-2,-1.41) -- (3,2.12);


		 \node at (0,0) {$\bullet$};
		 \node[right] at (0.15,0) {$V_3$};
		 \node at (0,4.5) {$\bullet$};
		 \node[above] at (0,4.5) {$V_1$};
		 \node at (4.5,0) {$\bullet$};
		 \node[right] at (4.5,0) {$V_2$};
	 
		 \node at (1,1) {$\bullet$}; \node at (0.91,1.12) {$S_1$};
		 \node at (2,2) {$\bullet$}; \node at (1.85,2.15) {$S_2$};
		 \node at (-1,2) {$\bullet$}; \node at (-1.16,2.1) {$S_3$};
		 \node at (2,-1) {$\bullet$}; \node at (1.9,-1.15) {$S_4$};
	 
		 \node[text=\coloram] at (0.71,-1.41) {$\bullet$}; \node[text=\coloram] at (0.55,-1.55) {$S_5^{-}$};
		 \node[text=\coloram] at (-1,-1.41) {$\bullet$}; \node[text=\coloram] at (-0.85,-1.55) {$S_6^{-}$};
		 \node[text=\coloram] at (0.71,1) {$\bullet$}; \node[text=\coloram] at (0.55,1.15) {$S_7^{-}$};
	 
		 \node[text=\colorbm] at (-1.41,0.71) {$\bullet$}; \node[text=\colorbm] at (-1.6,0.55) {$S_8^{-}$};
		 \node[text=\colorbm] at (-1.41,-1) {$\bullet$}; \node[text=\colorbm] at (-1.55,-0.85) {$S_9^{-}$};
		 \node[text=\colorbm] at (1,0.71) {$\bullet$}; \node[text=\colorbm] at (1.2,0.55) {$S_{10}^{-}$};
\end{tikzpicture}
	\caption{The $(3,2)$-configuration $\C_{-1,-1}$.\label{fig:dual_ZP_mm}}
\end{figure}
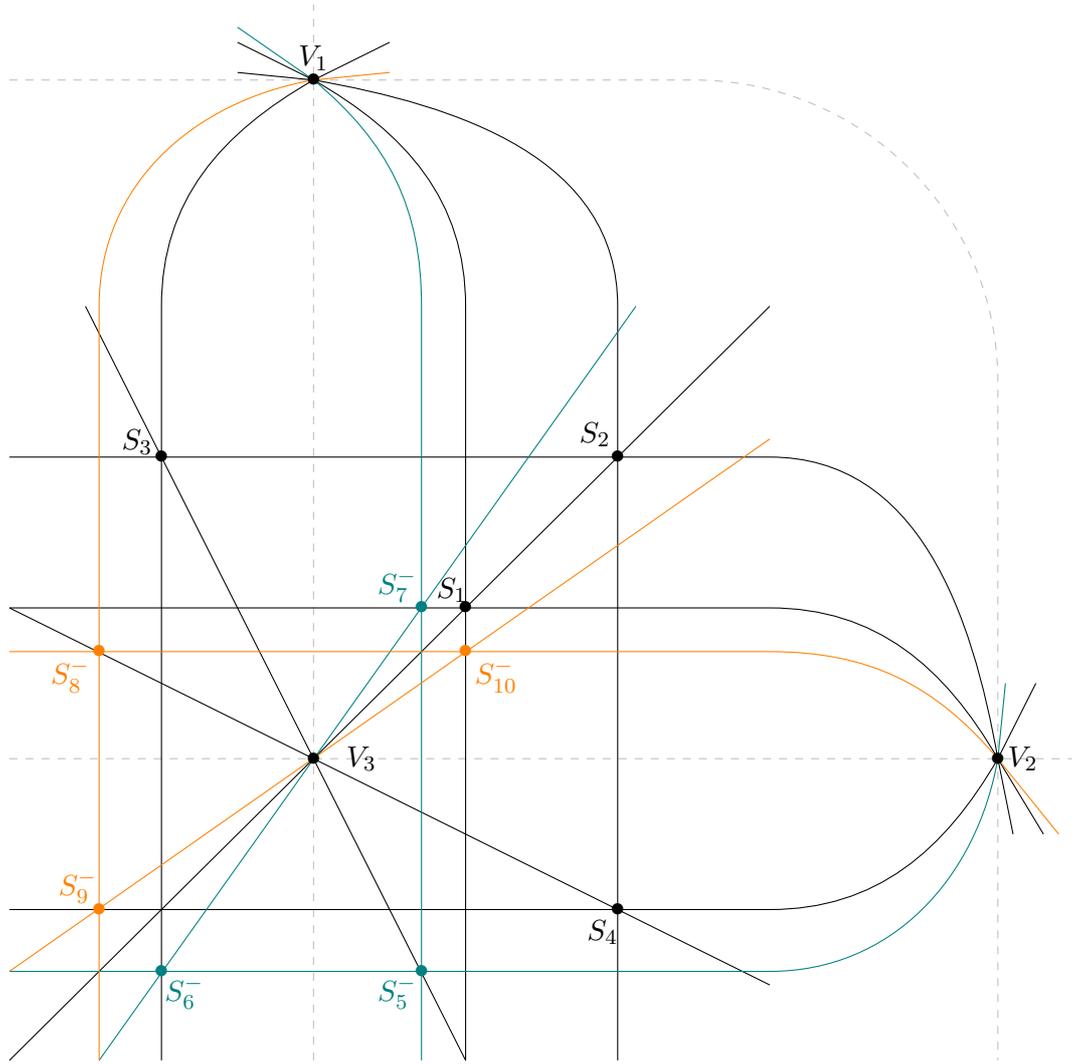
\vfill

\newpage\mbox{}

\vfill
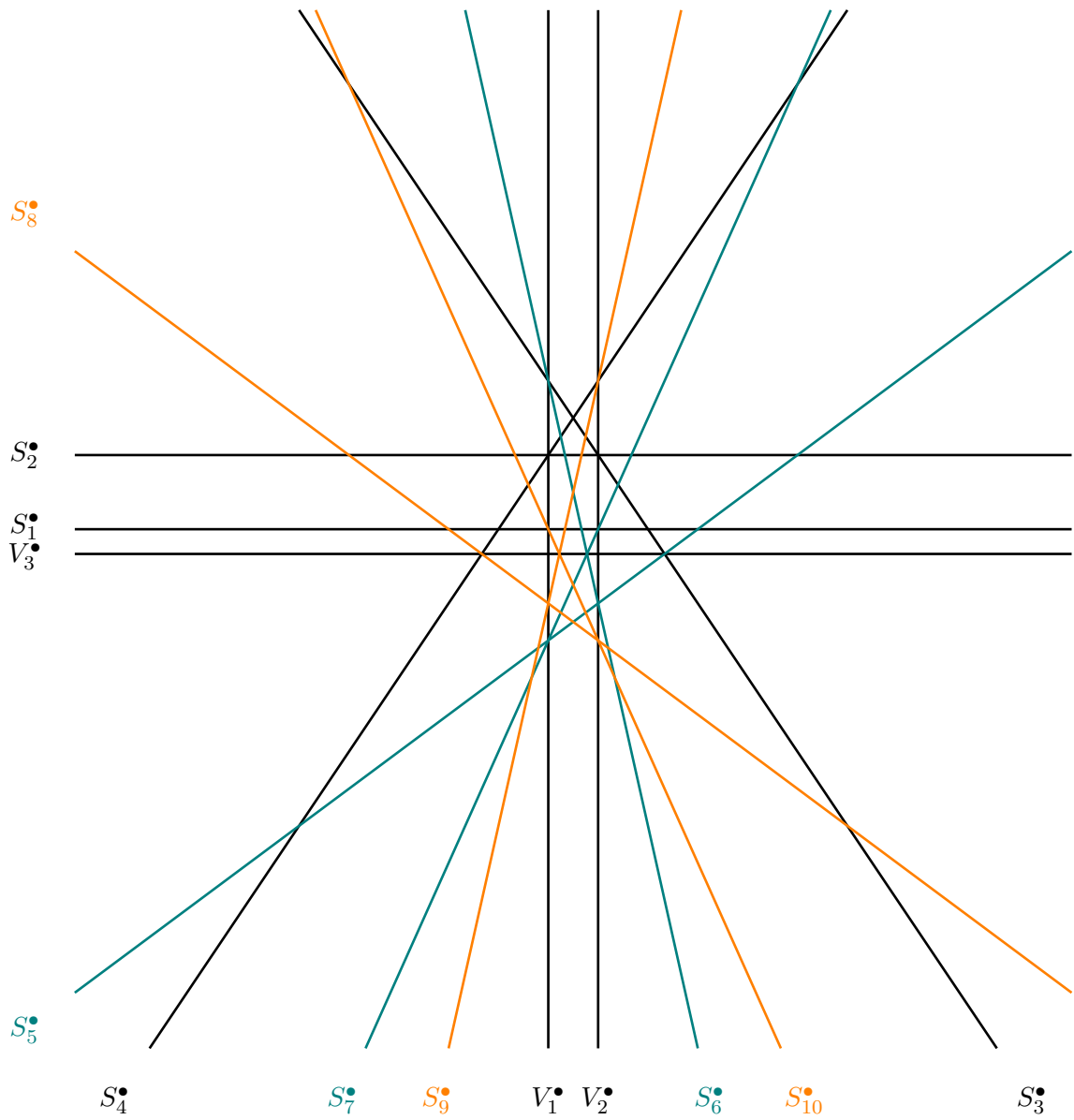
\begin{figure}[h]
	\centering	
	\begin{tikzpicture}[scale=0.35]
\def\colorV{black}
\def\colorS{black}
\def\colorSa{teal}
\def\colorSb{orange}
\coordinate (P1) at (-1, -20);
\coordinate (Q1) at (-1, 22);
\coordinate (P2) at (1, -20);
\coordinate (Q2) at (1, 22);
\coordinate (P3) at (-20, 0);
\coordinate (Q3) at (20, 0);
\coordinate (P4) at (-20, 1);
\coordinate (Q4) at (20, 1);
\coordinate (P5) at (-20, 4);
\coordinate (Q5) at (20, 4);
\coordinate (P6) at (17, -20);
\coordinate (Q6) at (-11, 22);
\coordinate (P7) at (-17, -20);
\coordinate (Q7) at (11, 22);
\coordinate (P8) at (-20, -71/4);
\coordinate (Q8) at (20, 49/4);
\coordinate (P9) at (5, -20);
\coordinate (Q9) at (-13/3, 22);
\coordinate (P10) at (-25/3, -20);
\coordinate (Q10) at (31/3, 22);
\coordinate (P11) at (-20, 49/4);
\coordinate (Q11) at (20, -71/4);
\coordinate (P12) at (-5, -20);
\coordinate (Q12) at (13/3, 22);
\coordinate (P13) at (25/3, -20);
\coordinate (Q13) at (-31/3, 22);
\draw[color=\colorV, line width=1] (P1)--(Q1) node[pos=-.05] {$V^\bullet_1$};
\draw[color=\colorV, line width=1] (P2)--(Q2) node[pos=-.05] {$V^\bullet_2$};
\draw[color=\colorV, line width=1] (P3)--(Q3) node[pos=-.05, yshift=-1] {$V^\bullet_3$};
\draw[color=\colorS, line width=1] (P4)--(Q4) node[pos=-.05, yshift=1] {$S^\bullet_1$};
\draw[color=\colorS, line width=1] (P5)--(Q5) node[pos=-.05] {$S^\bullet_2$};
\draw[color=\colorS, line width=1] (P6)--(Q6) node[pos=-.05] {$S^\bullet_3$};
\draw[color=\colorS, line width=1] (P7)--(Q7) node[pos=-.05] {$S^\bullet_4$};
\draw[color=\colorSa, line width=1] (P8)--(Q8) node[pos=-.05] {$S^\bullet_5$};
\draw[color=\colorSa, line width=1] (P9)--(Q9) node[pos=-.05] {$S^\bullet_6$};
\draw[color=\colorSa, line width=1] (P10)--(Q10) node[pos=-.05] {$S^\bullet_7$};
\draw[color=\colorSb, line width=1] (P11)--(Q11) node[pos=-.05] {$S^\bullet_8$};
\draw[color=\colorSb, line width=1] (P12)--(Q12) node[pos=-.05] {$S^\bullet_9$};
\draw[color=\colorSb, line width=1] (P13)--(Q13) node[pos=-.05] {$S^\bullet_{10}$};
\end{tikzpicture}
	\caption{The arrangement $\A^{-1,-1}$.\label{fig:ZP_mm}}
\end{figure}
\vfill

\newpage\mbox{}


\section{The $(3,2)$-configurations $\C_{\alpha,\beta}^2$ and their duals}\label{sec:appendix_ZP13+2pts5}

From the previous configurations, we obtain in Section~\ref{sec:deg} degenerated configurations $\C_{\alpha,\beta}^2$ with two alignments of five points, whose moduli space is studied in Section~\ref{sec:MS}. The dual arrangements $\A_2^{\alpha,\beta}$ of $\C_{\alpha,\beta}^2$ are pictured taking the line $S_1^\bullet$ (containing the two points of multiplicity five) as line at infinity.

\vfill
\begin{figure}[h]
	\centering	
	\begin{tikzpicture}[scale=2]

	\def\colorap{blue} 
	\def\coloram{teal}  
	\def\colorbp{red} 
	\def\colorbm{orange} 
	
	\tikzset{
	    myConic/.style = {color=cyan, dashed, smooth, samples=200, line width=1.5}
	}
	\draw[domain=sqrt(2)/2-0.025:2*sqrt(2)-0.09062][myConic] plot[variable=\x] ({\x}, { (1/4*sqrt(2)*(sqrt(2)*\x + sqrt(2) - sqrt(-6*\x*\x + 6*\x*(sqrt(2) + 2)
- 10*sqrt(2) + 3) + 1) });
	\draw[domain=sqrt(2)/2-0.025:2*sqrt(2)-0.09062][myConic] plot[variable=\x] ({\x}, { (1/4*sqrt(2)*(sqrt(2)*\x + sqrt(2) + sqrt(-6*\x*\x + 6*\x*(sqrt(2) + 2)
- 10*sqrt(2) + 3) + 1) });

		\draw[dashed, color=gray!50] (0,-2) -- (0,5) ;
		\draw[dashed, color=gray!50] (-2,0) -- (5,0) ;
		\draw[dashed, color=gray!50] (-2,4.5) -- (2.5,4.5) to[out=0,in=90] (4.5,2.5) -- (4.5,-2);

		\draw (-2,2) -- (3,2) to[out=0,in=100] (4.5,0) -- (4.6,-0.5);
		\draw (-2,1) -- (3,1) to[out=0,in=120] (4.5,0) -- (4.8,-0.5);
		\draw (1,-2) -- (1,3) to[out=90,in=-30] (0,4.5) -- (-0.5,4.75);
		\draw (2,-2) -- (2,3) to[out=90,in=-10] (0,4.5) -- (-0.5,4.55);
		\draw (-2,-2) -- (3,3);
		\draw (-2,-1) -- (3,1.5);
		\draw (1.5,3) -- (-1,-2);

		\draw[color=\colorap] (-2,1.41) -- (3,1.41) to[out=0,in=110] (4.5,0) -- (4.7,-0.5);
		\draw[color=\colorap] (0.71,-2) -- (0.71,3) to[out=90,in=-40] (0,4.5) -- (-0.5,4.85);
		\draw[color=\colorap] (2,2.82) -- (-1.41,-2);


		\draw[color=\colorbp] (-2,0.71) -- (3,0.71) to[out=0,in=130] (4.5,0) -- (4.9,-0.5);
		\draw[color=\colorbp] (1.41,-2) -- (1.41,3) to[out=90,in=-20] (0,4.5) -- (-0.5,4.65);
		\draw[color=\colorbp] (-2,-1.41) -- (3,2.12);



		 \node at (0,0) {$\bullet$};
		 \node[right] at (0.15,0) {$V_3$};
		 \node at (0,4.5) {$\bullet$};
		 \node[above] at (0,4.5) {$V_1$};
		 \node at (4.5,0) {$\bullet$};
		 \node[right] at (4.5,0) {$V_2$};
	 
		 \node at (1,1) {$\bullet$}; \node at (0.91,1.12) {$S_1$};
		 \node at (2,2) {$\bullet$}; \node at (1.85,2.15) {$S_2$};
		 \node at (1,2) {$\bullet$}; \node at (0.8561,2.1187) {$S_3$};
		 \node at (2,1) {$\bullet$}; \node at (2.1304,0.8689) {$S_4$};
	 
		 \node[text=\colorap] at (0.71,1.41) {$\bullet$}; \node[text=\colorap] at (0.5641,1.5365) {$S_5^{+}$};
		 \node[text=\colorap] at (1,1.41) {$\bullet$}; \node[text=\colorap] at (1.2265,1.5121) {$S_6^{+}$};
		 \node[text=\colorap] at (0.71,1) {$\bullet$};  \node[text=\colorap] at (0.5836,1.1138) {$S_7^{+}$};
	 
		 \node[text=\colorbp] at (1.41,0.71) {$\bullet$}; \node[text=\colorbp] at (1.5386,0.5807) {$S_8^{+}$};
		 \node[text=\colorbp] at (1.41,1) {$\bullet$}; \node[text=\colorbp] at (1.3087,1.1434) {$S_9^{+}$};
		 \node[text=\colorbp] at (1,0.71) {$\bullet$}; \node[text=\colorbp] at (1.1306,0.5781) {$S_{10}^{+}$};
\end{tikzpicture}
	\caption{The $(3,2)$-configuration $\C^2_{1,1}$ and the conic joining the six points $S_3,S_4,S_5^+,S_7^+,S_8^+,S_{10}^+$.\label{fig:dual_ZP_pp}}
\end{figure}
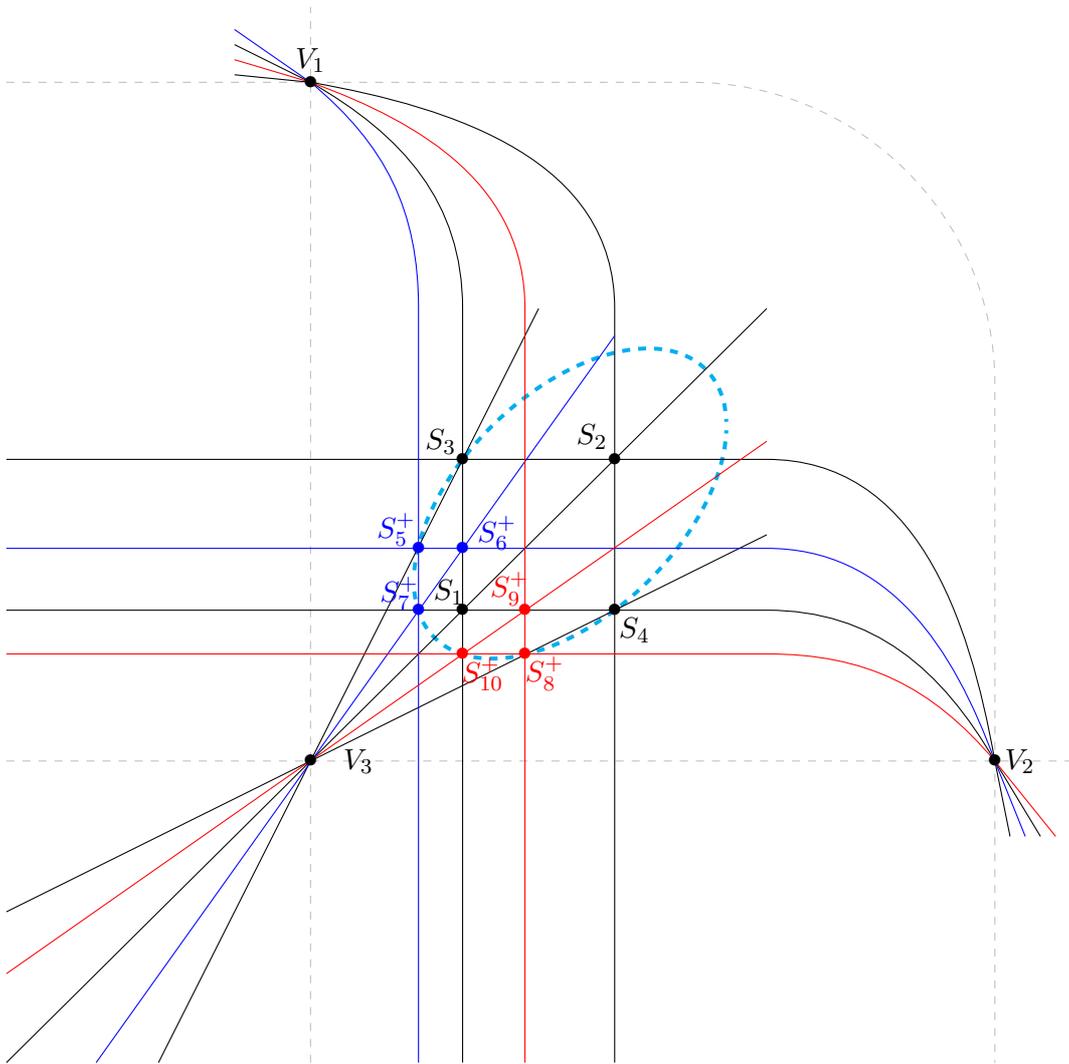
\vfill

\newpage\mbox{}

\vfill
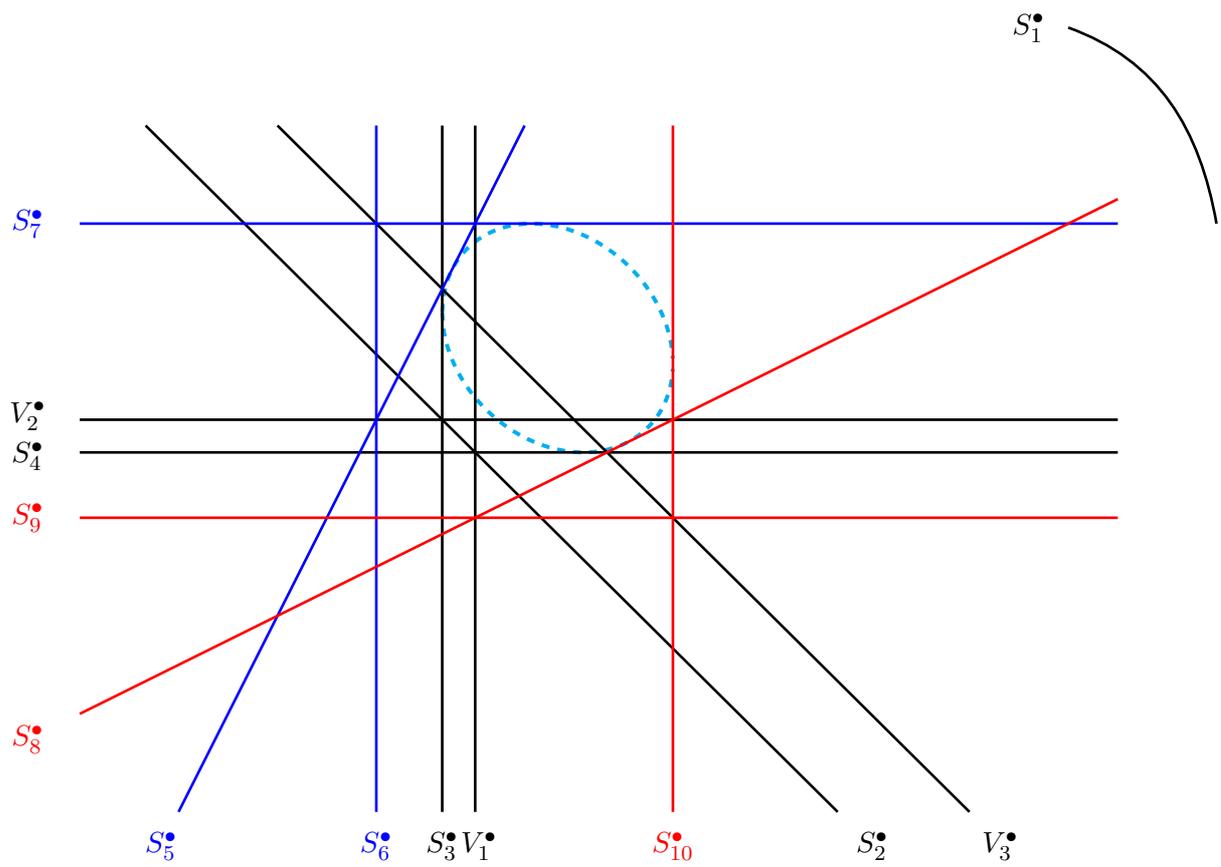
\begin{figure}[h]
	\centering	
	\begin{tikzpicture}[scale=1.3]
\def\colorV{black}
\def\colorS{black}
\def\colorSa{blue}
\def\colorSb{red}
\tikzset{
    myConic/.style = {color=cyan, dashed, smooth, samples=200, line width=1.5}
}

\draw[domain=-1/3:2][myConic] plot[variable=\x] ({\x}, { (50 - 11*\x + 2*sqrt(190)*sqrt(2 + 5*\x - 3*\x*\x))/49 });
\draw[domain=-1/3:2][myConic] plot[variable=\x] ({\x}, { (50 - 11*\x - 2*sqrt(190)*sqrt(2 + 5*\x - 3*\x*\x))/49 });

\coordinate (P1) at (0, -4);
\coordinate (Q1) at (0, 3);
\coordinate (P2) at (-4, 0);
\coordinate (Q2) at (13/2, 0);
\coordinate (P3) at (5, -4);
\coordinate (Q3) at (-2, 3);
\coordinate (P4) at (11/3, -4);
\coordinate (Q4) at (-10/3, 3);
\coordinate (P5) at (-1/3, -4);
\coordinate (Q5) at (-1/3, 3);
\coordinate (P6) at (-4, -1/3);
\coordinate (Q6) at (13/2, -1/3);
\coordinate (P7) at (-3, -4);
\coordinate (Q7) at (1/2, 3);
\coordinate (P8) at (-1, -4);
\coordinate (Q8) at (-1, 3);
\coordinate (P9) at (-4, 2);
\coordinate (Q9) at (13/2, 2);
\coordinate (P10) at (-4, -3);
\coordinate (Q10) at (13/2, 9/4);
\coordinate (P11) at (-4, -1);
\coordinate (Q11) at (13/2, -1);
\coordinate (P12) at (2, -4);
\coordinate (Q12) at (2, 3);
\draw[color=\colorV, line width=1] (P1)--(Q1) node[pos=-.05,xshift=.05cm] {$V^\bullet_1$};
\draw[color=\colorV, line width=1] (P2)--(Q2) node[pos=-.05,yshift=.05cm] {$V^\bullet_2$};
\draw[color=\colorV, line width=1] (P3)--(Q3) node[pos=-.05,xshift=-.05cm] {$V^\bullet_3$};
\draw[color=\colorS, line width=1] (P4)--(Q4) node[pos=-.05] {$S^\bullet_2$};
\draw[color=\colorS, line width=1] (P5)--(Q5) node[pos=-.05] {$S^\bullet_3$};
\draw[color=\colorS, line width=1] (P6)--(Q6) node[pos=-.05,yshift=-.05cm] {$S^\bullet_4$};
\draw[color=\colorSa, line width=1] (P7)--(Q7) node[pos=-.05] {$S^\bullet_5$};
\draw[color=\colorSa, line width=1] (P8)--(Q8) node[pos=-.05] {$S^\bullet_6$};
\draw[color=\colorSa, line width=1] (P9)--(Q9) node[pos=-.05] {$S^\bullet_7$};
\draw[color=\colorSb, line width=1] (P10)--(Q10) node[pos=-.05] {$S^\bullet_8$};
\draw[color=\colorSb, line width=1] (P11)--(Q11) node[pos=-.05] {$S^\bullet_9$};
\draw[color=\colorSb, line width=1] (P12)--(Q12) node[pos=-.05] {$S^\bullet_{10}$};

\draw[color=\colorS, line width=1] (6,4) to[out=-20,in=100] (7.5,2);
\node[text=\colorS] at (5.6,4) {$S_1^\bullet$};

\end{tikzpicture}
	\caption{The arrangement $\A_2^{1,1}$ and the conic tangent to the six lines $S_3^\bullet,S_4^\bullet,S_5^\bullet,S_7^\bullet,S_8^\bullet,S_{10}^\bullet$.\label{fig:ZP2pt5_pp}}
\end{figure}
\vfill

\newpage\mbox{}

\vfill
\begin{figure}[h]
	\begin{tikzpicture}[scale=2]

	\def\colorap{blue} 
	\def\coloram{teal}  
	\def\colorbp{red} 
	\def\colorbm{orange} 
	
		\draw[dashed, color=gray!50] (0,-2) -- (0,5) ;
		\draw[dashed, color=gray!50] (-2,0) -- (5,0) ;
		\draw[dashed, color=gray!50] (-2,4.5) -- (2.5,4.5) to[out=0,in=90] (4.5,2.5) -- (4.5,-2);

		\draw (-2,2) -- (3,2) to[out=0,in=100] (4.5,0) -- (4.6,-0.5);
		\draw (-2,1) -- (3,1) to[out=0,in=120] (4.5,0) -- (4.8,-0.5);
		\draw (1,-2) -- (1,3) to[out=90,in=-30] (0,4.5) -- (-0.5,4.75);
		\draw (2,-2) -- (2,3) to[out=90,in=-10] (0,4.5) -- (-0.5,4.55);
		\draw (-2,-2) -- (3,3);
		\draw (-2,-1) -- (3,1.5);
		\draw (1.5,3) -- (-1,-2);

		\draw[color=\colorap] (-2,1.41) -- (3,1.41) to[out=0,in=110] (4.5,0) -- (4.7,-0.5);
		\draw[color=\colorap] (0.71,-2) -- (0.71,3) to[out=90,in=-40] (0,4.5) -- (-0.5,4.85);
		\draw[color=\colorap] (2,2.82) -- (-1.41,-2);



		\draw[color=\colorbm] (-2,-0.71) -- (3,-0.71) to[out=0,in=-130] (4.5,0) -- (4.85,0.5);
		\draw[color=\colorbm] (-1.41,-2) -- (-1.41,3) to[out=90,in=-170] (0,4.5) -- (0.5,4.55);
		\draw[color=\colorbm] (-2,1.41) -- (3,-2.12);


		 \node at (0,0) {$\bullet$};
		 \node[right] at (0.15,0) {$V_3$};
		 \node at (0,4.5) {$\bullet$};
		 \node[above] at (0,4.5) {$V_1$};
		 \node at (4.5,0) {$\bullet$};
		 \node[right] at (4.5,0) {$V_2$};
	 
		 \node at (1,1) {$\bullet$}; \node at (0.91,1.12) {$S_1$};
		 \node at (2,2) {$\bullet$}; \node at (1.85,2.15) {$S_2$};
		 \node at (1,2) {$\bullet$}; \node at (0.8561,2.1187) {$S_3$};
		 \node at (2,1) {$\bullet$}; \node at (2.1304,0.8689) {$S_4$};
	 
		 \node[text=\colorap] at (0.71,1.41) {$\bullet$}; \node[text=\colorap] at (0.5641,1.5365) {$S_5^{+}$};
		 \node[text=\colorap] at (1,1.41) {$\bullet$}; \node[text=\colorap] at (1.2265,1.5121) {$S_6^{+}$};
		 \node[text=\colorap] at (0.71,1) {$\bullet$};  \node[text=\colorap] at (0.5836,1.1138) {$S_7^{+}$};
	 
		 \node[text=\colorbm] at (-1.41,-0.71) {$\bullet$}; \node[text=\colorbm] at (-1.5569,-0.5443) {$S_8^{-}$};
		 \node[text=\colorbm] at (-1.41,1) {$\bullet$}; \node[text=\colorbm] at (-1.5695,0.837) {$S_9^{-}$};
		 \node[text=\colorbm] at (1,-0.71) {$\bullet$}; \node[text=\colorbm] at (1.1685,-0.5504) {$S_{10}^{-}$};
\end{tikzpicture}
	\caption{The $(3,2)$-configuration $\C^2_{1,-1}$.\label{fig:dual_ZP_pm}}
\end{figure}
\vfill

\newpage\mbox{}

\vfill
\begin{figure}[h]
	\centering	
	\begin{tikzpicture}[scale=1.3]
\def\colorV{black}
\def\colorS{black}
\def\colorSa{blue}
\def\colorSb{orange}
\coordinate (P1) at (0, -4);
\coordinate (Q1) at (0, 3);
\coordinate (P2) at (-4, 0);
\coordinate (Q2) at (13/2, 0);
\coordinate (P3) at (5, -4);
\coordinate (Q3) at (-2, 3);
\coordinate (P4) at (11/3, -4);
\coordinate (Q4) at (-10/3, 3);
\coordinate (P5) at (-1/3, -4);
\coordinate (Q5) at (-1/3, 3);
\coordinate (P6) at (-4, -1/3);
\coordinate (Q6) at (13/2, -1/3);
\coordinate (P7) at (-3, -4);
\coordinate (Q7) at (1/2, 3);
\coordinate (P8) at (-1, -4);
\coordinate (Q8) at (-1, 3);
\coordinate (P9) at (-4, 2);
\coordinate (Q9) at (13/2, 2);
\coordinate (P10) at (-4, 7/3);
\coordinate (Q10) at (13/2, -35/12);
\coordinate (P11) at (-4, 1/3);
\coordinate (Q11) at (13/2, 1/3);
\coordinate (P12) at (2/3, -4);
\coordinate (Q12) at (2/3, 3);
\draw[color=\colorV, line width=1] (P1)--(Q1) node[pos=-.05,xshift=.05cm]
{$V^\bullet_1$};
\draw[color=\colorV, line width=1] (P2)--(Q2) node[pos=-.05]
{$V^\bullet_2$};
\draw[color=\colorV, line width=1] (P3)--(Q3) node[pos=-.05]
{$V^\bullet_3$};
\draw[color=\colorS, line width=1] (P4)--(Q4) node[pos=-.05]
{$S^\bullet_2$};
\draw[color=\colorS, line width=1] (P5)--(Q5) node[pos=-.05,xshift=-.05cm]
{$S^\bullet_3$};
\draw[color=\colorS, line width=1] (P6)--(Q6) node[pos=-.05, yshift=-.1cm]
{$S^\bullet_4$};
\draw[color=\colorSa, line width=1] (P7)--(Q7) node[pos=-.05]
{$S^\bullet_5$};
\draw[color=\colorSa, line width=1] (P8)--(Q8) node[pos=-.05]
{$S^\bullet_6$};
\draw[color=\colorSa, line width=1] (P9)--(Q9) node[pos=-.05]
{$S^\bullet_7$};
\draw[color=\colorSb, line width=1] (P10)--(Q10) node[pos=-.05]
{$S^\bullet_8$};
\draw[color=\colorSb, line width=1] (P11)--(Q11) node[pos=-.05,yshift=.1cm]
{$S^\bullet_9$};
\draw[color=\colorSb, line width=1] (P12)--(Q12) node[pos=-.05]
{$S^\bullet_{10}$};

\draw[color=\colorS, line width=1] (6,4) to[out=-20,in=100] (7.5,2);
\node[text=\colorS] at (5.6,4) {$S_1^\bullet$};
\end{tikzpicture}
	\caption{The arrangement $\A_2^{1,-1}$.\label{fig:ZP2pt5_pm}}
\end{figure}
\vfill

\newpage\mbox{}

\vfill
\begin{figure}[h]
	\begin{tikzpicture}[scale=2]

	\def\colorap{blue} 
	\def\coloram{teal}  
	\def\colorbp{red} 
	\def\colorbm{orange} 
	
		\draw[dashed, color=gray!50] (0,-2) -- (0,5) ;
		\draw[dashed, color=gray!50] (-2,0) -- (5,0) ;
		\draw[dashed, color=gray!50] (-2,4.5) -- (2.5,4.5) to[out=0,in=90] (4.5,2.5) -- (4.5,-2);

		\draw (-2,2) -- (3,2) to[out=0,in=100] (4.5,0) -- (4.6,-0.5);
		\draw (-2,1) -- (3,1) to[out=0,in=120] (4.5,0) -- (4.8,-0.5);
		\draw (1,-2) -- (1,3) to[out=90,in=-30] (0,4.5) -- (-0.5,4.75);
		\draw (2,-2) -- (2,3) to[out=90,in=-10] (0,4.5) -- (-0.5,4.55);
		\draw (-2,-2) -- (3,3);
		\draw (-2,-1) -- (3,1.5);
		\draw (1.5,3) -- (-1,-2);


		\draw[color=\coloram] (-2,-1.41) -- (3,-1.41) to[out=0,in=-100] (4.5,0) -- (4.55,0.5);
		\draw[color=\coloram] (-0.71,-2) -- (-0.71,3) to[out=90,in=-140] (0,4.5) -- (0.5,4.8);
		\draw[color=\coloram] (1.41,-2) -- (-2.12,3);

		\draw[color=\colorbp] (-2,0.71) -- (3,0.71) to[out=0,in=130] (4.5,0) -- (4.9,-0.5);
		\draw[color=\colorbp] (1.41,-2) -- (1.41,3) to[out=90,in=-20] (0,4.5) -- (-0.5,4.65);
		\draw[color=\colorbp] (-2,-1.41) -- (3,2.12);



		 \node at (0,0) {$\bullet$};
		 \node[right] at (0.15,0) {$V_3$};
		 \node at (0,4.5) {$\bullet$};
		 \node[above] at (0,4.5) {$V_1$};
		 \node at (4.5,0) {$\bullet$};
		 \node[right] at (4.5,0) {$V_2$};
	 
		 \node at (1,1) {$\bullet$}; \node at (0.91,1.12) {$S_1$};
		 \node at (2,2) {$\bullet$}; \node at (1.85,2.15) {$S_2$};
		 \node at (1,2) {$\bullet$}; \node at (0.8561,2.1187) {$S_3$};
		 \node at (2,1) {$\bullet$}; \node at (2.1304,0.8689) {$S_4$};
	 
		 \node[text=\coloram] at (-0.71,-1.41) {$\bullet$}; \node[text=\coloram] at (-0.8909,-1.5313) {$S_5^{-}$};
		 \node[text=\coloram] at (1,-1.41) {$\bullet$}; \node[text=\coloram] at (0.8129,-1.5499) {$S_6^{-}$};
		 \node[text=\coloram] at (-0.71,1) {$\bullet$}; \node[text=\coloram] at (-0.5429,1.1386) {$S_7^{-}$};
	 
		 \node[text=\colorbp] at (1.41,0.71) {$\bullet$}; \node[text=\colorbp] at (1.5386,0.5807) {$S_8^{+}$};
		 \node[text=\colorbp] at (1.41,1) {$\bullet$}; \node[text=\colorbp] at (1.3087,1.1434) {$S_9^{+}$};
		 \node[text=\colorbp] at (1,0.71) {$\bullet$}; \node[text=\colorbp] at (1.1306,0.5781) {$S_{10}^{+}$};
\end{tikzpicture}
	\caption{The $(3,2)$-configuration $\C^2_{-1,1}$.\label{fig:dual_ZP_mp}}
\end{figure}
\vfill

\newpage\mbox{}

\vfill
\begin{figure}[h]
	\centering	
	\begin{tikzpicture}[scale=1.3]
\def\colorV{black}
\def\colorS{black}
\def\colorSa{teal}
\def\colorSb{red}
\coordinate (P1) at (0, -4);
\coordinate (Q1) at (0, 3);
\coordinate (P2) at (-4, 0);
\coordinate (Q2) at (13/2, 0);
\coordinate (P3) at (5, -4);
\coordinate (Q3) at (-2, 3);
\coordinate (P4) at (11/3, -4);
\coordinate (Q4) at (-10/3, 3);
\coordinate (P5) at (-1/3, -4);
\coordinate (Q5) at (-1/3, 3);
\coordinate (P6) at (-4, -1/3);
\coordinate (Q6) at (13/2, -1/3);
\coordinate (P7) at (7/3, -4);
\coordinate (Q7) at (-7/6, 3);
\coordinate (P8) at (1/3, -4);
\coordinate (Q8) at (1/3, 3);
\coordinate (P9) at (-4, 2/3);
\coordinate (Q9) at (13/2, 2/3);
\coordinate (P10) at (-4, -3);
\coordinate (Q10) at (13/2, 9/4);
\coordinate (P11) at (-4, -1);
\coordinate (Q11) at (13/2, -1);
\coordinate (P12) at (2, -4);
\coordinate (Q12) at (2, 3);
\draw[color=\colorV, line width=1] (P1)--(Q1) node[pos=-.05]
{$V^\bullet_1$};
\draw[color=\colorV, line width=1] (P2)--(Q2) node[pos=-.05,yshift=.05cm]
{$V^\bullet_2$};
\draw[color=\colorV, line width=1] (P3)--(Q3) node[pos=-.05]
{$V^\bullet_3$};
\draw[color=\colorS, line width=1] (P4)--(Q4) node[pos=-.05]
{$S^\bullet_2$};
\draw[color=\colorS, line width=1] (P5)--(Q5) node[pos=-.05,xshift=-.05cm]
{$S^\bullet_3$};
\draw[color=\colorS, line width=1] (P6)--(Q6) node[pos=-.05,yshift=-.05cm]
{$S^\bullet_4$};
\draw[color=\colorSa, line width=1] (P7)--(Q7) node[pos=-.05]
{$S^\bullet_5$};
\draw[color=\colorSa, line width=1] (P8)--(Q8) node[pos=-.05,xshift=.05cm]
{$S^\bullet_6$};
\draw[color=\colorSa, line width=1] (P9)--(Q9) node[pos=-.05]
{$S^\bullet_7$};
\draw[color=\colorSb, line width=1] (P10)--(Q10) node[pos=-.05]
{$S^\bullet_8$};
\draw[color=\colorSb, line width=1] (P11)--(Q11) node[pos=-.05]
{$S^\bullet_9$};
\draw[color=\colorSb, line width=1] (P12)--(Q12) node[pos=-.05]
{$S^\bullet_{10}$};

\draw[color=\colorS, line width=1] (6,4) to[out=-20,in=100] (7.5,2);
\node[text=\colorS] at (5.6,4) {$S_1^\bullet$};
\end{tikzpicture}
	\caption{The arrangement $\A_2^{-1,1}$.\label{fig:ZP2pt5_mp}}
\end{figure}
\vfill

\newpage\mbox{}

\vfill
\begin{figure}[h]
	\begin{tikzpicture}[scale=2]

	\def\colorap{blue} 
	\def\coloram{teal}  
	\def\colorbp{red} 
	\def\colorbm{orange} 
	
	\tikzset{
	    myConic/.style = {color=cyan, dashed, smooth, samples=200, line width=1.5}
	}
	\draw[domain=-sqrt(2)-0.008:sqrt(2)+0.59481][myConic] plot[variable=\x] ({\x}, {1/4*sqrt(2)*(sqrt(2)*\x + sqrt(2) - sqrt(-6*\x*\x - 6*\x*(sqrt(2) - 2) + 10*sqrt(2) + 3) - 1) });
 	\draw[domain=-sqrt(2)-0.008:sqrt(2)+0.59481][myConic] plot[variable=\x] ({\x}, {1/4*sqrt(2)*(sqrt(2)*\x + sqrt(2) + sqrt(-6*\x*\x - 6*\x*(sqrt(2) - 2) + 10*sqrt(2) + 3) - 1) });
	
		\draw[dashed, color=gray!50] (0,-2) -- (0,5) ;
		\draw[dashed, color=gray!50] (-2,0) -- (5,0) ;
		\draw[dashed, color=gray!50] (-2,4.5) -- (2.5,4.5) to[out=0,in=90] (4.5,2.5) -- (4.5,-2);

		\draw (-2,2) -- (3,2) to[out=0,in=100] (4.5,0) -- (4.6,-0.5);
		\draw (-2,1) -- (3,1) to[out=0,in=120] (4.5,0) -- (4.8,-0.5);
		\draw (1,-2) -- (1,3) to[out=90,in=-30] (0,4.5) -- (-0.5,4.75);
		\draw (2,-2) -- (2,3) to[out=90,in=-10] (0,4.5) -- (-0.5,4.55);
		\draw (-2,-2) -- (3,3);
		\draw (-2,-1) -- (3,1.5);
		\draw (1.5,3) -- (-1,-2);


		\draw[color=\coloram] (-2,-1.41) -- (3,-1.41) to[out=0,in=-100] (4.5,0) -- (4.55,0.5);
		\draw[color=\coloram] (-0.71,-2) -- (-0.71,3) to[out=90,in=-140] (0,4.5) -- (0.5,4.8);
		\draw[color=\coloram] (1.41,-2) -- (-2.12,3);


		\draw[color=\colorbm] (-2,-0.71) -- (3,-0.71) to[out=0,in=-130] (4.5,0) -- (4.85,0.5);
		\draw[color=\colorbm] (-1.41,-2) -- (-1.41,3) to[out=90,in=-170] (0,4.5) -- (0.5,4.55);
		\draw[color=\colorbm] (-2,1.41) -- (3,-2.12);


		 \node at (0,0) {$\bullet$};
		 \node[right] at (0.15,0) {$V_3$};
		 \node at (0,4.5) {$\bullet$};
		 \node[above] at (0,4.5) {$V_1$};
		 \node at (4.5,0) {$\bullet$};
		 \node[right] at (4.5,0) {$V_2$};
	 
		 \node at (1,1) {$\bullet$}; \node at (0.91,1.12) {$S_1$};
		 \node at (2,2) {$\bullet$}; \node at (1.85,2.15) {$S_2$};
		 \node at (1,2) {$\bullet$}; \node at (0.8561,2.1187) {$S_3$};
		 \node at (2,1) {$\bullet$}; \node at (2.1304,0.8689) {$S_4$};
	 
		 \node[text=\coloram] at (-0.71,-1.41) {$\bullet$}; \node[text=\coloram] at (-0.8909,-1.5313) {$S_5^{-}$};
		 \node[text=\coloram] at (1,-1.41) {$\bullet$}; \node[text=\coloram] at (0.8129,-1.5499) {$S_6^{-}$};
		 \node[text=\coloram] at (-0.71,1) {$\bullet$}; \node[text=\coloram] at (-1.0173,1.148) {$S_7^{-}$};
	 
		 \node[text=\colorbm] at (-1.41,-0.71) {$\bullet$}; \node[text=\colorbm] at (-1.5569,-0.5443) {$S_8^{-}$};
		 \node[text=\colorbm] at (-1.41,1) {$\bullet$}; \node[text=\colorbm] at (-1.5695,0.837) {$S_9^{-}$};
		 \node[text=\colorbm] at (1,-0.71) {$\bullet$}; \node[text=\colorbm] at (1.1499,-0.9749) {$S_{10}^{-}$};
\end{tikzpicture}
	\caption{The $(3,2)$-configuration $\C^2_{-1,-1}$ and the conic joining the six points $S_3,S_4,S_5^-,S_7^-,S_8^-,S_{10}^-$.\label{fig:dual_ZP_mm}}
\end{figure}
\vfill

\newpage\mbox{}

\vfill
\begin{figure}[h]
	\centering	
	\begin{tikzpicture}[scale=1.3]
\def\colorV{black}
\def\colorS{black}
\def\colorSa{teal}
\def\colorSb{orange}
\tikzset{
    myConic/.style = {color=cyan, dashed, smooth, samples=200, line width=1.5}
}

\draw[domain=-4:-1/3][myConic] plot[variable=\x] ({\x}, {  (10 - 33*\x - 2*sqrt(10)*sqrt(-2 - 3*\x + 9*\x*\x))/27});
\draw[domain=-1.375:-1/3][myConic] plot[variable=\x] ({\x}, {  (10 - 33*\x + 2*sqrt(10)*sqrt(-2 - 3*\x + 9*\x*\x))/27});

\draw[domain=0.667:2.35][myConic] plot[variable=\x] ({\x}, {  (10 - 33*\x - 2*sqrt(10)*sqrt(-2 - 3*\x + 9*\x*\x))/27}); 
\draw[domain=0.667:6.5][myConic] plot[variable=\x] ({\x}, {  (10 - 33*\x + 2*sqrt(10)*sqrt(-2 - 3*\x + 9*\x*\x))/27});

\coordinate (P1) at (0, -4);
\coordinate (Q1) at (0, 3);
\coordinate (P2) at (-4, 0);
\coordinate (Q2) at (13/2, 0);
\coordinate (P3) at (5, -4);
\coordinate (Q3) at (-2, 3);
\coordinate (P4) at (11/3, -4);
\coordinate (Q4) at (-10/3, 3);
\coordinate (P5) at (-1/3, -4);
\coordinate (Q5) at (-1/3, 3);
\coordinate (P6) at (-4, -1/3);
\coordinate (Q6) at (13/2, -1/3);
\coordinate (P7) at (7/3, -4);
\coordinate (Q7) at (-7/6, 3);
\coordinate (P8) at (1/3, -4);
\coordinate (Q8) at (1/3, 3);
\coordinate (P9) at (-4, 2/3);
\coordinate (Q9) at (13/2, 2/3);
\coordinate (P10) at (-4, 7/3);
\coordinate (Q10) at (13/2, -35/12);
\coordinate (P11) at (-4, 1/3);
\coordinate (Q11) at (13/2, 1/3);
\coordinate (P12) at (2/3, -4);
\coordinate (Q12) at (2/3, 3);
\draw[color=\colorV, line width=1] (P1)--(Q1) node[pos=-.05,xshift=-.05cm]
{$V^\bullet_1$};
\draw[color=\colorV, line width=1] (P2)--(Q2) node[pos=-.05,yshift=-.05cm]
{$V^\bullet_2$};
\draw[color=\colorV, line width=1] (P3)--(Q3) node[pos=-.05]
{$V^\bullet_3$};
\draw[color=\colorS, line width=1] (P4)--(Q4) node[pos=-.05]
{$S^\bullet_2$};
\draw[color=\colorS, line width=1] (P5)--(Q5) node[pos=-.05,xshift=-.15cm]
{$S^\bullet_3$};
\draw[color=\colorS, line width=1] (P6)--(Q6) node[pos=-.05,yshift=-.15cm]
{$S^\bullet_4$};
\draw[color=\colorSa, line width=1] (P7)--(Q7) node[pos=-.05]
{$S^\bullet_5$};
\draw[color=\colorSa, line width=1] (P8)--(Q8) node[pos=-.05,xshift=.05cm]
{$S^\bullet_6$};
\draw[color=\colorSa, line width=1] (P9)--(Q9) node[pos=-.05,yshift=.15cm]
{$S^\bullet_7$};
\draw[color=\colorSb, line width=1] (P10)--(Q10) node[pos=-.05]
{$S^\bullet_8$};
\draw[color=\colorSb, line width=1] (P11)--(Q11) node[pos=-.05,yshift=.05cm]
{$S^\bullet_9$};
\draw[color=\colorSb, line width=1] (P12)--(Q12) node[pos=-.05,xshift=.15cm]
{$S^\bullet_{10}$};

\draw[color=\colorS, line width=1] (6,4) to[out=-20,in=100] (7.5,2);
\node[text=\colorS] at (5.6,4) {$S_1^\bullet$};

\end{tikzpicture}
	\caption{The arrangement $\A_2^{-1,-1}$ and the conic tangent to the six lines $S_3^\bullet,S_4^\bullet,S_5^\bullet,S_7^\bullet,S_8^\bullet,S_{10}^\bullet$.\label{fig:ZP2pt5_mm}}
\end{figure}
\vfill

\newpage\mbox{}

\bibliographystyle{alpha}
\bibliography{biblio}

\end{document}